\documentclass[twoside,11pt]{article}

\usepackage{graphicx, subfigure}
\usepackage[top=1in,bottom=1in,left=1in,right=1in]{geometry}
\usepackage[numbers,sort&compress]{natbib} 
\usepackage{amsfonts, amsmath, amssymb, amsthm, bbm}
\usepackage[normalem]{ulem}
\usepackage{eufrak}

\usepackage{booktabs}

\usepackage{xcolor}
\definecolor{darkred}{RGB}{100,0,0}
\definecolor{darkgreen}{RGB}{0,100,0}
\definecolor{darkblue}{RGB}{0,0,150}

\usepackage{hyperref}
\hypersetup{colorlinks=true, linkcolor=darkred, citecolor=darkgreen, urlcolor=darkblue}
\usepackage{url}

\usepackage{tikz}
\usetikzlibrary{plothandlers,plotmarks,arrows,automata}

\newtheorem{thm}{Theorem}
\newtheorem{prp}{Proposition}
\newtheorem{lem}{Lemma}

\def\beq{\begin{equation}}
\def\eeq{\end{equation}}
\def\beqn{\begin{eqnarray*}}
\def\eeqn{\end{eqnarray*}}
\def\bitem{\begin{itemize}}
\def\eitem{\end{itemize}}
\def\benum{\begin{enumerate}}
\def\eenum{\end{enumerate}}
\def\bmult{\begin{multline*}}
\def\emult{\end{multline*}}
\def\bcenter{\begin{center}}
\def\ecenter{\end{center}}

\DeclareMathOperator*{\argmin}{arg\, min}

\DeclareMathOperator{\rank}{rank}

\def\cA{\mathcal{A}}
\def\cB{\mathcal{B}}
\def\cC{\mathcal{C}}

\def\cP{\mathcal{P}}

\def\cR{\mathcal{R}}

\def\cT{\mathcal{T}}

\def\cW{\mathcal{W}}
\def\cX{\mathcal{X}}
\def\cY{\mathcal{Y}}
\def\cZ{\mathcal{Z}}

\def\bA{\boldsymbol{A}}
\def\bB{\boldsymbol{B}}

\def\bD{\boldsymbol{D}}
\def\bE{\boldsymbol{E}}

\def\bJ{\boldsymbol{J}}

\def\bQ{\boldsymbol{Q}}

\def\bT{\boldsymbol{T}}
\def\bU{\boldsymbol{U}}

\def\bW{\boldsymbol{W}}

\def\bb{\mathbf{b}}

\newcommand{\bxi}{{\boldsymbol\xi}}

\def\bbE{\mathbb{E}}

\def\bbP{\mathbb{P}}

\def\bbR{\mathbb{R}}

\def\bb\bU{\mathbb{\bU}}

\newcommand{\E}{\operatorname{\mathbb{E}}}
\renewcommand{\P}{\operatorname{\mathbb{P}}}
\newcommand{\Var}{\operatorname{Var}}

\def\\bUnif{\text{\bUnif}}

\DeclareMathOperator{\sign}{sign}

\usepackage{dsfont}
\newcommand{\1}{\mathds{1}}

\def\cut{\square}

\def\bTheta{\boldsymbol{\Theta}}

\title{Optimal graphon estimation in cut distance}
\author{Olga Klopp\footnote{ESSEC Business School and CREST, FRANCE, \url{olga.klopp@math.cnrs.fr}}~\ and Nicolas Verzelen\footnote{INRA, UMR 729 MISTEA, F-34060 Montpellier, FRANCE, \url{nicolas.verzelen@inra.fr}}}
\begin{document}

\maketitle

\begin{abstract}
Consider the twin problems of estimating the connection probability matrix of an  inhomogeneous random graph and the graphon of a $W$-random graph. We establish the minimax estimation rates with respect to the cut metric for classes of block constant matrices and  step function graphons. Surprisingly, our results imply that, from the minimax point of view, the raw data, that is, the adjacency matrix of the observed graph, is already optimal and more involved procedures cannot improve the convergence rates for this metric. This phenomenon contrasts with optimal rates of convergence with respect to other classical distances for graphons such as the $l_1$ or $l_2$ metrics. 
\end{abstract}


\section{Introduction}

In the last decade, network analysis has become an important research field driven by applications in  social sciences, computer sciences, statistical physics, genomics, ecology, etc.
A flourishing line of literature amounts to fit observed networks to parametric or non-parametric models of random graphs. Among the parametric models, one of the most popular is the stochastic block model~\cite{holland_laskey}. In the stochastic block model with $n$ vectices and $k$ blocks, the class $Z_i$ of each vertex $i\in [n]$ is drawn independently in $[k]$ according to some probability distribution $\pi$. Given $Z$, the edges of the graph are then sampled independently, the probability that there is an edge between $i$ and $j$ being equal to $\bQ_{Z_iZ_j}$ where 
$\bQ=(\bQ_{ij})\in [0,1]^{k\times k}$ is a given symmetric matrix.  Although this model is suitable for analyzing small networks, it does not allow to analyze the finer structures of extremely large networks.
To go beyond the possible limitation of parametric models, non-parametric models of random graphs have been introduced~\cite{hoff2002latent,diaconis2007graph}.

One possible non-parametric generalization of the stochastic block models is given by the $W$-random graph model~\cite{diaconis2007graph} based on the notion of graphon. Graphons are symmetric  measurable functions $W :[0,1]^{2}\rightarrow [0,1]$.
In the sequel, the space of graphons is denoted by $\cW^+$.  Given a graphon $W_0\in \cW^+$,
 a graph on $n$ vertices is sampled according to the $W$-random graph model in the following way. Let $\bTheta_0=(\bTheta_{ij})$ be a $n\times n$ random symmetric matrix defined by 
\begin{equation}\label{sparse_graphon_mod}
      \bTheta_{ij}=\rho_n W_0(\xi_i,\xi_j)\ , \forall i\neq j \, \text{ and } \bTheta_{ii}=0
\end{equation} 
where $1\geq \rho_n>0$ is the scale parameter that can be interpreted as the expected proportion of non-zero edges and $\xi_1,\dots,\xi_n$ are unobserved (latent) i.i.d. random variables uniformly distributed on $[0,1]$. Then, given $\bTheta_0$, the graph is sampled according to the inhomogeneous random graph model~\cite{bollobas_janson}. More precisely, vertices $i$ and $j$ are connected by an edge with probability $\bTheta_{ij}$ and these events are independent for all pairs $(i,j)$ with $i<j$. When $\bTheta_0$ is considered as a deterministic matrix, we call it inhomogeneous random graph model with respect ot $\bTheta_0$. 
If $W_0$ is a step-function with $k$ steps, the graph is distributed as  a stochastic block model with $k$ groups. The case of a dense graph corresponds to $\rho_n=1$, whereas  the choice $\rho_n\rightarrow 0$ when $n\rightarrow\infty$ produces sparser graphs. This model was recently studied by a number of authors, see e.g., \cite{bickel2009nonparametric, bickel2011method,yang_han_airoldi, latouche2013bayesian, choi_wolfe,Chatterjee_mc}.

In the present paper we consider the problems  of estimating the matrix of connection probabilities $\bTheta_0$  and the graphon $f_0=\rho_n W_0$ from a single observation of a graph. Suppose that we observe the $n\times n$ adjacency matrix $\bA=(\bA_{ij})$ of a graph that has either been sampled according to the inhomogeneous random graph model with a fixed matrix $\bTheta_0$ or to the $W$-random graph model with graphon $W_0$. Then, given a single observation  $\bA$, we want to estimate $\bTheta_0$ or $f_0$.

Graphon estimation is more challenging than probability matrix estimation, in particular, because of identifiability issues: multiple graphons
can lead to the same distribution on the space of graphs of size $n$. This is not unexpected as the distribution of the network is invariant with respect to any change of labeling of its nodes. More precisely, two graphons $U$ and $W$ in $\cW^+$ define the same probability distribution if and only if there exist measure preserving maps $\phi$, $\psi$: $[0,1]\to [0,1]$ such that $U\left (\phi(x),\phi(y)\right )= W\left (\psi(x),\psi(y)\right )$ almost everywhere. This equivalence relation is called a weak isomorphism~\cite{LovaszBook}. The corresponding quotient space is denoted by $\widetilde{\cW}^+$. As a consequence, one can only estimate the equivalence class of $\rho_n W_0$ in $\widetilde{\cW}^+$  and we refer henceforth to  graphon estimation as the problem of estimating this equivalence class from the adjacency matrix $\bA$ sampled from the $W$-random graph model~\eqref{sparse_graphon_mod}. When there is no amibiguity, we shall identify  a graphon $W\in \cW^{+}$ and its corresponding equivalence class.

The problem of estimating $\bTheta_0$ was previously considered in a number of papers. For matrix estimation problem, the quality of an estimator $\widehat{\bTheta}$  is usually assessed through the Frobenius loss $\|\widehat{\bTheta}-\bTheta_0\|_2$. For instance, \cite{Chatterjee_mc}  obtain sub-optimal convergence rates for this problem using a singular thresholding algorithm. Relying on a least-square estimator  \cite{gao2014rate} have established the minimax estimation rates for $\bTheta_0$  on classes of block constant matrices and smooth graphon classes. Their analysis is restricted to the dense case with constant $\Vert \bTheta_0\Vert_{\infty}$. More recently, \cite{klopp_graphon} extended their results to  sparse case when $\Vert \bTheta_0\Vert_{\infty}$ depends on $n$ and goes to zero when $n\rightarrow \infty$.

As for graphon estimation, most of results on estimation error are expressed in terms of  $l_2$ loss $\|\widehat{W}-W_0\|_2$ (see below for a formal definition of this metric). For classes of smooth graphons, estimators based on maximum likelihood, restricted least-squares estimators, or neighborhood smoothing have been studied  in \cite{WolfeOlhede,chan2014consistent,airoldi2013stochastic,cai2014iterative,zhang2015estimating,klopp_graphon}. For classes of step-function graphons, restricted least-squares estimators have been considered in  \cite{borgs2015,klopp_graphon} and the minimax optimal rates of convergence have been derived in~\cite{klopp_graphon}.

Although one can take advantage of the Euclidean structure of the Frobenius matrix norm and the $l_2$ metric on $\cW^+$, both these metrics do not readily reflect the closeness in terms of the topology of the random graphs. As the structure of the graphon space is infinite-dimensional, not all norms are equivalent and one may wonder whether one should not study the graphon estimation problem with respect to a more suitable distance. We argue below that the cut distance which plays a central role in the random graph theory is a good candidate for this.

\subsection{Cut metric}\label{sec:cut}

One of the fundamental questions in graph theory is the following one: what does it mean for two large graphs to be similar or close? There are different ways of defining the distance of two graphs. For example, the edit distance is defined as normalized Hamming distance of the edge sets. Up to a normalization, it corresponds to $l_1$ distance between the adjacency matrices. One of the troubles with this notion of distance is that it does not reflect well structural similarities between two graphs. For instance, the edit distance between two independent graphs drawn from the  Erd\"os-R\'enyi model $\mathcal G(n,p)$ with $p=1/2$  is close to 1/2 with high probability.
Another notion of distance, called \textit{cut distance}, better reflects the structural similarity. The cut norm of a matrix $\bB=(\bB_{ij})\in \mathbb{R}^{n\times n}$ has been introduced by Frieze and Kannan \cite{frieze_kannan}. It is defined by 
\[\|\bB\|_{\cut}=\frac{1}{n^{2}}\underset{S,T\subset [n]}{\max} \Big|\sum_{i\in S,j\in T}\bB_{ij}\Big\vert.\]
In other words, $\|\bB\|_{\cut}$ corresponds (up to to a renormalization) to the maximal sum of entries over all submatrices of $\bB$. 
Then, the cut distance $d_{\square}(G,G')$ between two graphs $G$ and $G'$ defined on the same set of nodes and with adjacency matrices $\bA$ and $\bA'$ is defined as the cut norm $\|\bA-\bA'\|_{\cut}$. Denoting $e_G(S,T)$ the number of edge between nodes in $S$ and $T$ in the graph $G$, the cut distance $d(G,G')$ is the supremum over all $S,T$ of $(e_G(S,T)- e_{G'}(S,T))/n^2$. In other words, $d_{\square}(G,G')$ is small if the restrictions of $G$ and $G'$ to all subsets $S,T$ have similar edge densities.

Let us denote $\cW$ the collection of symmetric measurable functions $[0,1]^2\rightarrow [-1,1]$. 
By analogy with the matrix cut norm,  we can define the cut norm of a  kernel  $W\in \cW$: 
\begin{equation} \label{def_cut_norm}
\|W\|_{\cut}=\underset{S,T\subset [0,1]}{\sup} \left \vert\underset{S\times T}{\int}W(x,y)\mathrm{d} x\mathrm{d} y\right \vert\ ,
\end{equation}
where the supremum is taken over all measurable subsets $S$ and $T$. Then, the distance $d_{\cut}(W,W')$ between two graphons $W$ and $W'$ in $\cW^+$ is simply $\|W-W'\|_{\cut}$. As explained earlier in the introduction, graphons in $\cW^+$ are not identifiable. This is why we consider the metric induced by $\Vert\cdot\Vert_{\cut}$
on the quotient space  $\widetilde{\cW}^+$ defined by 
  \begin{equation}\label{eq:def_cut_graphon_norm}
 \delta_{\cut}(W_1,W_2)=\underset{\tau\in \mathcal M}{\inf}\Vert W_1-W^{\tau}_2\Vert_{\cut}\ , 
 \end{equation}
 where we take the infimum in the set $\mathcal M$ of all measure-preserving bijections $\tau : [0,1]\rightarrow [0,1]$ and $W^{\tau}(x,y)=W(\tau(x),\tau(y))$. 
 
 The cut distance is also a cornerstone in the  graph limit theory introduced by Lov{\'a}sz and Szegedy \cite{LovaszSzegedy}
  and further developed in, e.g., \cite{borgs_chayes_2008,borgs_chayes_2012}. In particular, this theory states that graphons can be interpreted as limits (with respect to $\delta_{\square}$) of  graph sequences. Besides,  convergence in $\delta_{\cut}$ is equivalent to other structural properties such as the convergence of all homomorphisms numbers. Given a simple graph $F$ with $q$ nodes and a graphon $W_0$, the homomorphisms number $t(F,W_0)$ is the probability that the edge set  of size $q$ of a graph sampled from the model~\eqref{sparse_graphon_mod} (with $\rho_n=1$)  contains the edge set of $F$. As a consequence, the homomorphisms numbers $t(F,W_0)$ and $t(F,W'_0)$ are close  when the expected number of subgraphs $F$ for a size $n$  random graph $G$ sampled from $W_0$ is close to that of a size $n$ random graph sampled from $W'_0$. It has been established that convergence in the cut distance is equivalent to convergence of homomorphism numbers for all simple graphs $F$ (see Theorem 11.5 in~\cite{LovaszBook} for more details). Hence, estimating well the graphon $W_0$ in the cut distance allows to estimate well the number of small patterns induced by $W_0$.
On the other hand, the cut distance controls other quantities of interest for computer scientists such as the size of multi-way cuts~\cite{borgs2014p2,borgs2015}. So, a consistent estimator of $W_0$ in cut distance gives consistent estimators for the multi-way cuts.

 \medskip

 The construction of $\delta_{\square}$  can be extended to any other norm $N$ that is invariant under measure preserving maps:
  \begin{equation}\label{eq:def_graphon_norms}
  \delta_{N}(W_1,W_2)=\underset{\tau\in \mathcal M}{\inf}\Vert W_1-W^{\tau}_2\Vert_{N}.
  \end{equation}
 Besides the cut norm, we already mentioned  the $l_1$ and $l_2$-norms on $\cW$ defined by $\|W\|_1= \int_{[0,1]^2} |W(x,y)|dxdy$ and $\|W\|_2= [\int_{[0,1]^2} W^2(x,y)dxdy]^{1/2}$. These two norms define the corresponding distances  $\delta_{1}$ and  $\delta_{2}$ on the quotient space $\widetilde{\cW}^+$. The distance $\delta_{\square}$ is dominated by $\delta_{1}$ and $\delta_{2}$ (for details see Section \ref{preliminaries}). As already noted for instance in \cite{borgs2015}, this immediately implies that the convergence rate of an estimator $\widehat{W}$ with respect to the $\delta_{\square}$-distance is at least as fast as its convergence rate with respect to the $\delta_2$-distance. Then, one may wonder whether 
 the convergence rates in $\delta_{\square}$-distance can be significantly faster and whether those faster rates are achieved by the estimators that are already minimax optimal with respect to other metrics.

\medskip 

In fact, a partial result on uniform convergence rates has already been proved. One of the striking consequences of the celebrated Szemer\'edi's Lemma~\cite{szemeredi1975regular} states that an adjacency matrix sampled from a $W$-random graph model converges to the true graphon $W_0$ in cut distance,  this at  an  {\it uniform} rate over all graphons. To be more specific, let $W_0\in \cW^{+}$ be a graphon and let $\bA$ be the size $n$ adjacency matrix sampled according to the $W$-random graph model~\eqref{sparse_graphon_mod} with $\rho_n=1$.   It has been shown in  \cite{borgs_chayes_2008} (see also \cite{alon2003random} or \cite{LovaszBook})  that, with high probability,  the empirical graphon $\widetilde{f}_{\bA}$ associated to the adjacency matrix $\bA$ (see \eqref{eq:empirical_graphon} for a precise definition) is $O(1/\sqrt{\log(n)})$ close in the cut distance to the true graphon $W_0$: 
  \begin{prp}[Lemma 10.16 \cite{LovaszBook}]\label{sampling_lemma_lovasz}
  Let $n\geq 1$ and let $W_0\in \cW^{+}$ be a graphon. Then, with probability at least $1-\exp\left \{-n/(2\log n)\right \}$,
  \beq\label{eq:convergence_cut_norm}
  \delta_{\cut} \left (\tilde{f}_{\bA}, W_0\right )\leq \dfrac{22}{\sqrt{\log(n)}}\ .   \eeq
  \end{prp}
An important point is that the above result is valid for all $W_0\in \cW^+$.  Note that if we replace the cut-distance by $\delta_1$ or $\delta_2$-distance this is not true any more: even in the simple case of a constant graphon $W_0\equiv a$ (with $a\in (0,1)$), the $l_2$ distance between $\tilde{f}_{\bA}$ and $W_0$ does not converge to zero.

\subsection{Our contribution and related results}

Our purpose in this paper is to go beyond uniform convergence rates over all graphons in $\cW^+$ and to understand the optimal cut distance convergence rates when $W_0$ 
has a specific structure.  First,  optimal convergence rates are derived for the estimation of the connection probability matrix $\bTheta_0$  when it belongs to classes of block-constant matrices. Second, we establish the optimal convergence rates for all classes of step-function graphons $f=\rho_n W_0$ both in sparse and dense case. In particular for $\rho_n=1$ (dense case),  our results imply that, for any integer $k\in [2,n]$ and  $k$--steps graphon $W_0$, one has
\beq \label{eq:upper_cut}
   \E_{W_0}\left[\delta_{\cut} \left (\widetilde{f}_{\bA}, W_0\right ) \right]\leq C\sqrt{\frac{k}{n\log (k)}}\ ,
\eeq 
where $C$ is a numerical constant (independent of $n$ and $k$) and that this convergence rate is optimal from the minimax point of view. This result has some interesting implications. In particular, this guarantees the optimality of the $\log(n)^{-1/2}$ rate in Proposition \ref{sampling_lemma_lovasz} for general graphons. On the other hand, our results imply that for more structured classes of graphons ($k\ll n$) much faster rates are achievable. 
Interestingly, we show that the adjacency matrix and its associated empirical graphons are already adaptive to the unknown number of blocks of the matrix $\bTheta_0$ or steps of $W_0$ and minimax optimal. As a consequence, there is no need to look for more involved estimators.

In practice, it could be disappointing that the raw data are already optimal with respect to the cut distance, whereas they perform really badly with respect to the $\delta_2$ distance. This is why we prove 
 that a singular value hard thresholding estimator  is still optimal with respect to the cut metric $\delta_{\square}$ while achieving the best known rate in $\delta_{2}$-distance in the  class of polynomial-time estimators.

 \medskip 
 
 Our results are in sharp contrast to all aforementioned manuscripts \cite{WolfeOlhede,chan2014consistent, borgs2015, airoldi2013stochastic,cai2014iterative,zhang2015estimating,klopp_graphon} whose primary focus is the $\delta_2$-distance  and whose convergence rates with respect to the $\delta_{\square}$-distance are derived from the domination of $\delta_{\square}$ by $\delta_{2}$. Closest to our contributions, is the recent paper \cite{borgs_chayes_cohn} where the authors  introduce a least-cut norm estimator for a more general model of unbounded graphons. Translated in our framework, their non-polynomial time algorithm achieves, in some cases, the optimal convergence rate (up to a logarithmic loss) and it is slower in other cases. In Section \ref{sec:unbounded} we extend our study to  unbounded graphons and compare our results to those of \cite{borgs_chayes_cohn}. In particular, our Proposition \ref{prop:unbounded} implies that the empirical graphon associated to the adjacency matrix and to the singular value hard thresholding estimator are optimal (up to a logarithmic factor) also in the general case of unbounded graphons. Note that the main difference with the method proposed in \cite{borgs_chayes_cohn} is that both our estimators can be easily computed in polynomial time.
 
 \medskip

 From a technical point of view, the tools needed for deriving  optimal cut distance rates differ from those used for the $\delta_2$-distance. For establishing the minimax lower bounds, the main technical hurdle is to build a collection of well-spaces graphons with respect to the cut distance. Indeed, the cut distance  $\delta_{\square}(W_1,W_2)$ is difficult to lower bound as it is defined as an infimum over all measure-preserving transformations.  As for the minimax upper bound on the estimation error in \eqref{eq:upper_cut}, it can be obtained  quite easily without the correct logarithmic term thanks to the Bernstein's inequality together with some bounds from \cite{klopp_graphon} for the stronger metric $\delta_2$. 
 However, 
 recovering the right logarithmic term in \eqref{eq:upper_cut} is much more challenging. The proof relies among other things on a careful application of Szemer\'edi's regularity lemma to distorted versions of the graphon.

 \medskip 
 
 The manuscript is organized as follows. First, we recall some basic results related to the cut metric.  The problem  of estimating the matrix of connection probabilities is considered in Section \ref{sec:proba_matrix}. We study the problem of graphon estimation in Section \ref{sec:graphon}. The appendix  contains all the proofs where in Appendix \ref{sec:proof_methods} we recall some  basic facts and  results that are often used in the proofs.

\section{Notation and Preliminaries}

\subsection{Notation}\label{sec:notation}

We gather here some of the notation used throughout this paper. Some of them have already been defined in the introduction.
\begin{itemize}
\item For a matrix $\bB$,  $\bB_{ij}$ (or $\bB_{i,j}$, or $(\bB)_{ij}$) is its $(i, j)$-th entry.  Let $\bB_{i,\cdot}$ and  $\bB_{\cdot,j}$ stand for its $i$th row and $j$th column respectively.  We denote by $\mathbb{R}^{k\times k}_{\rm sym}$ the class of all symmetric $k\times k$ matrices with real-valued entries. Given a matrix $\bB$ and $p\in [1,\infty]$, $\|\bB\|_p$ denotes its entry-wise $l_p$ norm, that is $\|\bB\|_p^{p}= \sum_{i,j}|\bB_{ij}|^p$ for $p<\infty$ and $\|\bB\|_{\infty}= \max_{i,j}|\bB_{ij}|$. Given $(p,q)\in [1,\infty]$, $\|\bB\|_{p\rightarrow q}$ stands for  its $l_p\rightarrow l_q$ operator norm:
	\begin{equation*}
	\|\bB\|_{p\rightarrow q}=\inf\{c\geq 0:\Vert \bB v\Vert_{l_q}\leq c \Vert  v\Vert_{l_p}\;\text{for all}\; v\in l_p\}
	\end{equation*}

 Finally, $\langle \bD, \bB\rangle = \sum_{i,j} \bD_{ij}\bB_{ij}$ stands for the  canonical  inner product between matrices $\bD,\bB\in \mathbb{R}^{n\times n}$.

\item  $\cW$ is the collection of symmetric measurable functions $[0,1]^2\rightarrow [-1,1]$. Given a kernel $W\in \cW$ and $p\in (1,\infty)$, its $l_p$ norm is defined by $\|W\|^p_p=\int |W(x,y)|^pdxdy$, whereas $\|W\|_{\infty}=\mathrm{ess\ sup}_{x,y} |W(x,y)|$. $\cW^+$ is the space of graphons and $\widetilde{\cW}^+$ is the corresponding quotient space. The  cut distance $\delta_{\square}(\cdot,\cdot)$ in the graphon spaces is  defined by \eqref{eq:def_cut_graphon_norm}. Also, $\delta_{1}(\cdot,\cdot)$ and $\delta_{2}(\cdot,\cdot)$ defined by \eqref{eq:def_graphon_norms} respectively correspond to the $l_1$ and $l_2$ distances on the quotient space of graphons $\widetilde{\cW}^+$.  Given a symmetric square matrix $\bTheta$ with values in $[0,1]$,  $\widetilde{f}_{\bTheta}$ is the empirical graphon  $\bTheta$ as defined in \eqref{eq:empirical_graphon}.

\item Given a probability matrix $\bTheta_0$, we  denote by $\E_{\bTheta_0}$ the expectation with respect to the distribution of $\bA$ if we consider  the inhomogeneous random graph model and given a graphon $W$ and $\rho_n$, we write $\E_{W}$ for the expectation with respect to the joint distribution of $({\boldsymbol\xi},\bA)$.

\item We denote by $\lfloor x\rfloor$  the maximal integer less than or equal to $x$ and by $\lceil x \rceil$  the smallest integer greater than or equal to $x$. For an positive integer $m$, set $[m]=\{1,\dots,m\}$. $\mathds{1}_{A}(\cdot)$ denotes the indicator function of a set $A$.  In the sequence,  $C$ stands for a positive constant that can vary from line to line. These are absolute constants unless otherwise mentioned. For two positive functions $f$ and $g$, we write $f\asymp g$ when there exist two positive numerical constants $C$ and $C'$ such $C g\leq  f\leq C' g$.  Finally, $\lambda$ is the Lebesgue measure on the interval $[0,1]$.

\item  Given a $n\times n$ matrix $\bTheta$ with entries in $[0,1]$, we define the empirical graphon $\widetilde{f}_{\bTheta}$ as the following piecewise constant function: $
	\widetilde{f}_{\bTheta}(x,y)= \bTheta_{\lceil nx\rceil, \lceil ny \rceil}$
	for all  $x$ and $y$ in $(0,1]$.

\end{itemize}

\subsection{Preliminaries}\label{preliminaries}

We start with a few basic properties of the cut norm for matrices $\bA$ and graphons $W$.  It is easy to see that  $$\|\bA\|_{\cut}\leq \frac{1}{n^{2}}\|\bA\|_{1}\leq \frac{1}{n} \|\bA\|_{2}$$
where $\|\cdot\|_{1}$ and $\|\cdot\|_{2}$ are the usual entry-wise $l_1$ and $l_2$-norms of a matrix. 
     For a function $W\in \cW$, we have 
$$\|W\|_{\cut}\leq \|W\|_{1}\leq  \|W\|_{2}\leq  \|W\|_{\infty}\leq 1$$
 where $\|\cdot\|_{1}$ and $\|\cdot\|_{2}$ denote  $l_1$ and $l_2$-norms of a graphon. In the opposite direction, we have $\|W\|_{2}\leq  \sqrt{\|W\|_{1}}$. As a consequence, the metric $\delta_1$ and $\delta_2$  define the same topology on the space $\widetilde{\cW}^+$ of graphons. In contrast, the cut distance $\delta_{\square}$ defines a weaker topology on the space $\widetilde{\cW}^+$ as illustrated by the aforementioned sampling result (Proposition \ref{sampling_lemma_lovasz}).

 We shall also sometimes rely on the equivalence between the cut norm and to the $l_{\infty}\rightarrow l_{1}$ operator  norm:
  \begin{equation}\label{eq:def_1infty_operator_norm}
  \|W\|_{\infty\rightarrow 1}=\underset{\Vert f\Vert_{\infty},\Vert g\Vert_{\infty}\leq 1}{\sup} \left \vert\underset{[0,1]^{2}}{\int}W(x,y)f(x)g(y)\mathrm{d} x\mathrm{d} y\right \vert
    \end{equation}
  where the supremum is taken over all (real-valued) functions $f$ and $g$ with values
  in $[-1, 1]$. It is known that (see e.g.,\ \cite{janson_cut})
   \begin{equation}\label{eq:cut_operator}
   \|W\|_{\cut}\leq \|W\|_{\infty\rightarrow 1}\leq 4 \|W\|_{\cut}\ .
  \end{equation}

\section{Probability matrix estimation}\label{sec:proba_matrix}

\subsection{Cut norm minimax risk}
We start with a simple proposition that bounds the expected cut distance between $\bTheta_0$ and the sampled adjacency matrix $\bA$. 	
Similar results already appeared in the literature, see e.g.,\ \cite[Lemma 10.11]{LovaszBook}, \cite{borgs_chayes_cohn} or \cite{guedon2016}. 
Its proofs is based on Bernstein's inequality and is given  in Section \ref{proof_lemma_LovaszBook}. 
\begin{prp}\label{lemma_LovaszBook}
	For any probability matrix $\bTheta_0$ we have
	\begin{equation}\label{eq:bound_bernstein_l1}
	\E_{\bTheta_0}\|\bA - \bTheta_0\|_{\square}\leq 12\sqrt{\frac{\Vert \bTheta_0\Vert_{1}+n}{n^{3}}}.
	\end{equation}
	In particular, if $\Vert \bTheta_0\Vert_{\infty}\geq 1/n$, we get
	 \begin{equation*}
	 \E_{\bTheta_0}\|\bA - \bTheta_0\|_{\square}\leq 24\sqrt{\frac{\Vert \bTheta_0\Vert_{\infty}}{n}}.
	 \end{equation*}
\end{prp}

This implies that the adjacency matrix $\bA$ is $\sqrt{\Vert \bTheta_0\Vert_{\infty}/n}$-close in cut-distance to the probability matrix $\bTheta_0$. This bound is valid for all matrices $\bTheta_0$. It turns out that no estimator can perform much better than $\bA$, even on some simple classes of parameters $\bTheta_0$.

  Let $n,k$ be integers such that $2\le k\le n$ and $\cT[k]$ be defined by 
 $$\cT[k]= \{\bTheta_0:\ \exists \ z\in \cZ_{n,k},\ \bQ\in [0,1]^{k\times k}_{\rm sym} \text{ such that } \bTheta_{ij}=\bQ_{z(i)z(j)}, \ i\ne j,   \ \text{ and }  \bTheta_{ii}=0 \ \forall  i\}$$
 where we denote by $\cZ_{n,k}$ the set of all mappings $z$ from $[n]$ to $[k]$. In other words $\cT[k]$ is made of matrices that, up to a permutation of their rows and their columns, are (up to the diagonal) block constants with at most $k$ blocks. Also, this corresponds to connection probability matrices of $k$-class stochastic blocks models whose vector label $Z=(Z_a)$ has been fixed.  For any $\rho_n\in (0,1]$, consider the set 
 \beqn 
 \cT[k,\rho_{n}]= \left\{\bTheta_0\in \cT[k]  :\ \|\bTheta_0\|_{\infty}\leq \rho_{n}\right\}\ ,
 \eeqn
 of matrices whose largest value is smaller or equal to $\rho_n$.  The following Proposition, proved in section \ref{proof_low-bound_dense}, gives a lower bound on the minimax risk over the class $\cT[2,\rho_{n}]$ of block-constant matrices with only two blocks:
\begin{prp}\label{lemma_low-bound_prob}
The minimax risk measured in cut norm satisfies
 \beqn
 \inf_{\widehat{\bTheta}}\sup_{\bTheta_0\in \cT[2,\rho_{n}]}\E_{\bTheta_0}\left[\left\Vert\widehat{\bTheta}-\bTheta_0\right\Vert_{\cut}\right]\geq C\min\left(\sqrt{\frac{\rho_n}{n}},\rho_n\right)
 \eeqn
 where $\E_{\bTheta_0}$ denotes the expectation with respect to the distribution of $\bA$ when the underlying probability matrix is $\bTheta_0$.
\end{prp}

Comparing Proposition \ref{lemma_low-bound_prob} with Proposition \ref{lemma_LovaszBook} we observe that the raw data $\bA$ is minimax optimal for the class $\cT[2,\rho_n]$ for all $\rho_n\geq 1/n$. As a consequence, there is no need to look for a more involved estimator.
Since for $\rho_n\leq 1/n$ the constant estimator $\widehat{\bTheta}=0$ satisfies $\E_{\bTheta_0}[\Vert\widehat{\bTheta}-\bTheta_0\Vert_{\cut}]\leq \rho_n$ and using that the collections $\cT[k,\rho_{n}]$ are nested, the two previous propositions 
imply that the optimal cut norm estimation rates for $\cT[k,\rho_n]$ with $k\geq 2$ is given by 
 $$
 \inf_{\widehat{\bTheta}}\sup_{\bTheta_0\in \cT[k,\rho_{n}]}\E_{\bTheta_0}\left[\left\Vert\widehat{\bTheta}-\bTheta_0\right\Vert_{\cut}\right]\asymp
 \min\left(\sqrt{\frac{\rho_n}{n}},\rho_n\right).
 $$

 \bigskip 
 
 Until now, we left aside the specific case of constant matrices $\cT[1,\rho_n]$ which correspond to Erd\"os-R\'enyi random graphs. 
It turns out that the situation is quite different for this simple class. 
For a constant matrix $\bTheta_0$, estimating  $\bTheta_0$ given $\bA$ amounts to infer the parameter $p$ of a Bernoulli distribution given a sample of size $n(n-1)/2$. From this analogy, we consider the  matrix $\overline{\bA}$ whose all non-diagonal entries are equal to $\bar A=\sum_{i,j}\bA_{ij}/(n(n-1))$. Then, it is straightforward to prove that 
\[
 \E_{\bTheta_0}\Big[\|\bTheta_0- \overline{\bA}\|_{\square }\Big]\leq \dfrac{1}{n^{2}}\E_{\bTheta_0}\Big[ \|\bTheta_0- \overline{\bA}\|_{1} \Big]\leq \dfrac{n-1}{n}\sqrt{\Var(\bar A)} \leq \sqrt{\frac{2p}{n(n-1)}}\ , 
\]
which is $\sqrt{n}$-faster than what is achieved by the adjacency matrix $\bA$. Using again the analogy with the problem of Bernoulli parameter estimation, one may easily get the following minimax lower bound: 
\[
 \inf_{\widehat{\bTheta}}\sup_{\bTheta_0\in \cT[1,\rho_{n}]}\E_{\bTheta_0}\left[\left\Vert\widehat{\bTheta}-\bTheta_0\right\Vert_{\cut}\right]\geq C
 \min\left(\frac{\sqrt{\rho_n}}{n},\rho_n\right) 
\]
which assesses that the $\sqrt{\rho_n}/n$-rate achieved by $\overline{\bA}$ is optimal.

\subsection{Comparison with $l_1$ and $l_2$-estimation }\label{sec:discussion_comparison_matrix_estimation}

The cut norm optimal estimation rate  is quite different from what has been established for the Frobenius norm (also called $l_2$) estimation rate in \cite{klopp_graphon} (see also \cite{gao2014rate} for the dense case), that is 
 \beq\label{eq:optimal_frobenius}
 \inf_{\widehat{\bTheta}}\sup_{\bTheta_0\in \cT[k,\rho_{n}]}\E_{\bTheta_0}\left[\dfrac{1}{n}\left\Vert\widehat{\bTheta}-\bTheta_0\right\Vert_{2}\right]\asymp \min\left(\sqrt{\frac{\rho_n\log(k)}{n}}+\dfrac{\sqrt{\rho_n}k}{n},\rho_n\right)\ , 
 \eeq
 for any $k=2, \ldots, n$. Besides, the minimax risk bound is achieved  by the restricted least-square estimators~\cite{klopp_graphon} defined by
\beq\label{eq:definition_restricted_least_squares}
\widehat{\bTheta}_{k,\rho_n}:= \arg\min_{\bTheta\in \cT[k,\rho_{n}]}\|\bTheta-\bA\|_2^2\ . 
\eeq
 Since the Frobenius norm dominates the cut norm, it is expected that  the cut norm convergence rate is faster than the Frobenius norm estimation rate. When $\rho_n$ is not too small and the number of blocks remains small ($k\leq \sqrt{n\log(n)}$),  the gain is  a $\log(k)$ factor, whereas, for larger $k$, the gain is of order  $k/\sqrt{n}$. 
More importantly, the optimal Frobenius norm convergence rate \eqref{eq:optimal_frobenius} is only known to be achieved by non-polynomial time estimators such as  \eqref{eq:definition_restricted_least_squares}.

\bigskip

In view of the above discussion, one may wonder whether it is possible to build estimators that are  near optimal is terms of  both the cut and Frobenius distances. Since for any matrix $\bB$,   $\Vert \bB\Vert_{\cut}\leq \Vert  \bB\Vert_{2}/n$, it follows that, for $k\leq \sqrt{n}$, the  restricted least-square estimator $\widehat{\bTheta}_{k,\rho_n}$~\eqref{eq:definition_restricted_least_squares} 
 is also near optimal (up to $\sqrt{\log(k)}$ factor) with respect to the cut distance, that is,
 \[
\E_{\bTheta_0}\Big[\big\|\widehat\bTheta_{k,\rho_n} - \bTheta_0\big\|_{\square}\Big]\leq  C\sqrt{\frac{\rho_n \log(k)}{n}}\ . 
\]
For matrices $\bTheta_0$ with more than $\sqrt{n}$ blocks, it is not clear whether the estimator $\widehat\bTheta_{k,\rho_n}$ achieves a fast rate of convergence in the cut norm.

\bigskip

In any case, the computational complexity of  $\widehat{\bTheta}_{k,\rho_n}$ is  non polynomial. In fact, no polynomial-time algorithm is known to achieve the minimax risk \eqref{eq:optimal_frobenius} with respect to the Frobenius norm. Below, we describe an estimator that is optimal in the cut distance and also achieves the best known rate in Frobenius distance in the class of polynomial-time estimators. Let us write the singular value decomposition of $\bA$:
\begin{equation}
\bA=\overset{\rank(\bA)}{\underset{j=1}{\Sigma}}\sigma_j(\bA)u_j(\bA)v_j(\bA)^{T},
\end{equation}
where 
$\sigma_j(\bA)>0$ are the singular values of $\bA$ indexed in the decreasing order, $u_j(\bA)$ are eigenvectors of $\bA$ and $v_j(\bA)=\pm u_j(\bA)$.
Given a tuning parameter $\lambda>0$,  we define 
\begin{equation}\label{hard_estimator}
\widetilde \bTheta_{\lambda}=\underset{j:\sigma_j(\bA)\geq \lambda}{\Sigma}\sigma_j(\bA)u_j(\bA)v_j(\bA)^{T}
\end{equation}
as the  singular value hard thresholding  estimator of $\bTheta_0$. We have the following 
 \begin{prp}\label{corollary_spectral}
 Assume that  $\rho_n\geq \log(n)/n$. 
 Let $\lambda=c\sqrt{\rho_nn}$ where $c$ is a sufficiently large numerical constant.  Then, for any $k\in [n]$ and any  $\bTheta_0\in \cT[k,\rho_n]$,
 the hard thresholding estimator $\widetilde \bTheta_{\lambda}$ simultaneously satisfies, with probability larger than $1-1/n$, 
 \begin{eqnarray}
 \frac{1}{n}\Vert \widetilde \bTheta_{\lambda}-\bTheta_0\Vert_{2}&\leq& C\sqrt{\frac{\rho_n k}{n}}\ ,\label{frobenius_bound}\\
 \Vert \widetilde \bTheta_{\lambda}-\bTheta_0\Vert_{\cut}\leq \dfrac{1}{n}\Vert \widetilde \bTheta_{\lambda}-\bTheta_0\Vert_{2\rightarrow 2}&\leq& C\sqrt{\dfrac{\rho_n}{n}}\label{cut_bound}\ ,
 \end{eqnarray}
 where $C$ is a numerical constant. 
 \end{prp}
The low-rank estimator $ \widetilde \bTheta_{\lambda}$ was previously considered in  \cite{Chatterjee_mc} for Frobenius norm estimation, but error bounds obtained in \cite{Chatterjee_mc} are much more pessimistic than \eqref{frobenius_bound}. It follows from \eqref{cut_bound}, that   for $\rho_n\geq \log(n)/n$, with high probability, $ \widetilde \bTheta_{\lambda}$ achieves the optimal rate in the cut norm  and  the $\sqrt{\rho_n k/n}$ rate in Frobenius norm, which is the best known rate among polynomial-time estimators.

We close this section by the following proposition which gives the minimax optimal rate of estimation in $l_1$-norm. This will allow us to further compare the $\delta_1$ and $\delta_{\cut}$ convergence rates for graphon estimation in the next section. 
\begin{prp}\label{prop:matrix_l_1}
For any sequence $\rho_n>0$ and any positive integer $2\leq k\leq n$, one has
 \beq\label{eq:minimax_matrix_L1}
  \inf_{\widehat{\bTheta}}\sup_{\bTheta_0\in \cT[k,\rho_{n}]}\E_{\bTheta_0}\left[\frac{1}{n^{2}}\left\Vert\widehat{\bTheta}-\bTheta_0\right\Vert_{1}\right] \asymp \min\left\{\sqrt{\frac{\rho_n\log(k)}{n}}+ \frac{\sqrt{\rho_n}k}{n},\rho_n\right\}\ .
  \eeq
\end{prp}
To prove the upper bound we can use the following result which provides the control of the estimation error measured in Frobenius norm of the restricted least-squares estimator $\bTheta^{r}$ proven in \cite{klopp_graphon}:
	
\begin{prp}\label{prp:l2_Thres} Consider the network sequence  model.
	There exist positive absolute constant $C$ such that the following holds. If $\|\bTheta_0\|_{\infty}\le \rho_n$, then \beq\label{eq:risk_l2thres}
	\E\left[\frac{1}{n^{2}}\|\widehat{\bTheta}^{r}- \bTheta_0\|_F^2\right]\leq C \rho_n\left (\frac{\log(k)}{n}+ \frac{k^2}{n^{2}}\right ) \ .
	\eeq
\end{prp} 
	The upper bound in  \eqref{eq:minimax_matrix_L1} is a consequence of the inequality $\Vert \bB\Vert_{1}/n^2\leq \Vert \bB\Vert_{2}/n$ and \eqref{eq:risk_l2thres}. The lower bound of the minimax risk in \eqref{eq:minimax_matrix_L1} is proved following the same lines as the proof of Proposition 2.4 in \cite{klopp_graphon} with $\Vert\cdot\Vert_{2}$ replaced by $\Vert\cdot\Vert_{1}$. We skip the details.

\section{Graphon estimation problem}\label{sec:graphon}

In this section, we are interested in  estimating the graphon $W_0$ in the sparse $W$-random graph model \eqref{sparse_graphon_mod}.
 Let $\cW^+[k]$ be the collection of $k$--step graphons, that is, the subset of graphons $W\in \cW^{+}$ such that for some $\bQ\in [0,1]^{k\times k}_{\text{sym}}$ and some $\phi:[0,1]\to [k]$, 
\beq \label{eq:def_step_function}
W(x,y)= \bQ_{\phi(x),\phi(y)}\quad \text{  for all }x,y\in [0,1]\ .
\eeq
Note $\cW^+[k]$ is also in correspondence with the collection of stochastic block models with $k$ blocks. Our purpose here, is to characterize the minimax convergence rates over classes $\cW^+[k]$.

\subsection{Cut distance minimax risk}

 Following \cite{klopp_graphon},  we start by associating a graphon to any $n\times n$ probability matrix $\bTheta_0$. Then, we can estimate graphon $f_0(\cdot,\cdot)= \rho_n W_0(\cdot,\cdot)$ using the empirical graphon associated to an estimate of $\bTheta_0$.
Recall that, given a $n\times n$ matrix $\bTheta$ with entries in $[0,1]$, we define the graphon $\widetilde{f}_{\bTheta}$ as the following piecewise constant function:
 \beq\label{eq:empirical_graphon}
 \widetilde{f}_{\bTheta}(x,y)= \bTheta_{\lceil nx\rceil, \lceil ny \rceil}
 \eeq
 for all  $x$ and $y$ in $(0,1]$. 
 For any estimator $\widehat \bT$ of $\bTheta_0$ and any  norm $N$ that is invariant under measure preserving maps   the triangle inequality implies
\beq\label{eq:integrated_risk_decomposition}
\E_{W_0}\left[\delta_{N}(\widetilde{f}_{\widehat{\bT}}, f_0)\right]\leq \E_{W_0}\left[\|\widehat{\bT}-\bTheta_0\|_N\right] +\E_{W_0}\left[\delta_{N}\left(\widetilde{f}_{\bTheta_0} , f_0\right)\right].
\eeq
 We have two parts in \eqref{eq:integrated_risk_decomposition}. The first term is the {\it estimation error} term $\|\widehat{\bT}-\bTheta_0\|_{N}$ that has been considered in the previous section. The second term  $\delta_{N}(\widetilde{f}_{\bTheta_0} , f_0)$ is the {\it agnostic error}. It measures the $\delta_{N}$-distance between the true graphon $f_0$ and its discretized version sampled at the unobserved random design points $\xi_1,\ldots , \xi_n$. The behavior of $\delta_{N}(\widetilde{f}_{\bTheta_0} , f_0)$ depends on the topology of the considered class of graphons.
The following theorem, proved in Section \ref{proof_agnostic_cut}, gives the upper bound on the agnostic error,  measured in $\delta_{\cut}$-distance for step function graphons:
\begin{thm}[Agnostic error measured in cut distance]\label{prp_agnostic_block}
Consider the $W$-random graph model \eqref{sparse_graphon_mod}. 
For all integers $k\geq 2$, all positive integers $n$, all $W_0\in \cW^+[k]$ and $\rho_n>0$, we have 
\beqn 
 \E_{W_0}\left[\delta_{\square}(\widetilde f_{\bTheta_0}, f_0)\right]\leq C\rho_n\left\{
 \begin{array}{cc}
  \sqrt{\dfrac{k}{n\log(k)}} &\text{ if $k\leq n$,}\\
  \sqrt{\dfrac{1}{\log(n)}} &\text{ if $k> n$.}
 \end{array}
\right. 
 \eeqn
\end{thm}
Note that the case $k>n$ is a consequence of Proposition \ref{sampling_lemma_lovasz} from \cite{LovaszBook}, so that we effectively only have to consider the case $k\leq n$. The proof combines two ideas. First, we  build $W$ and $\widehat{W}$  as the representatives of $W_0$ and $\widetilde{f}_{{\bTheta}_0}$ in the quotient space $\widetilde{\cW}^+$ such that $W$ and $\widehat{W}$ match everywhere except on a set of Lebesgue measure of order at most $\sqrt{k/n}$. This allows us to get a risk bound of order $\sqrt{k/n}$. In order to recover the correct logarithmic factor $\sqrt{\log(k)}$, we rely on the weak Szemer\'edi's Lemma. Here, the key idea is to build a cut-norm approximation of a distorted transformation of $W$ where the weights of the group have been modified to take into account the geometry of the sampling error.

As an immediate consequence of \eqref{eq:integrated_risk_decomposition},   Proposition \ref{lemma_LovaszBook} and Theorem \ref{prp_agnostic_block}, we get the following upper bound on the risk of the empirical graphon $\widetilde{f}_{\bA}$. For any $k\geq 2$, it holds that 
\beq\label{eq:minimax_risk_graphon}
\sup_{W_0\in \cW^+[k]}\E_{W_0}\left[\delta_{\square}\left(\widetilde{f}_{\bA}, f_0\right)\right]
\leq C \min \left(
       \rho_n \left(\sqrt{\frac{k}{n\log(k)}}, \frac{1}{\sqrt{\log(n)}}\right)+\sqrt{\dfrac{\rho_n}{n}}\right)\ , 
\eeq
where $C$ is an absolute constant. Here, $\E_{W_0}$ denotes the expectation with respect to the distribution of observations $\bA=(\bA_{ij}, 1\le j<i\le n)$ when the underlying sparse graphon is $f_0=\rho_n W_0$. The following Proposition provides a  matching lower bound for $2\leq k\leq n$.
\begin{thm}\label{prp:lower_minimax_cut}
There exists a universal constant $C>0$ such that for any sequence $\rho_n>0$ and any positive integer $2\leq k\leq n$, 
\beq \label{eq:lower_minimax_integrated_risk_distance_cut_prp}
 \inf_{\widehat{f}}\sup_{W_0\in \cW^+[k]}\E_{W_0}\left[\delta_{\square}\left(\widehat{f}, f_0\right)\right]\geq C \min\left(
       \rho_n \sqrt{\frac{k}{n\log(k)}}+\sqrt{\dfrac{\rho_n}{n}},
        \rho_n\right)\ ,
 \eeq
 where $\inf_{\widehat{f}}$ is the infimum over all estimators. 
\end{thm} 
 Since the collections $\cW^+[k]$ are nested, it follows that for all $k\geq n$, one has 
 \[
\inf_{\widehat{f}}\sup_{W_0\in \cW^+[k]}\E_{W_0}\left[\delta_{\square}\left(\widehat{f}, f_0\right)\right]\geq C \min \left( 
       \rho_n \sqrt{\frac{1}{\log(n)}}+\sqrt{\dfrac{\rho_n}{n}},        \rho_n\right). 
 \]
 In view of \eqref{eq:minimax_risk_graphon} and \eqref{eq:lower_minimax_integrated_risk_distance_cut_prp}, we observe that, as long as, $\rho_n\geq 1/n$,  the empirical graphon $\widetilde{f}_{\bA}$ is minimax optimal over all classes $\cW^+[k]$, $k\geq 2$. For sparser graphs ($\rho_n\leq 1/n$), the trivial estimator $\widehat{f}\equiv 0$ achieves the optimal rate $\rho_n$. 
 
Note that there are two distinct regimes in the minimax convergence rate. When $\rho_n\geq \log (k)/k$ (weakly sparse graphs or large number of groups),  the agnostic error dominates and the minimax risk is  of order $\rho_n\sqrt{k/(n\log(k))}$. For moderately sparse graphs or equivalently a small number of steps ($n^{-1}\leq \rho_n\leq \log (k)/k$), the error arising from the probability matrix $\bTheta_0$ estimation dominates and the minimax risk is of order $\sqrt{\rho_n/n}$.

\medskip

As in the previous section, we left aside  the specific case of constant graphons $\cW^{+}[1]$. Note that for a graphon $W_0\in \cW^{+}[1]$ the agnostic error is always zero and the loss comes from the probability matrix estimation. Following the arguments of the previous section, we derive that the graphon $\widetilde{f}_{\overline{\bA}}$ converges to $\rho_n W_0$ at the rate $\sqrt{\rho_n}/n$ which is optimal as soon as $\rho_n\geq 1/n^2$.

\subsection{Comparison with $\delta_1$ and $\delta_2$-estimation}

Minimax risk for graphon estimation in the $\delta_{2}$-distance was obtained in \cite[Proposition 3.2]{klopp_graphon} : 
\begin{equation}\label{eq:minimax_rate_l2}
\inf_{\widehat{f}}\sup_{W_0\in \cW^+[k]}\E_{W_0}\left[\delta_{2}\left(\widehat{f}, f\right)\right]\asymp \min\left (\frac{\sqrt{\rho_n}k}{n}+  \sqrt{\frac{\rho_n\log(k)}{n}}+  \rho_n \left (\frac{k}{n}\right )^{1/4},\rho_n\right )
\end{equation}
The following proposition, proved in Section \ref{proof_prp:lower_minimax_graphon}, gives the minimax $\delta_1$-convergence rate:
 \begin{prp}\label{prp:lower_minimax_graphon}
 For any sequence $\rho_n>0$ and any positive integer $2\leq k\leq n$, we have
 \beq\label{eq:lower_minimax_L1}
  \inf_{\widehat{f}}\sup_{W_0\in \cW^+[k]}\E_{W_0}\left[\delta_1\left(\widehat{f}, f_0\right)\right]\geq C_1
  \min \left( 
        \dfrac{\sqrt{\rho_n}k}{n}+\sqrt{\dfrac{\rho_n}{n}}+ \rho_n \sqrt{\frac{k}{n}}, 
         \rho_n\right) \ .
  \eeq
Conversely, there exists an  estimator $\widehat{f}$  based on the restricted least-squares estimator~\eqref{eq:definition_restricted_least_squares} such that
  \beq\label{eq:minimax_risk_graphon_l1}
\sup_{W_0\in \cW^+[k]}\E_{W_0}\left[\delta_{1}\left(\widehat{f}_{\widehat{\bTheta}_{k,\rho_n}}, f_0\right)\right]
\leq C_2\min \left(
        \rho_n \sqrt{\frac{k}{n}}+\dfrac{\sqrt{\rho_n}k}{n}+\sqrt{\dfrac{\rho_n\log(k)}{n}},
         \rho_n\right). 
\eeq

 \end{prp}
 The upper and lower bounds given by Proposition \ref{prp:lower_minimax_graphon} match (up to a $\sqrt{\log(k)}$ multiplicative term in one of the regimes). 
 There are  three regions in \eqref{eq:minimax_risk_graphon_l1} for $\delta_1$ graphon estimation. The first one corresponds to the case of weakly sparse graphs with $\rho_n\geq k^{-1}\vee (k/n) $. In this case, the agnostic error dominates and the optimal risk is of order  $\rho_n\sqrt{k/n}$. For moderately sparse graphs with $ n^{-1}\vee (k/n)^2\leq\rho_n\leq  k^{-1}\vee (k/n)$,   the probability matrix estimation error dominates and the minimax rate is of order  $\sqrt{\rho_n/n}+\sqrt{\rho_n}k/n$ (up to  a $\log(k)$ multiplicative term). In the case of
highly sparse graphs with $\rho_n\leq   n^{-1}\vee (k/n)^2\vee \left (\frac{k}{n}\right )^{2}$, the minimax risk is $\rho_n$ which corresponds to the risk of the null estimator $\tilde f\equiv 0$.

Let us compare the optimal convergence rates with respect to  the $\delta_1$~\eqref{eq:minimax_risk_graphon_l1}, $\delta_2$~\eqref{eq:minimax_rate_l2} and $\delta_{\square}$~\eqref{eq:lower_minimax_integrated_risk_distance_cut_prp}. Bearing in mind that $\delta_2$ dominates $\delta_1$, which in turn dominates $\delta_{\square}$, one should not be surprised that optimal rates with respect to $\delta_2$ are the slowest. When the number of steps $k$ is less than $\sqrt{n}$ or when the graph is weakly sparse ($\rho_n\geq \sqrt{k/n}$), then the $\delta_1$ and $\delta_{\square}$ optimal rates only differ by a $\log(k)$ multiplicative term. For larger $k$ and sparser graph, the optimal $\delta_1$-risk can be $k/\sqrt{n}$ larger than the $\delta_{\square}$-risk.

\bigskip 

Following the discussion in Section \ref{sec:discussion_comparison_matrix_estimation}, one may easily build graphon estimators performing well in all these three distances. For instance, 
the graphon $f_{\widehat{\bTheta}_{k,\rho_n}}$ based on the restricted-least-squares estimator is optimal  with respect to $\delta_2$ and $\delta_1$ and near optimal (up to a possible $\sqrt{\log(k)}$ loss)  with respect to $\delta_{\square}$ for $k\leq \sqrt{n}$. Besides, the graphon $f_{\widetilde \bTheta_{\lambda}}$ based on the singular value thresholding estimator is optimal with respect to $\delta_{\square}$ and achieves best known convergence rates with respect to $\delta_1$ and $\delta_2$ among polynomial time algorithms. 
\subsection{Cut distance estimation of  $L_1$ and $L_2$ graphons}\label{sec:unbounded}

Until now we have restricted our attention to graphons $W$ taking values in $[0,1]$. As argued in \cite{borgs2014p,borgs2014p2}, in this case the empirical degree distribution of a graph sampled from the corresponding $W$-random graph model~\eqref{sparse_graphon_mod} is light. This contrasts with many practical situations, where the degree distribution is heavy tailed. To circumvent this limitation,   Borgs et al~\cite{borgs2014p,borgs2014p2} introduce, for $p\geq 1$, the class $\cW_p^+$ of symmetric measurable functions $W:[0,1]^2\to \mathbb{R}^+$ such that $\int|W(x,y)|^pdxdy<\infty$. This collection $\cW_p^+$ is referred as the collection of $L_p$ graphons.
We have the inclusions $\cW^+\subset \cW_p^+ \subset \cW_{p'}^+$ for $p> p'\geq 1$. Given a graphon $W_0\in \cW_p^+$ and a sparsity parameter $1\geq \rho_n>0$, the corresponding $W$-random graph model amounts to generating a graph with $n$ vertices according to the random matrix $\bTheta_0$ sampled as follows 
\begin{equation}\label{sparse_graphon_mod_p}
      \bTheta_{ij}=\big[\rho_n W_0(\xi_i,\xi_j)\big]\wedge 1\ , \quad \forall i\neq j \, \text{ and } \bTheta_{ii}=0\ ,
\end{equation} 
where $\xi_1,\ldots, \xi_n$ are, as in \eqref{sparse_graphon_mod}, i.i.d. random variables uniformly distributed in $[0,1]$. Note that since $W_0$ is now unbounded, we have to take the minimum with $1$ in \eqref{sparse_graphon_mod_p}. We write $f'_0=(\rho_n W_0)\wedge 1$. Since $W_0$ is now allowed to be unbounded, graphs sampled according to the model~\eqref{sparse_graphon_mod_p} may have power law degree distribution~\cite{borgs2014p}. 
As in the introduction, we may extend the norms $\|.\|_{\square}$ and $\|.\|_{q}$ and the distances $\delta_{\square}$ and $\delta_{q}$ to any graphon $W_0\in \cW_p^{+}$ with $p\leq q$. Also, we write $\widetilde{\cW}_p^+$ for the quotient space of $L_p$ graphons under weak isometry.

Let us also define the collection $\cW^+_p[k]$ of $k$-steps $L_p$ graphons, that is the subsets of graphon $W\in \cW_p^+$ such that $W(x,y)=\bQ_{\phi(x),\phi(y)}$ for some $\bQ\in (\mathbb R^{+})^{k\times k}_{\text{sym}}$ and some $\phi:[0,1]\to [k]$ (note that $\cW^+_p[k]$ does not depend on $p$).
For $1\geq\mu>0$ we denote by $\cW^+_p[k,\mu]$ the subset of  $\cW^+_p[k]$ of ``balanced" step functions, that is,  $W\in \cW^+_p[k,\mu]$ if $\lambda(\phi^{-1}(a))\geq \mu/k$ for all $a\in[k]$. This means that the size of each  step is larger than $\mu/k$.

  Without lost of generality we can consider normalized graphons, that is, we assume that $\Vert W_0\Vert_{1}=1$. 
  The following proposition proved in Appendix \ref{proof_prop:unbounded} gives an oracle inequality for the risk of the empirical graphon associated to the adjacency matrix and to the singular value hard thresholding estimator: 
\begin{prp}\label{prop:unbounded} 
Let $\lambda=c\sqrt{\rho_nn}$ where $c$ is a sufficiently large numerical constant.
Given a  graphon $W_0$ and $\rho_n>0$, write $W'_0= \rho^{-1}_n[(\rho_n W_0)\wedge 1]$.
\begin{itemize}
	\item [(1)] Let $W_0\in \cW^+_1$ with $\|W_0\|_1=1$,  $\rho_n\geq 1/n$ and $1\geq\mu>0$. Then, for any positive integer $k\leq \mu n$, we have  
	\beq\label{eq:minimax_risk_graphon_biais_1}
	\E_{W_0}\left[\delta_{\cut}\left(\widetilde{f}_{\bA}, f'_0\right)\right]
	\leq 2\rho_n\inf_{W\in \cW^+_1[k,\mu]}\delta_{1}\left(W, W'_0\right)+ C \left  [
	\rho_n \sqrt{\frac{k}{\mu n}} +\sqrt{\dfrac{\rho_n}{n}} 
	\right]
	\eeq
	and 
	\beq\label{eq:minimax_risk_graphon_biais_11}
	\E_{W_0}\left[\delta_{\cut}\left(\widetilde{f}_{\widetilde\bTheta_{\lambda}}, f'_0\right)\right]
	\leq 2\rho_n\inf_{W\in \cW^+_1[k,\mu]}\delta_{1}\left(W, W'_0\right)+C \left  [
	\rho_n \sqrt{\frac{k}{\mu n}} +\sqrt{\dfrac{\rho_n}{n}}
	\right] \ .
	\eeq
	\item [(2)] Assume that $W_0\in \cW^+_2$ with $\|W_0\|_1=1$ and   $\rho_n\geq 1/n$. 
	For any positive integer $k\leq n$, we have
	\beq\label{eq:minimax_risk_graphon_biais_2}
	\E_{W_0}\left[\delta_{\cut}\left(\widetilde{f}_{\bA}, f'_0\right)\right]
	\leq 2\rho_n\inf_{W\in \cW^+_2[k]}\delta_{2}\left(W, W'_0\right)+C \left  [
	\rho_n\Vert W_0\Vert_{2} \sqrt{\frac{k}{n}}+\sqrt{\dfrac{\rho_n}{n}}
	\right],
	\eeq
	\beq\label{eq:minimax_risk_graphon_biais_22}
		\E_{W_0}\left[\delta_{\cut}\left(\widetilde{f}_{\widetilde\bTheta_{\lambda}}, f'_0\right)\right]
		\leq 2\rho_n\inf_{W\in \cW^+_2[k]}\delta_{2}\left(W, W'_0\right)+C \left  [
		\rho_n\Vert W_0\Vert_{2} \sqrt{\frac{k}{n}}+\sqrt{\dfrac{\rho_n}{n}}
		\right].
		\eeq
\end{itemize}
\end{prp}
If $W_0$ belongs to some $\cW^+_2[k]$ or to $\cW^+_1[k,\mu]$  the convergence rates given by Proposition \ref{prop:unbounded} are the same as the optimal rates for bounded graphons up to a $\log^{-1/2}(k)$ factor. We conjecture that the $\log^{-1/2}(k)$ factor should appear in Proposition \ref{prop:unbounded}. Indeed, for bounded graphons, this logarithmic terms derives from Szemer\'edi's Regularity lemma and extensions of this lemma to  $L_p$ graphons have been recently proved~\cite{borgs2014p}. Nevertheless, our arguments in the  proof of Theorem \ref{prp_agnostic_block} makes heavily use of the boundedness of the graphons. In particular, one should replace all applications of McDiarmid's inequality (Lemma \ref{lem:mac_diarmid}) by more involved concentration inequalities~\cite{MR1989444}. We leave this for future work. 

When the graphon $W_0$ is not a finite step graphon, a bias term is occurring in the risk bounds (\ref{eq:minimax_risk_graphon_biais_1}--\ref{eq:minimax_risk_graphon_biais_22}). As the estimation risk is measured in the cut-distance, one could have hoped to obtain a bias term in the cut distance also (instead of the larger $l_1$ and $l_2$ distances). It is an interesting open problem to prove whether one can obtain oracle inequalities with cut distance bias terms. Note that, for bounded graphons $W\in \cW^+$,  using Theorem \ref{prp_agnostic_block}, we can also get an oracle inequality with the  $\delta_1$ bias term and minimax optimal error term.

Upper bounds of the  cut distance risk  for $L_p$ graphons estimation  were previously obtained in \cite{borgs_chayes_cohn}  where the authors introduced the  least cut norm estimator $\widehat{f}_{LC}$. For any $L_1$ normalized graphon $W_0$ any  $\kappa\in  [\log n/n, 1]$, Borgs et al.~\cite{borgs_chayes_cohn} show in their Theorem 4.1 that this estimator $\widehat{f}_{LC}$ achieves the risk bound
\beq\label{eq:upper_borgs_LC}
\E_{W_0}\Big[\delta_{\square}(\widehat{f}_{LC},f'_0)\Big]\leq  C\Big[\rho_n \inf_{W\in \cW^+_1[\lfloor \kappa\rfloor^{-1},1/2]}\delta_{1}\left(W, W'_0\right)+ \rho_n\sqrt{\frac{\log n}{\kappa n}}+ \sqrt{\frac{\rho_{n}}{n}} \Big].
\eeq

For $L_1$ graphons, this bound is quite similar (up to an additional $\log^{1/2}(n)$ term) to those we obtained in (\ref{eq:minimax_risk_graphon_biais_1}--\ref{eq:minimax_risk_graphon_biais_11})  for the  empirical estimators $\widetilde{f}_{\bA}$ and $\widetilde{f}_{\widetilde\bTheta_{\lambda}}$. Note that the least cut norm estimator can not be computed in polynomial time contrary to the empirical graphons associated to the adjacency matrix and to the singular value hard thresholding estimator. Also, when the true graphon $W_0$ either belongs to $\cW^{+}_2$ or to $\cW^+[k]$, then the rate in \eqref{eq:upper_borgs_LC} is much slower than what has been obtained in Proposition \ref{corollary_spectral} and Theorem \ref{prp_agnostic_block}.

\appendix

\section{Proof methods}\label{sec:proof_methods}
In this section, we summarize some basic facts and fundamental results that we use in the proofs.

 \subsection{Non-symmetric kernels}\label{non_symm}

 At some point, we will need to work with non-symmetric kernels and with kernel defined on general measurable subsets of $\mathbb{R}$ . In this section we define the corresponding spaces. Let $\cX$ and $\cY$ denote two bounded measurable subsets of $\mathbb{R}$. Then, $\cW_{\cX,\cY}$ refers to the collection of bounded measurable functions $W: \cX\times \cY\rightarrow [-1,1]$. We will denote by $\cW^{+}_{\cX,\cY}$ the collection of bounded measurable and non-negative functions $W: \cX\times \cY\rightarrow [0,1]$.
 Let $\cW_{\cX,\cY}[k]$ be the collection of $k-$step kernels, that is, the subset of kernels $W\in \cW_{\cX,\cY}$ such that for some $\bQ\in \bbR^{k\times k}$ and some $\phi_1:\cX\to [k]$, $\phi_2:\cX\to [k]$, 
 \beq \label{eq:def_step_function_asym}
 W(x,y)= \bQ_{\phi_1(x),\phi_2(y)}\quad \text{  for all }(x,y)\in\cX\times \cY\ .
 \eeq
 A kernel $W$ is also said to be a $q_1\times q_2$-step function when it decomposes as in \eqref{eq:def_step_function_asym} but where $\bQ$ is a size $q_1\times q_2$ matrix, $\phi_1$ mapping $\cX$ to $[q_1]$, and $\phi_2$ mapping $\cY$ to $[q_2]$.
 The cut norm can be readily  extended to kernels $W\in \cW_{\cX,\cY}$ in the following way:
 \beq\label{eq:def_cut_norm_asym}
  \|W\|_{\square}:= \sup_{X\subset \cX,\  Y\subset \cY}\left |\int_{X\times Y}W(x,y)dxdy\right |
 \eeq
 where the supremum is taken over all measurable subsets $X$ and $Y$.

\subsection{Concentration inequalities}
In the proofs we repeatedly use Bernstein's inequality. We state it here for the readers' convenience. Let $X_1,\dots,X_N$ be independent zero-mean random variables. Suppose that $|X_i|\leq M$ almost surely, for all $i$. Then, for any $t>0$,
\begin{equation}\label{bernstein}
\bbP\left \{\sum_{i=1}^{N} X_i\geq t\right\} \leq \exp\Big[-\frac{t^2}{2 \sum_{i}\E[X_i^2]+ 2Mt/3}\Big]\ .
\end{equation}
We shall also rely on the bounded difference inequality (also called McDiarmid's inequality). 
\begin{lem}[Bounded difference inequality]\label{lem:mac_diarmid}
Let $X_1,\ldots, X_n$ denote $n$ independent real random variables. Assume that
$g:\mathbb{R}^n \rightarrow \mathbb{R}$ is a measurable function satisfying, for some positive constants $(c_i)_{1\leq i\leq n }$, the bounded difference condition 
\[
 |g(x_1,\ldots, x_i,\ldots, x_n)- g(x_1,\ldots, x'_i,\ldots, x_n)|\leq c_i\ ,
\]
for all $x=(x_1,\ldots, x_i,\ldots, x_n)\in \mathbb{R}^n$, $x'=(x_1,\ldots, x'_i,\ldots, x_n)\in \bbR^n$ and all $i\in [n]$. Then, the random variable $Z=g(X_1,\ldots, X_n)$ satisfies 
\[
 \P[Z\geq \E[Z]+ t]\leq \exp\left [- \frac{2t^2}{\sum_{i=1}^n c_i^2}\right ]\ , 
\]
for all $t>0$.

\end{lem}

\subsection{Fano's lemma}

In the sequel, $\mathcal{KL}(.,.)$ denotes the Kullback-Leibler divergence between two distributions. In this manuscript, all the proofs of  the minimax lower bounds  rely on Fano's method. The following version of Fano's lemma is borrowed from~\cite{tsybakov_book}: 
\begin{lem}\cite[Theorem 2.7]{tsybakov_book}\label{lem:Fano_tsybakov}
Consider a parametric model $\P_{\theta}$, with $\theta\in \Theta$ and a metric $d(.,.)$ on $\Theta$. Assume that $\Theta$ contains elements $\theta_1, \ldots, \theta_M$, $M\geq 3$, such that for all $j,k\in [M]$ with $j\neq k$
\begin{enumerate}
 \item[(i)] $d(\theta_j,\theta_k)\geq s>0\  ,$
 \item[(ii)]  $\mathcal{KL}(\mathbb{P}_{\theta_j},\P_{\theta_k})\leq \log(M)/32\ $.
\end{enumerate}
Then, we have
\[
 \inf_{\widehat{\theta}}\sup_{\theta\in \Theta}\E_{\theta}\big[d(\widehat{\theta},\theta)\big]\geq C s\ ,
\]
where the constant $C>0$ is numeric. 
\end{lem}

\subsection{Khintchine's inequality}
Next, we state a particular case of Khintchine's inequality that turns out to be useful for bounding  the cut norm of step kernels in terms of their $l_1$ norm:
\begin{lem}\label{lem:khintchine}~\cite{szarek}
 Let $\epsilon_1,\ldots ,\epsilon_p$ be i.i.d. Rademacher random variables and let $x_1,\ldots, x_p$ be some real numbers. Then, 
 \beq\label{eq:khintchine}
\E\left [\left |\sum_{i=1}^p \epsilon_i x_i\right |\right ] \geq \frac{1}{\sqrt{2}}\left [\sum_{i=1}^p x_i^2\right ]^{1/2}\ . 
 \eeq
\end{lem}

We use this result to prove the following  lower bound on the cut norm of step kernels:
\begin{lem}\label{lem:1cut_l1}
 Let $U:\cX\times \cY\mapsto [-1,1]$ denote a measurable $q_1\times q_2$--step function. Then,  
 \beq\label{eq:lower1_cut_l1}
 \|U\|_{\square}\geq \frac{1}{4\sqrt{2q_2}}\|U\|_1\ . 
 \eeq
\end{lem}

\begin{proof}[Proof of Lemma \ref{lem:1cut_l1}]
There exist partitions $\mathcal{X}= \mathcal{X}_1\cup \ldots \mathcal{X}_{q_1}$ and $\mathcal{Y}= \mathcal{Y}_1\cup \ldots \mathcal{Y}_{q_2}$ such that, for any  fixed $y \in \mathcal{Y}$, $U(x,y)$ is constant over  $x\in\mathcal{X}_i$ for all $i\in[q_1]$ and,  for any fixed $y \in \mathcal{X}$, $U(x,y)$ is constant over  $y\in\mathcal{Y}_i$ for all $i\in[q_2]$. For any $a\in [q_1]$ (resp. $b\in [q_2]$), denote $x_a$ (resp. $y_b$) any element of $\mathcal{X}_a$ (resp. $\mathcal{Y}_b$). By definition of $\|U\|_{\square}$,
\beqn 
 \|U\|_{\square}&=&\sup_{S\subset \cX,T\subset\cY}\Big|\int_{S\times T}U(x,y)dxdy\Big|\\
 & =& \sup_{S\subset \cX,T\subset\cY} \Big|\sum_{a=1}^{q_1}\sum_{b=1}^{q_2} \lambda(S\cap \cX_a)\lambda(T\cap \cY_b)  U(x_a,y_b) \Big|\\
 & = & \sup_{\epsilon\in [0,1]^{q_1}}\sup_{\epsilon'\in [0,1]^{q_2}}\Big|\sum_{a=1}^{q_1}\sum_{b=1}^{q_2} \epsilon_a \lambda(\cX_a)\epsilon'_b\lambda(\cY_b)  U(x_a,y_b)\Big| \ ,
\eeqn 
where we used in the last line that the value of the sum only depends on $S$ and $T$ through the quantities $\lambda(S\cap \cX_a)$ and $\lambda(T\cap \cY_b)$. 
Since the maximum of a linear function on a convex set is achieved at an extremal point, it follows that 
\beqn 
\|U\|_{\square}&=&\sup_{\epsilon\in \{0,1\}^{q_1},\ \epsilon'\in \{0,1\}^{q_2}}\Big|\sum_{a=1}^{q_1}\sum_{b=1}^{q_2} \epsilon_a \lambda(\cX_a)\epsilon'_b\lambda(\cY_b)  U(x_a,y_b)\Big|  \\
&\geq & \frac{1}{4} \sup_{\epsilon\in\{-1,1\}^{q_1}, \epsilon'\in \{-1,1\}^{q_2}}\Big|\sum_{a\in [q_1] , b\in [q_2]}\epsilon_a \epsilon'_b\lambda(\mathcal{X}_a) \lambda(\mathcal{Y}_b) U(x_a,y_b)\Big|\\
&\geq & \frac{1}{4} \sup_{ \epsilon'\in \{-1,1\}^{q_2}}\sum_{a\in [q_1]} \lambda(\mathcal{X}_a)\left |\sum_{ b\in [q_2]} \epsilon'_b \lambda(\mathcal{Y}_b)U(x_a,y_b)\right |
\eeqn 
  where we use \eqref{eq:cut_operator} and take $\epsilon_{a}=\sign \sum_{ b\in [q_2]} \epsilon'_b \lambda(\mathcal{Y}_b)U[x_a,y_b]$. Let $v=(v_1,\ldots, v_{q_2})$ denote i.i.d.\  Rademacher random variables and let $\E_{v}[.]$ denotes the expectation with respect to $v$. 
  Now,  Khintchine's inequality~\eqref{eq:khintchine}  and  Cauchy-Schwarz inequality 
imply
\beqn 
\sup_{ \epsilon'\in \{-1,1\}^{q_2}}\sum_{a\in [q_1]} \lambda(\mathcal{X}_a)\Big|\sum_{ b\in [q_2]} \epsilon_b \lambda(\mathcal{Y}_b)U(x_a,y_b)\Big|&\geq& 
\E_{v}\Big[ \sum_{a\in [q_1]} \lambda(\mathcal{X}_a)\Big|\sum_{ b\in [q_2]} v_b \lambda(\mathcal{Y}_b)U(x_a,y_b)\Big|\Big]
\\
&\geq& \frac{1}{\sqrt{2}} \sum_{a\in [q_1]} \lambda(\mathcal{X}_a) \left(\sum_{b\in [q_2]} \lambda^2(\mathcal{Y}_b)U^2(x_a,y_b)\right)^{1/2}\ \\
&\geq & \frac{1}{\sqrt{2q_2}}\sum_{a\in [q_1]}\sum_{b\in [q_2]} \lambda(\mathcal{X}_a) \lambda (\mathcal{Y}_b)  \left |U (x_a,y_b)\right | \\
& & =\frac{1}{\sqrt{2q_2}} \|U\|_1 \ . 
\eeqn 
\end{proof}

\section{Proof of Proposition \ref{lemma_LovaszBook}}\label{proof_lemma_LovaszBook}

Since the diagonals of $\bA$ and $\bTheta$ are both zero, it suffices to control the supremum over disjoints subsets $S$ and $T$ (see, e.g., \cite{borgs_chayes_2008})
$$\Vert \bA-\bTheta_0\Vert_{\cut}\leq \frac{4}{n^{2}}\underset{S\cap T=\emptyset}{\max} \Big\vert \sum_{i\in S,j\in T}\left (\bA_{ij}-\bTheta_{ij}\right )\Big\vert\ .$$
Let $S$ and $T$ be any two disjoint subsets of $[n]$. Using Bernstein's inequality \eqref{bernstein} we have that 
\begin{align*}
\P\left\{\left\vert \sum_{i\in S,j\in T}(\bA_{ij}-\bTheta_{ij})\right\vert\geq 3\sqrt{\left (\Vert \bTheta_0\Vert_{1}+n\right )n}\right\}&\leq 2\exp\left(-\frac{9\left (\Vert \bTheta_0\Vert_{1}+n\right )n}{2\Vert \bTheta_0\Vert_{1}+2\sqrt{\left (\Vert \bTheta_0\Vert_{1}+n\right )n}}\right)
	\\&\leq
2\exp\left(-\frac{9}{4}n\right)
\end{align*}
 Now, using that the number of disjoint pairs $(S,T)$ is $3^n$ and the union bound, we get that the probability that $|\sum_{i\in S,j\in T}\bA_{ij}-\bTheta_{ij}|$ exceeds $3\sqrt{\left (\Vert \bTheta_0\Vert_{1}+n\right )n}$ for some $(S,T)$ is bounded by $2\exp(-n)$. Hence, we have 
\begin{align*}
\Vert A-\bTheta_0\Vert_{\cut}\leq 4\underset{S\cap T=\emptyset}{\sup}\frac{1}{n^{2}}\left\vert \underset{(i,j)\in S\times T}{\sum}(\bA_{ij}-\bTheta_{ij})\right\vert \leq 12\sqrt{\dfrac{\Vert \bTheta_0\Vert_{1}+n}{n^{3}}}
\end{align*}
with probability $1-2e^{-n}$. Now bounding the distance by $1$ in the exceptional case we get the statement of Proposition \ref{lemma_LovaszBook}.
\section{Proof of Proposition \ref{lemma_low-bound_prob}}\label{proof_low-bound_dense}
Fix $\rho_n\in (0,1)$. 
This proof is based on Fano's method.  To apply Fano's Lemma (Lemma \ref{lem:Fano_tsybakov}), it is enough to check that there exists a finite subset $\Omega$ of $\cT[2,\rho_{n}]$  such that for any two distinct $\bTheta,\bTheta'$ in $\Omega$ we have
    \begin{itemize}
    \item[(a)]$\|  \bTheta-\bTheta'\|_{\cut}\geq C\,\sqrt{\rho_{n}}\left(\frac{1}{\sqrt{n}}\wedge \sqrt{\rho_n}\right)$ and 
    \item[(b)]$\mathcal{KL}(\mathbb{P}_{\bTheta},\mathbb{P}_{\bTheta'})\leq  \log(|\Omega|)/32\,$
    \end{itemize}
    for some constants $C>0$.
     Then, Applying  Lemma \ref{lem:Fano_tsybakov} to $\Omega$ leads to the desired result. It remains to prove the existence of $\Omega$. 
As it is classical for this kind of proof, we first build a collection $\Omega'\subset \cT[2,\rho_n]$ and then extract a maximal subset $\Omega\subset \Omega'$ satisfying (a).  Then, we control the Kullback divergence between any two probability to show (b).

    \medskip

    \noindent \underline{Construction of $\Omega'$}. Fix $\epsilon \in (0,\rho_n/4)$. 
For any $u\in \{-1,1\}^n$, define $\bTheta_u$ by $(\bTheta_u)_{i,j}= \rho_{n}/2+u(i)u(j)\epsilon$ where  $u=\left (u(1),\dots,u(n)\right )$. In other words, the entries $\bTheta_u$ are equal to $\rho_n/2+\epsilon$ if $u(i)u(j)=1$ and  $\rho_n/2-\epsilon$ if  $u(i)u(j)=-1$.  Obviously, the collection  $\Omega':= \left \{\bTheta_u\; :\; u\in \{-1,1\}^n\right \}$ is included in $\cT[2,\rho_n]$.

\medskip 

\noindent \underline{Computation of the cut distances and extraction of a maximal subset}. Given $u\in \{-1,1\}^n$, denote $V_{u}:=\{i\in[n]\;:\;u(i)=1\}$ 
the set of indices corresponding to $u(i)=1$ and  $\bar{V}_u$ its complement.  Then, given two vector $u$ and $v$, we  define $S:=V_{u}\setminus V_{v}$  and $T :=  V_{v}\cap V_{u}$, we easily obtain 
\[\Big\vert\sum_{i\in S,j\in T} (\bTheta_u-\bTheta_v)_{ij} \Big\vert = 2\epsilon\vert V_{u}\setminus V_{v}\vert \vert V_{v}\cap V_{u}\ \vert . \]
By symmetry, we derive that 
\begin{align*}
n^{2}\Vert \bTheta_u-\bTheta_v\Vert_\cut&\geq  2\epsilon\max\{\vert V_{u}\setminus V_{v}\vert,\vert V_{v}\setminus V_{u}\vert \}\max\{\vert \bar V_{u}\cap \bar V_{v}\vert,\vert V_{v}\cap V_{u}\vert \}\\&\hskip 0.5 cm\geq \frac{\epsilon}{2}\vert V_{u}\triangle V_{v}\vert (n-\vert V_{u}\triangle V_{v}\vert)\ ,
\end{align*}
where $A\triangle B$ is the symmetric difference of $A$ and $B$. As a consequence, the cut distance between any two graphons is large as long as the symmetric difference between $u$ and $v$ is both bounded away from zero and from $n$.

By Varshamov-Gilbert combinatorial bound (see, e.g., \cite[Lemma 2.9]{tsybakov_book}), we can in fact pick $u_1,\dots,u_N$ satisfying
$$\frac{n}{4}\leq \vert V_{u_i}\triangle V_{u_j} \vert\leq \frac{3n}{4}\quad \text{for}\quad i\not=j\in[N]$$ 
with $N\geq \exp(c_1n)$ for some $c_1>0$. In the sequel, we consider $\Omega=\{\bTheta_{u_i}\; :\;i=1,\dots,N\}$. Hence, we have $\log\vert \Omega\vert\geq c_1n$, whereas the previous inequalities ensure that 
\begin{equation*} 
\Vert \bTheta_{u_i}-\bTheta_{u_j}\Vert_\cut\geq \epsilon/14 \ . \end{equation*}
which proves (a) when one takes $\epsilon$ as defined in \eqref{eq:epsilon_choice} below. 

\medskip

\noindent 
\underline{Control of the Kullback Divergence}. 
To prove (b) we use the definition of
 Kullback-Leibler divergence $\mathcal{KL}(\P_{\bTheta_{u}},\P_{\bTheta_{v}})$ and $\log x\leq x-1$ for $x>0$ to get
 \begin{align*}
 \mathcal{KL}(\P_{\bTheta_{u}},\P_{\bTheta_{v}})&=\sum_{ij}(\bTheta_u)_{i,j}\log\left (\dfrac{(\bTheta_u)_{i,j}}{(\bTheta_v)_{i,j}}\right )+\left (1-(\bTheta_u)_{i,j}\right )\log\left (\dfrac{(1-\bTheta_u)_{i,j}}{1-(\bTheta_v)_{i,j}}\right )\\
 &\hskip 0.5 cm \leq \sum_{ij}\dfrac{\left ((\bTheta_u)_{i,j}-(\bTheta_v)_{i,j}\right )^{2}}{(\bTheta_v)_{i,j}\left (1-(\bTheta_v)_{i,j}\right )}.
 \end{align*}
Now,  $(\bTheta_v)_{i,j}\geq \rho_n/4$ and $\rho_n\leq 1$ imply
  \begin{align*}
  \mathcal{KL}(\P_{\bTheta_{u_i}},\P_{\bTheta_{u_j}})\leq \dfrac{16}{3\rho_n}\sum_{ij}\left ((\bTheta_u)_{i,j}-(\bTheta_v)_{i,j}\right )^{2}\leq  \dfrac{16n^{2}\epsilon^{2}}{3\rho_n} .
  \end{align*}
 Taking
 \beq \label{eq:epsilon_choice}
 \epsilon=c_2\sqrt{\rho_{n}}\left(\frac{1}{\sqrt{n}}\wedge \sqrt{\rho_n}\right)
 \eeq
 with a constant $c_2>0$ small enough, we derive from the lower bound $\log(|\Omega|)\geq c_1 n$ that 
 \begin{equation*} 
 \mathcal{KL}(\P_{\bTheta_{u_i}},\P_{\bTheta_{u_j}})\leq \log \vert \Omega\vert/32
 \end{equation*}
which proves (b).

\section{Proof of Proposition \ref{corollary_spectral}}\label{proof_spectral}
 Set $\bE=\bA-\bTheta_0$. We have the following simple proposition (see Theorem 5 in \cite{Klopp_rank})
     \begin{prp}\label{prop_spectral}
     If $\lambda\geq \Vert \bE\Vert_{2\rightarrow 2}$, then 
      \begin{equation*}
      \Vert \widetilde \bTheta_{\lambda}-\bTheta_0\Vert_{2\rightarrow 2}\leq 2\lambda.
      \end{equation*}
     \end{prp}
     In view of Proposition \ref{prop_spectral}   we need to estimate $\Vert \bE\Vert$ with high probability in order to specify the value of the regularization parameter $\lambda$. Let $\bE^{*}=(\bE^{*}_{ij})$ be such that $\bE^{*}_{ij}=\bE_{ij}$ for $i<j$ and $\bE^{*}_{ij}=0$ for $i\geq j$. Then $\Vert \bE\Vert_{2\rightarrow 2}\leq 2\Vert \bE^{*} \Vert$.
      We can upper bound $\Vert \bE^{*} \Vert$ using the following bound on the spectral norm of random matrices from~\cite{Bandeira}:
        \begin{prp}\label{pr1}
        Let $\bW$ be the $n\times m$ rectangular matrix whose entries $\bW_{ij}$ are independent centered random variables bounded (in absolute value) by some $\sigma_*>0$. Then, for any $0<\epsilon\leq 1/2$ there exists a universal constant $c_{\epsilon}$ such that, for every $t\geq 0$
        $$ \mathbb P\left \{ \left \Vert \bW\right\Vert_{2\rightarrow 2}\geq (1+\epsilon)2\sqrt{2}(\sigma_1\vee \sigma_2)+t\right \}\leq (n\wedge m)\exp\left (\frac{-t^{2}}{c_{\epsilon}\sigma^{2}_*}\right )$$ 
        where we have defined 
        $$\sigma_{1}=\underset{i}{\max}\sqrt{\sum_{j}\mathbb \E[\bW_{ij}^{2}]},\quad\sigma_{2}=\underset{j}{\max}\sqrt{\sum_{i}\mathbb \E[\bW_{ij}^{2}]}\ .$$
           \end{prp}
       For $\bE^{*}$, we have
           $\sigma_1\leq \sqrt{\rho_nn}$, $\sigma_2\leq \sqrt{\rho_nn}$, and $\sigma_{*}\leq 1$.
        Taking $\epsilon=1/2$ and $t=\sqrt{2c_{\epsilon}\log(n)}$  in Proposition \ref{pr1}, we obtain that there exists absolute constants $c^{*}$ such that  
        \beq \label{eq:upper_E*}
        \left \Vert \bE\right \Vert_{2\rightarrow 2}\leq 2\left \Vert \bE^{*}\right \Vert_{2\rightarrow 2}\leq 6\sqrt{2\rho_nn}+2c^{*}\sqrt{\log(n)}\ , 
        \eeq
        with probability at least $1-1/n$. Since $\rho_n\geq \log(n)/n$, we can take $\lambda=c\sqrt{\rho_nn}$ where $c\geq 12\sqrt{2}+4c^*$ so that $\left \Vert \bE\right \Vert_{2\rightarrow 2}\leq \lambda/2$.
    Then,   Proposition \ref{prop_spectral} implies
   \begin{equation*}
 \Vert \widetilde \bTheta_{\lambda}-\bTheta_0\Vert_{2\rightarrow 2}\leq C\sqrt{\rho_nn}\ .
 \end{equation*}
 It is easy to see that the cut-norm of a matrix can be bounded by its spectral norm:
  \begin{equation*}
      \Vert \bA\Vert_{\cut}\leq \dfrac{1}{n}\Vert \bA\Vert_{2\rightarrow 2}.
      \end{equation*}
Bound on the cut-norm \eqref{cut_bound} then follows from 
    \begin{equation*}
     \Vert \widetilde \bTheta_{\lambda}-\bTheta_0\Vert_{\cut}\leq \dfrac{1}{n}\Vert \widetilde \bTheta_{\lambda}-\bTheta_0\Vert_{2\rightarrow 2}\leq C\sqrt{\dfrac{\rho_n}{n}}.
     \end{equation*}
     In order to prove the Frobenius bound~\eqref{frobenius_bound}, we use the argument from  \cite{Klopp_rank}: we can equivalently write the singular value hard thresholding estimator as the solution to the following optimization problem:
     \begin{equation*} 
     \widetilde{\bTheta}_{\lambda} \in\underset{\bTheta\in \mathbb{R}^{n\times n}}{\argmin}\big\{\parallel \bA-\bTheta \parallel _{2}^{2}+\lambda^{2} \rank (\bTheta)\big\}
     \end{equation*}
     which implies that, with probability larger than $1-1/n$, 
     \begin{equation*}
     \begin{split}
     \Vert \widetilde \bTheta_{\lambda}-\bTheta_0\Vert^{2}_{2}&\leq 2\big \vert\langle\bE,\widetilde\bTheta_{\lambda}-\bTheta_0\rangle\big\vert+\lambda^{2} \rank (\bTheta_0)-\lambda^{2} \rank (\widetilde\bTheta_{\lambda})\\ 
     &\leq 2\left \Vert\bE\right \Vert_{2\rightarrow 2}\left \Vert\widetilde\bTheta_{\lambda}-\bTheta_0\right \Vert_{2}\sqrt{\rank(\widetilde\bTheta_{\lambda}-\bTheta_0)}+\lambda^{2} \rank (\bTheta_0)-\lambda^{2} \rank (\widetilde\bTheta_{\lambda})
     \\ 
     &\leq  \dfrac{1}{2} \Vert \widetilde \bTheta_{\lambda}-\bTheta_0\Vert^{2}_{2}+ 2\left \Vert\bE\right \Vert^{2}_{2\rightarrow 2}
     	\left (\rank(\widetilde\bTheta_{\lambda})+\rank(\bTheta_0)\right )
     +\lambda^{2} \rank (\bTheta_0)-\lambda^{2} \rank (\widetilde\bTheta_{\lambda})
     \\&\leq \dfrac{1}{2} \Vert \widetilde \bTheta_{\lambda}-\bTheta_0\Vert^{2}_{2}+2\lambda^{2}\rank (\bTheta_0)\ , 
     \end{split}
     \end{equation*}
     where we used in the last line that $\|\bE\|_{2\rightarrow 2}\leq \lambda/2$. 
     Since $\rank (\bTheta_0)\leq k$, we have proved \eqref{frobenius_bound}.

\section{Proof of Theorem \ref{prp_agnostic_block}} \label{proof_agnostic_cut}

Note that both $f_0=\rho_n W_0$ and $\tilde{f}_{\bTheta_0}$ are proportional to $\rho_n$, so without loss of generality we can assume  that $\rho_n=1$.
 For $k\geq n/2$,  the result is a straightforward consequence of the second Sampling Lemma for Graphons of \cite{LovaszBook} stated in Proposition \ref{sampling_lemma_lovasz}.
 Given any graphon $W_0\in \cW^+[k]$, one can always divide some of the steps into smaller steps   in such a way that $W_0$ is a $2k$--step graphon whose weights are all less than or equal to $1/k$. Thus, we only need to prove the results for all  graphons $W_0\in \cW^+[k]$ with $32\leq k\leq n$ and such that its weights are all smaller or equal to $2/k$.

  Let  ${\bTheta}_0'$ be the matrix with entries $({\bTheta}_0')_{ij}=W(\xi_i,\xi_j)$ for all $i,j$. As opposed to $\bTheta_0$, the diagonal entries of ${\bTheta}_0'$ are not constrained to be null. By the triangle inequality, we  have
  \beq\label{eq:agnostic_decomposition}
  \E\left[\delta_{\square}\left(\widetilde{f}_{\bTheta_0} , W_0\right)\right]\leq \E\left[\delta_{\square}\left(\widetilde{f}_{\bTheta_0} , \widetilde{f}_{{\bTheta}_0'}\right)\right]+ \E\left[\delta_{\square}\left(\widetilde{f}_{{\bTheta}_0'} , W_0\right)\right].
  \eeq
  As the entries of $\bTheta_0$ coincide with those of ${\bTheta}_0'$ outside the diagonal, the difference $\widetilde{f}_{\bTheta_0}- \widetilde{f}_{{\bTheta}_0'}$ is null outside of a set of measure $1/n$. Since $\|W_0\|_{\infty}\leq 1$,  $\E[\delta_{\square}(\widetilde{f}_{\bTheta_0} , \widetilde{f}_{{\bTheta}_0'})]\leq 1/n$. Thus, we only need to prove that 
  \beq\label{eq:upper_risk_cut1}
  \E[\delta_{\square}(\widetilde{f}_{{\bTheta}_0'} , W_0)]\leq C \sqrt{\frac{k}{n\log(k)}}\ . 
  \eeq
We first need to build two suitable representations of $W_0$ and $\widetilde{f}_{{\bTheta}_0'}$ in the quotient space $\widetilde{\cW}^+$.

 As a first idea, one may want to define a representation  $\widehat{W}$ of  $\widetilde{f}_{{\bTheta}_0'}$ that matches $W_0$ on the largest possible (with respect to the Lebesgue measure)  Borel set. 
 In fact, one can match the two representations everywhere expcept on a Borel set of measure of the order of  $\sqrt{k/n}$. This turns
 out to lead to a suboptimal bound of the order of $\sqrt{k/n}$. In order to recover the correct logarithmic term, we refine the argument by showing that, for a suitable representation, the difference $\widehat{W}-W_0$, when non-zero, is well approximated in cut distance by a $\lfloor \sqrt{k}\rfloor$-step function which is zero exĉept on a Borel set of measure much smaller than $\sqrt{k/n(\log(n)}$. 
 To prepare the proof, we  carefully build the  representations of $W_0$ and $\widetilde{f}_{{\bTheta}_0'}$.

  \medskip

  \noindent 
  {\bf Step 1}: {\it Construction of a suitable representation $W$ of $W_0$ in $\widetilde{\cW}^+$.} 
  \newline In the sequel, we denote $q_1:=\lfloor \sqrt{k}\rfloor$. Here, we want to choose $W$ in such a way that a distortion of $W$ is well approximated in the cut norm by a $q_1$--step kernel. We use the following lemma which is based on a variation of Szemer\'edi's lemma.  Let $\bQ_0\in \bbR^{k\times k}_{\text{sym}}$ and $\phi_0:[0,1]\to [k]$ be associated to $W_0$ as in definition \eqref{eq:def_step_function}.

 \begin{lem}\label{lem:decomposition_szemeredy_inflated}
  There exist a permutation $\pi$ of $[k]$ and a partition $\mathcal{P}=(P_1,\ldots, P_{q_1})$ of $[k]$ made of successive intervals such that the following holds.
 Let ${\bf Q}$ be the matrix obtained from ${\bf Q}_0$ by jointly applying the permutation $\pi$ to its rows and its  columns. Denote by $\phi= \pi \circ \phi_0$, and for $a=1,\ldots, k$, $\lambda_a:= \lambda(\phi^{-1}(a))$. There are two matrices  ${\bf Q}^{(ap)}$ and  ${\bf Q}^{(ap,+)}\in [0,1]^{k\times k}$ that are $q_1$-block-constant according to the partition $\mathcal{P}$ and that satisfy
  \begin{eqnarray}\label{control_cut1}
  \sup_{\epsilon\in\{0,1\}^{k},\ \epsilon'\in\{0,1\}^k}  \Big|  \sum_{a,b=1}^k \epsilon_a \epsilon'_b \lambda_b\sqrt{\lambda_a } \big(\bQ_{ab}-\bQ^{(ap)}_{ab}\big)\Big|&\leq & C\sqrt{\frac{k}{\log(k)}}\ , 
\\
  \label{control_cut2}
  \sup_{\epsilon\in\{0,1\}^{k},\ \epsilon'\in\{0,1\}^k}  \Big|  \sum_{a,b=1}^k \epsilon_a \epsilon'_b \sqrt{\lambda_b}\sqrt{\lambda_a } \big(\bQ_{ab}-\bQ^{(ap,+)}_{ab}\big)\Big|& \leq & C \frac{k}{\sqrt{\log(k)}}\ .
  \end{eqnarray}
  \end{lem}

 According to Lemma \ref{lem:decomposition_szemeredy_inflated}, there exists two $q_1$-block constant matrices ${\bf Q}^{(ap)}$ and  ${\bf Q}^{(ap,+)}$ that approximate well  ${\bf Q}$ with respect to some weighted cut norm.  As for \eqref{control_cut1}, the weights are respectively $\lambda_b$ and $\sqrt{\lambda_a}$ whereas for \eqref{control_cut2}, the weights are $\sqrt{\lambda_a}$ and $\sqrt{\lambda_b}$. Informally, these weights arise for the following reason: writing $\widehat{\lambda}_a$ as the empirical weight of group $a$ in $\widehat{W}$ (see Step 2 for the definition), we have $\widehat{\lambda}_a-\lambda_a = O_P(\sqrt{\lambda_a/n})$.

 Invoking Lemma \ref{lem:decomposition_szemeredy_inflated},  we consider the  graphons 
  \beq\label{eq:def_W}
  W(x,y):= \bQ_{\phi(x)\phi(y)} \ ,\, \quad W_1(x,y):= \bQ^{(ap)}_{\phi(x)\phi(y)},\quad  W_1^+(x,y):=   \bQ^{(ap,+)}_{\phi(x)\phi(y)}
\ .
\eeq
  Obviously, $W$ is weakly isomorphic to $W_0$.

  \bigskip

  \noindent 
  {\bf Step 2}: {\it Construction of a suitable representation  $\widehat{W}$ of  $\widetilde{f}_{{\bTheta}_0'}$ in the quotient space $\widetilde{\cW}^+$.} \newline  Recall that $\xi_1,\ldots ,\xi_n$
 are the i.i.d. uniformly distributed random variables in the $W$-random graph model \eqref{sparse_graphon_mod} and that $\phi$ is defined in the previous step.  For $a=1,\ldots, k$, let 
  $$\widehat{\lambda}_a=\frac1n \sum_{i=1}^n \mathds{1}_{\{ \xi_i\in \phi^{-1}(a)\}}$$ 
  be the (unobserved) empirical frequency of the group $a$ corresponding to  a finer partition of $[0,1]$ given by $\phi$.
    For $l=1,\ldots, q_1$, let 
    $$\widehat{\omega}_l=\frac1n \sum_{i=1}^n\sum_{b\in P_l} \mathds{1}_{\{ \xi_i\in \phi^{-1}(b)\}}$$ 
    be the (unobserved) empirical frequency of the group $l$ corresponding to  a coarser partition $P$ of $[0,1]$ given by $\mathcal{P}\circ\phi$.

  The relations $\sum_{a=1}^k \lambda_a=\sum_{a=1}^k \hat\lambda_a=1$ imply
  \begin{equation}\label{eq:lam}
  \sum_{a: \lambda_a>\widehat{\lambda}_a} (\lambda_a-\widehat{\lambda}_a)=\sum_{a: \widehat{\lambda}_a> \lambda_a} (\hat \lambda_a-\lambda_a)\quad \text{and}\quad \sum_{l: \omega_l>\widehat{\omega}_l} (\omega_l-\widehat{\omega}_l)=\sum_{l: \widehat{\omega}_l> \omega_l} (\hat \omega_l-\omega_l).
  \end{equation}
  Consider a function $\psi:[0,1]\rightarrow [k]$  such that: 
  \begin{itemize}
   \item[(i)] For all $a\in [k]$, $\lambda(\{x,\ \psi(x)= \phi(x)= a\})=\widehat{\lambda}_a \wedge \lambda_a $ ,
   \item[(ii)] for all $l\in [q_1]$, $\lambda\Big[ \{x\ , \psi(x) \in P_l \text{ and }\phi(x)\in P_l\} \Big]= \omega_l\wedge \widehat{\omega}_l$,
   \item[(iii)] for all $a\in [k]$, $\lambda(\psi^{-1}(a))= \widehat{\lambda}_a$. 
  \end{itemize}
  Such a function $\psi$ exists.  To see it, we first construct $\psi$ to satisfy (i) and (iii):
  	\begin{itemize}
  		\item For each $a$ such that $\lambda_a>\widehat{\lambda}_a$, conditions (i) and (iii) are trivially satisfied if we take $\psi^{-1}(a)$ to be any subset of $\phi^{-1}(a)$ of Lebesgue measure $\widehat{\lambda}(a)$. Then, there is a subset of  $\phi^{-1}(a)$ of Lebesgue measures  $\lambda_a-\widehat{\lambda}_a$ left non-assigned. Summing over all such $a$, we see that there is a union of subsets  with Lebesgue measure  $m_+:=\sum_{a: \lambda_a>\widehat{\lambda}_a} (\lambda_a-\widehat{\lambda}_a)$ left non-assigned.
  		\item For $a$ such that $\lambda_a<\widehat{\lambda}_a$, we must have $\psi(x)=a$ for $x\in \phi^{-1}(a)$ to satisfy (i). On the other hand, to meet condition (iii) we need additionally to assign $\psi(x)=a$ for $x$ on a set of Lebesgue measure $\hat \lambda_a-\lambda_a$.  Summing over all such $a$, we need additionally to find a set of Lebesgue measure  $m_-:=\sum_{a: \widehat{\lambda}_a> \lambda_a} (\lambda_a-\widehat{\lambda}_a)$ to make such assignments. But this set is readily available as the union of non-assigned intervals for all $a$ such that $\lambda_a>\widehat{\lambda}_a$ since $m_+=m_-$ by virtue of \eqref{eq:lam}.
  	\end{itemize}
  	     
  Now, to ensure that condition (ii) is satisfied, we assign as a priority $\psi(x)$ to values belonging to the same partition element  as $\phi(x)$. Again, \eqref{eq:lam} ensures that this is possible. 
  
  Finally, define the graphons $\widehat{W}(x,y)= \bQ_{\psi(x),\psi(y)}$, $\widehat{W}_1(x,y)= \bQ^{(ap)}_{\psi(x),\psi(y)}$, and $\widehat{W}_1^+(x,y)= \bQ^{(ap,+)}_{\psi(x),\psi(y)}$  where $\bQ$, $\bQ^{(ap)}$, and $\bQ^{(ap,+)}$ are as in \eqref{eq:def_W}. 
   Notice that in view of (iii) $\widehat{W}$ is weakly isomorphic to the empirical graphon $\widetilde{f}_{\bTheta'_0}$. Let $\mathcal{R}= \{x\ , \phi(x)\neq \psi(x)\}$. 
    Since $W$ and $\widehat{W}$ match on  $\mathcal{R}^c\times \mathcal{R}^c$, the purpose of (i) is to minimize the  Lebesgue measure of the support of $W-\widehat{W}$. With properties (i) and (iii) alone, it would be possible to prove that $\E[\|W-\widehat{W}\|_{\square}]\leq C \sqrt{k/n}$ as the Lebesgue measure of its support is at most of order $\sqrt{k/n}$. We will improve this rate by a logarithmic term as (ii) will enforce that the cut norm of $W-\widehat{W}$ is much smaller than its Lebesgue measure.

  \bigskip
  
  \noindent 
  {\bf Step 3}: {\it Control of the cut norm.} 
  Since $\delta_{\square}(\cdot,\cdot)$ is a metric on the quotient space $\widetilde{\cW}^+$, 
    \beqn 
   \delta_{\square}(W_0,\widetilde{f}_{{\bTheta}_0'})&\leq&\left\|W - \widehat{W}\right\|_{\square}=\sup_{S, T} \left|\int_{S\times T}(W(x,y) - \widehat{W}(x,y))dxdy\right|.
   \eeqn 
  By definition of $\psi$, the two functions $W(x,y)$ and $\widehat{W}(x,y)$ are equal except possibly when either $x$ or $y$ belongs to $\mathcal{R}$.  As a consequence of triangular inequality and of the symmetry of $W-\widehat{W}$, we get 
    \begin{eqnarray} \nonumber
   \left\|W - \widehat{W}\right\|_{\square}
    &\leq & 2\sup_{S\subset \mathcal{R}, T\subset \mathcal{R}^c} \left|\int_{S\times T} \left(W(x,y)- \widehat{W}(x,y)\right)dxdy\right|\\& &+ \sup_{S,T \subset \mathcal{R}} \left|\int_{S\times T} \left(W(x,y)- \widehat{W}(x,y)\right)dxdy\right|\nonumber \\
    & &=  2 \left\| (W-\widehat{W})\Big|_{\mathcal{R}\times \mathcal{R}^c}\right\|_{\square}+\left\| (W-\widehat{W})\Big|_{\mathcal{R}\times \mathcal{R}}\right\|_{\square}  \ . \label{eq:decomposition_premier_niveau}  
    \end{eqnarray}
  First, we focus on $\mathbb{E}[\| (W-\widehat{W})|_{\mathcal{R}\times \mathcal{R}^c}\|_{\square}]$, the second term being handled similarly at the end of the proof. For $a$ and $b$ in $[k]$, we write $a\sim_{P} b$ (resp. $a \nsim_{P} b$) when $a$ and $b$ belongs  (resp. do not belong) to the same element of the partition $P$. Define
  \[\mathcal{R}_2:= \{x, \, \psi(x)\nsim_{P}\phi(x)\}\ .\] 
  Obviously, we have $\mathcal{R}_2\subset \mathcal{R}$.  Property (ii) of $\psi$, implies that $\lambda(\mathcal{R}_2)=\sum_{a=1}^{q_1}(\omega_a-\widehat{\omega}_a)_+$.
  We shall rely on the decomposition $W= W_1 + (W-W_1)$ and $\widehat{W}= \widehat{W}_1 + (\widehat{W} - \widehat{W}_1)$.  For any $x\in \mathcal{R}\setminus \mathcal{R}_2$, we have by definition \eqref{eq:def_W} of $W_1$ that $(W_1-\widehat{W}_1)(x,y)=0$. Together with the triangular inequality, this yields
  \beq\label{eq:decomposition_W_2niveau}
   \left\| (W-\widehat{W})|_{\mathcal{R}\times \mathcal{R}^c}\right\|_{\square}\leq\left \| (W_1-\widehat{W}_1)\Big|_{\mathcal{R}_2\times \mathcal{R}^c}\right\|_{\square}+ \left\|(W-W_1)\Big|_{\mathcal{R}\times \mathcal{R}^c}\right\|_{\square} + \left\|(\widehat{W}-\widehat{W}_1)\Big|_{\mathcal{R}\times \mathcal{R}^c}\right\|_{\square}.
  \eeq 
  To control the first expression in the rhs, we simply bound the cut norm of the difference by its $l_1$ norm 
  \[
  \left\| (W_1-\widehat{W}_1)\Big|_{\mathcal{R}_2\times \mathcal{R}^c}\right\|_{\square}\leq \left\| (W_1-\widehat{W}_1)\Big|_{\mathcal{R}_2\times \mathcal{R}^c}\right\|_{1}\leq \lambda(\mathcal{R}_2)\left\|W_1-\widehat{W}_1\right\|_{\infty}\leq  \lambda(\mathcal{R}_2)\ ,
  \]
  since $W_1$ and $\widehat{W}_1$ take values in $[0,1]$. Then, relying on the fact that $n\widehat{\omega}_a$ is distributed as a Binomial random variable with parameters $(n,\omega_a)$ and on Cauchy-Schwarz inequality, we get $\mathbb{E}\left|\omega_a -\widehat{\omega}_a\right|\leq \sqrt{\frac{\omega_a(1-\omega_a)}{n}}$ and
  \begin{eqnarray}\nonumber
   \mathbb{E}\left[\left\| (W_1-\widehat{W}_1)\Big|_{\mathcal{R}_2\times \mathcal{R}^c}\right\|_{\square}\right]&\leq& \mathbb{E}\left[\sum_{a=1}^{q_1}\left|\omega_a -\widehat{\omega}_a\right|\right]\\
   &\leq& \sum_{a=1}^{q_1}\sqrt{\frac{\omega_a(1-\omega_a)}{n}}\leq \sqrt{\frac{q_1}{n}} \leq \frac{k^{1/4}}{\sqrt{n}} \label{eq:upper_W1-W1_hat},
  \end{eqnarray}
  where we used again Cauchy-Schwarz in the last line.  Let us turn to the second and third expressions in \eqref{eq:decomposition_W_2niveau}. To this end, we introduce a new kernel function $U$.  
 For $a=1,\ldots, k$, define $\widehat{\lambda}^{\delta}_a=|\lambda_a-\widehat{\lambda}_a|$ and  the  functions $F_{\widehat{\lambda}^{\delta}}\;:\;[k]\rightarrow \left [0,\sum_a|\lambda_a-\widehat{\lambda}_a|\right ]$ and $F_{\phi}\;:\;[k]\mapsto [0,1]$ by
  \begin{eqnarray}\nonumber
   F_{\phi}(b)&=&\sum_{a=1}^{b} \lambda_a \quad \text{and set}\quad F_{\phi}(0)=0\\
   F_{\widehat{\lambda}^{\delta}}(b)&=&\sum_{a=1}^{b} \widehat{\lambda}^{\delta}_a \quad \text{and set}\quad F_{\widehat{\lambda}^{\delta}}(0)=0\ .\label{eq:F_lambda_delta}
 \end{eqnarray}
     For any $a,b\in [k]$, set $\widehat{\Pi}_{a,b}=  [F_{\widehat{\lambda}^{\delta}}(a-1),F_{\widehat{\lambda}^{\delta}}(a) )\times [F_{\phi}(b-1), F_{\phi}(b) )$ and let $U$ be a $k\times k$ step kernel on $[0, \sum_a|\widehat{\lambda}_a-\lambda_a|]\times [0,1]$ defined by $$U(x,y):= \sum_{a,b=1}^{k} \left [{\bf Q}_{ab}- {\bf Q}^{(ap)}_{ab} \right ]\mathds{1}_{\widehat{\Pi}_{a,b}}(x,y).$$
     By definition of $\cR$ and of the function $\psi$, we have that for any $a\in [k]$,  $\lambda(\phi^{-1}(a))\cap \cR)= (\lambda_a-\widehat{\lambda})_+$ and  $\lambda(\psi^{-1}(a))\cap \cR^c)= \lambda_a\wedge \widehat{\lambda}$. As a consequence, the restriction of $(W-W_1)$ to $\mathcal{R}\times \mathcal{R}^c$ is, up to a measure preserving bijection of its rows and of its columns, equal to the restriction of $U$ to the set $(\cup_{a:\ \lambda_a>\widehat{\lambda}_a}[F_{\widehat{\lambda}^{\delta}}(a-1),F_{\widehat{\lambda}^{\delta}}(a) ))\times (\cup_{a} [F_{\phi}(a-1),F_{\phi}(a-1)+\widehat{\lambda}_a\wedge \lambda_a) $. This entails that 
  \beq\label{eq:W-W1_RRc}
  \left \|(W-W_1)\Big|_{\mathcal{R}\times \mathcal{R}^c}\right \|_{\square}\leq \left \|U\right \|_{\square}.
  \eeq
  On the other hand, for any $(x,y)\in \mathcal{R}\times \mathcal{R}^c$, 
  \[
  (\widehat{W}-\widehat{W}_1)(x,y)= {\bf Q}_{\psi(x)\psi(y)}- {\bf Q}^{(ap)}_{\psi(x)\psi(y)} = {\bf Q}_{\psi(x)\phi(y)}- {\bf Q}^{(ap)}_{\psi(x)\phi(y)}
  \]
  by the definition of $\cR$. In view of the definition of $\psi$, for any $a\in [k]$ we have 
  $\lambda(\phi^{-1}(a))\cap \cR)= (\widehat{\lambda}- \lambda_a)_+$. As a consequence, 
   the restriction of $(\widehat{W}-\widehat{W}_1)$ to $\mathcal{R}\times \mathcal{R}^c$ is, up to a measure preserving bijection of its rows and of its columns, equal to the restriction of $U$ to the set $(\cup_{a:\ \lambda_a<\widehat{\lambda}_a}[F_{\widehat{\lambda}^{\delta}}(a-1),F_{\widehat{\lambda}^{\delta}}(a) ))\times (\cup_{a} [F_{\phi}(a-1),F_{\phi}(a-1)+\widehat{\lambda}_a\wedge \lambda_a) $. This implies that $ \|(\widehat{W}-\widehat{W}_1)|_{\mathcal{R}\times \mathcal{R}^c} \|_{\square}\leq \|U \|_{\square}$. 
  Thus, we   only have to control $\mathbb{E}[\|U\|_{\square}]$.
\bigskip
  
  \noindent 
  {\bf Step 4}: {\it Control of   $\mathbb{E}[\|U\|_{\square}]$.}
   Define the sets 
  $\mathcal{B}_1:= \prod_{a=1}^{k}[0,|\widehat{\lambda}_a-\lambda_a |]$ and $\mathcal{B}_2:= \prod_{a=1}^{k}\left [0,\left |\lambda_a\right |\right ]$. Then, the cut norm of $U$ writes as
  \begin{eqnarray}
  \|U\|_{\square}&\leq& \sup_{\gamma\in  \mathcal{B}_1,  \gamma'\in \mathcal{B}_2}\left | \sum_{a,b=1}^{k} \gamma_a \gamma'_b \left ({\bf Q}_{ab}- {\bf Q}^{(ap)}_{ab}\right )\right |\nonumber\\
  &\leq & \sup_{S,T\in [k]} \left | \sum_{a\in S,b\in T}  \lambda_b |\widehat{\lambda}_a-\lambda_a|  \left ({\bf Q}_{ab}- {\bf Q}^{(ap)}_{ab}\right )\right |\ ,\label{eq:proof_prop_agnostic_1} 
  \end{eqnarray}
  since the supremum of a linear function on a convex set is achieved at an extremal point. The random variable $|\widehat{\lambda}_a-\lambda_a|$ is in expectation of the order $\sqrt{\lambda_a/n}$.  If we could replace each  $|\widehat{\lambda}_a-\lambda_a|$ by $\sqrt{\lambda_a/n}$ in \eqref{eq:proof_prop_agnostic_1}, then thanks to  \eqref{control_cut1}, we could prove that $\|U\|_{\square}$ is (up to a multiplicative constant) less than $\sqrt{k/(n\log(k))}$. Unfortunately, if we directly applied Bernstein's inequality or the bounded difference inequality to simultaneously control $|\widehat{\lambda}_a-\lambda_a|$ over all $a\in [k]$ or to simultaneously control $\sum_{a\in S,b\in T}  \lambda_b |\widehat{\lambda}_a-\lambda_a|   ({\bf Q}_{ab}- {\bf Q}^{(ap)}_{ab})$ over all $S,T\subset [k]$, we would lose at least a logarithmic factor. 
  
  To bypass this issue, we adapt Lemma 10.9 of \cite{LovaszBook}, which is a key point in the proof of sampling Lemma for graphons  (Lemma 10.5 in \cite{LovaszBook}).
  Given a bounded non-symmetric kernel $W\in \cW_{\cX,\cY}$, let us define the following one-side version of the cut norm:
 \[
 \|W\|^+_{\square} = \sup_{X\subset\cX, \ Y\subset\cY}\int_{X\times Y}W(x,y)dxdy\ , 
 \]
 where we take the supremum without any absolute value. As a consequence, the cut norm $\|W\|_{\square}$ is the maximum $\|W\|^{+}_{\square}$ and  $\|-W\|^{+}_{\square}$.

  \begin{lem}\label{lem:lovasz_q_random}
  Let $W\in \cW_{[0,u],[0,v]}[k]$  and let $\bQ\in \mathbb{R}^{k\times k}$, $\phi_1:[0,u]\to [k]$ and $\phi_2:[0,v]\to [k]$ be associated to $W$ as in \eqref{eq:def_step_function_asym}. For $a=1,\ldots, k$, define $\alpha_a:= \lambda(\phi_1^{-1}(\{a\}))$ and $\beta_a:= \lambda(\phi_2^{-1}(\{a\}))$. Given any subset $R\subset [k]$, let 
  \begin{equation}\label{def_sets}
   R^{l,W}:=\big\{b, \, \sum_{a\in R}\alpha_a \bQ_{ab}>0\big\}\ ,\quad R^{r,W}:=\big\{a, \, \sum_{b\in R} \beta_b \bQ_{ab}>0\big\}\ .
 \end{equation}
Finally, we define for any $S,T\subset [k]$, $W[S,T]:=  \sum_{a\in S,b\in T} \alpha_a\beta_b \bQ_{ab} \ . $
Then, for any integer $q$ with $1\leq q\leq k$, we have 
   \beq \label{eq:upper_cut_q_random}
    \|W\|^+_{\square}\leq \max_{R_i\subset[k], |R_i|\leq q}W\left [R_2^{r,W},R_1^{l,W}\right ]+ \frac{u \sqrt{k\sum_{a=1}^k \beta_a^2}+v \sqrt{k\sum_{a=1}^k \alpha_a^2}}{\sqrt{q}}\ . 
   \eeq
  \end{lem}
Note that in contrast to Equation \eqref{eq:proof_prop_agnostic_1} where one considers a supremum of $2^{2k}$ sums, only $k^{2q}$ terms are involved in \eqref{eq:upper_cut_q_random} up to the price of an additive term of order $q^{-1/2}$. The difficulty is that we will apply this lemma to $U$ for which these $k^{2q}$ will turn out to be random.

 In the sequel, we fix $q= \lfloor \sqrt{k}\rfloor$  and apply Lemma \ref{lem:lovasz_q_random} to $U$. Then, we can take $u=v=1$.
 Since  $\sum_{a=1}^k\lambda_a=1$ and since we assumed at the beginning of the proof that the weights  $\lambda_a$ are all smaller than $2/k$, it follows that $(k\sum_{a=1}^k\lambda_a^2)^{1/2}\leq \sqrt{2}$.
 Let $M$ and $N$ denote the random variables $M:=\sum_{a=1}^k |\widehat{\lambda}_a-\lambda_a|$ and $N:=\left (\sum_{a=1}^k k|\widehat{\lambda}_a-\lambda_a|^2\right )^{1/2}$. Both $M$ and $N$ are functions of the independent random variables  $(\xi_1,\ldots, \xi_n)$.   Besides, if we change the values of one of these $\xi'_i$ the value of $M$ changes by at most $2/n$ and the value of $N$ changes by at most $\sqrt{2k}/n$. 
 As a consequence, we may apply the bounded difference inequality (Lemma \eqref{lem:mac_diarmid}) to these two random variables. Then,  with probability larger than $1- 2\exp(-\sqrt{k}/\log(k))$, one has 
\begin{eqnarray}
\sum_{a=1}^k\left  |\widehat{\lambda}_a-\lambda_a\right |&\leq& \E\left [\sum_{a=1}^k |\widehat{\lambda}_a-\lambda_a|\right ]   + \sqrt{\frac{2 k^{1/2} }{n\log(k)}} \leq C \sqrt{\frac{k}{n}}\ ,  \label{eq:upper_lambda_a}\\
\left (k\sum_{a=1}^k |\widehat{\lambda}_a-\lambda_a|^2\right )^{1/2}&\leq& \E\left [\left (k\sum_{a=1}^k |\widehat{\lambda}_a-\lambda_a|^2\right )^{1/2}\right ]   + \sqrt{\frac{2 k^{3/2} }{n\log(k)}} \leq Ck^{1/4}\sqrt{\frac{k}{n\log(k)}} \label{eq:upper_lambda_a2}\ .
 \end{eqnarray}
 In \eqref{eq:upper_lambda_a} - \eqref{eq:upper_lambda_a2} we bound the expectation using that, since $\xi_1,\ldots ,\xi_n$
are i.i.d. uniformly distributed random variables,  $n\widehat{\lambda}_a$ has a binomial distribution with parameters ($n$, $\lambda_a$) and the Cauchy-Schwarz inequality:
\begin{equation*}
\E\left [\sum_{a=1}^k |\widehat{\lambda}_a-\lambda_a|\right ] \leq \sum_{a=1}^k\sqrt{\frac{\lambda_a(1-\lambda_a)}{n}}\leq \sqrt{\dfrac{k}{n}}\quad \text{and}
\end{equation*}
\begin{equation*}
\E\left [\left (k\sum_{a=1}^k |\widehat{\lambda}_a-\lambda_a|^2\right )^{1/2}\right ] \leq \sqrt{k\sum_{a=1}^k\frac{\lambda_a(1-\lambda_a)}{n}}\leq \sqrt{\dfrac{k}{n}}.
\end{equation*}
Bound \eqref{eq:upper_lambda_a2} and $(k\sum_{a=1}^k\lambda_a^2)^{1/2}\leq \sqrt{2}$, implies that for $U$,  with probability larger than $1- 2\exp(-\sqrt{k}/\log(k))$,
\begin{equation}\label{eq:resuming_second_bound}
\frac{\sqrt{k\sum_{a=1}^k \beta_a^2}+\sqrt{k\sum_{a=1}^k \alpha_a^2}}{k^{1/4}}\leq  C \sqrt{\frac{k}{n\log(k)}}.
\end{equation}

 Fix any two subsets $R_1, R_2\subset [k]$  of size less than or equal to $q$. In view of \eqref{eq:upper_cut_q_random}, one needs to control the following random variable 
  \begin{equation}\label{def_Z}
  Z_{R_1,R_2}:= U\left [R_2^{r,U},R_1^{l,U}\right ]= \sum_{a\in R_2^{r,U}}\left |\widehat{\lambda}_a-\lambda_a\right | \sum_{b\in R_1^{l,U}} \lambda_b (\bQ_{ab}-\bQ^{ad}_{ab})\ .
  \end{equation}
 It is done in the following Lemma:
 \begin{lem}\label{lem:control_Z}
 Let $R_1, R_2$ be two subsets of $[k]$  of size less than or equal to $q$ and $Z_{R_1,R_2}$ given by \eqref{def_Z}. Then, we have that  with probability larger than $1 - (1+2k)\exp(-\sqrt{k}/\log(k))$,
 \[
 \max_{R_1,R_2\,:\,|R_1|\leq q |R_2|\leq q} Z_{R_1,R_2}\leq   C \sqrt{\frac{k}{n\log(k)}}.
 \]	
\end{lem}

Now, it follows from Lemma \ref{lem:lovasz_q_random} together with \eqref{eq:resuming_second_bound} and Lemma \ref{lem:control_Z} that,  with probability larger than $1 - (3+2k)\exp(-\sqrt{k}/\log(k))$,
  \[
  \|U\|^{+}_{\square}\leq     C \sqrt{\frac{k}{n\log(k)}}. 
\]
Controlling analogously $\|-U\|^{+}_{\square}$, we conclude that there exists an event $\cA$ of probability larger than $1- 10\exp(-\sqrt{k}/\log(k))$ such that, on $\cA$, 
\[
  \|U\|_{\square}\leq C\sqrt{\frac{k}{n\log(k)}}\ . 
\]
To finish the control of $\E[\|U\|_{\square}]$, we use the rough bound $\|U\|_{\square}\leq \|U\|_1\leq \sum_{a=1}^k |\widehat{\lambda}_a-\lambda_a|$ on the complementary event $\bar{\cA}$. 
\begin{eqnarray} \nonumber
 \E[\|U\|_{\square}]&\leq & E\left [\|U\|_{\square}\1_{\cA}\right ]+ \E\left [\|U\|_{\square}\1_{\bar{\cA}}\right ]\\ \nonumber
 &\leq& C\sqrt{\frac{k}{n\log(k)}} + \sqrt{\P\left (\bar{\cA}\right )}\left [\E\left (\sum_{a=1}^k |\widehat{\lambda}_a-\lambda_a|\right )^2\right ]^{1/2}\\
 &\leq & C\sqrt{\frac{k}{n\log(k)}} + C' e^{-\sqrt{k}/(2\log(k))}\frac{k}{\sqrt{n}}\leq C'' \sqrt{\frac{k}{n\log(k)}} \label{eq:upper_U_cut}
\end{eqnarray}  
  where we use \eqref{eq:upper_lambda_a}. Now, using the decomposition  \eqref{eq:decomposition_W_2niveau}, \eqref{eq:upper_W1-W1_hat} and \eqref{eq:W-W1_RRc}, we can conclude that 
  \[
  \mathbb{E}\left [\left \| (W-\widehat{W})|_{\mathcal{R}\times \mathcal{R}^c}\right \|_{\square}\right ]\leq   C \sqrt{\frac{k}{n\log(k)}}. 
  \]
 
  The following lemma gives a corresponding bound on the second term $\left \| (W-\widehat{W})\right |_{\mathcal{R}\times \mathcal{R}}\|_{\square}$ in \eqref{eq:decomposition_premier_niveau}. The proof is somewhat analogous to that of the  control of $\left \| (W-\widehat{W})|_{\mathcal{R}\times \mathcal{R}^c}\right \|_{\square}$ 
  and is postponed to the end of the section.
    
  \begin{lem}\label{lem:second_term_decomposition}
  We have
  \[
  \mathbb{E}\left [\left \| (W-\widehat{W}) |_{\mathcal{R}\times \mathcal{R}}\right \|_{\square}\right ]\leq  C\sqrt{\frac{k}{n\log(k)}}\ .
   \]
  \end{lem}
  In view of \eqref{eq:decomposition_premier_niveau}, we have proved Theorem \ref{prp_agnostic_block}.\qed
  
 \bigskip

  \begin{proof}[Proof of Lemma \ref{lem:decomposition_szemeredy_inflated}]

For $a\in[k]$, we denote $(\lambda_0)_a= \lambda(\phi_0^{-1}(a))$ and  $u_a= \frac{\sqrt{(\lambda_0)_a}}{\sum_b \sqrt{(\lambda_0)_b}}$ . For any $b\in [k]$, define the cumulative distribution functions  $F_{0}(b)=\sum_{a=1}^{b} (\lambda_0)_a$  and $F_1(b)= \sum_{a=1}^{b} u_a$. For $a,b\in [k]$, let $(\Pi_d)_{ab}=[F_{0}(a-1),F_{0}(a))\times [F_{1}(b-1),F_{1}(b))$ and $(\Pi^+_d)_{ab}=[F_{1}(a-1),F_{1}(a))\times [F_{1}(b-1),F_{1}(b))$.   In order to construct a suitable $q_1$-step kernel we consider first  the (non necessarily symmetric) kernels  $W_d$ and $W^+_d$ defined by 
\[
W_d(x,y)= \sum_{a=1}^k \sum_{b=1}^k (\bQ_0)_{ab} \mathds{1}_{(\Pi_d)_{ab}}(x,y)\ ,\, \quad W^+_d(x,y)= \sum_{a=1}^k \sum_{b=1}^k (\bQ_0)_{ab}
 \mathds{1}_{(\Pi^+_d)_{ab}}(x,y)\ .
\]
In comparison to $W_0$, the length of the steps in  $W_d$ and $W_d^+$ has been modified. 
\begin{lem}\label{regularity_lemma}
Let $W\in \cW_{[0,1],[0,1]}$ be a k-step kernel defined by \[
W(x,y)= \sum_{a=1}^k \sum_{b=1}^k \bQ_{ab} \mathds{1}_{S_{a}\times T_{b}}(x,y)
\]
where $\bQ\in [0,1]^{k\times k}$ and $ (S_{1},\dots, S_{k})$ and  $(T_{1},\dots, T_{k})$  are two partitions of $[0,1]$ into a finite number of measurable sets. For any integer $q_0\geq 2$, 
	 there exist a $q_0$--step kernel $W^{(ap)}\in\cW^{+}_{[0,1],[0,1]}$ satisfying 
	\begin{itemize}
		\item [(i)] for any $(a,b)\in [k]$, $W^{(ap)}$ is constant on $S_{a}\times T_{b}$ and
		\item [(ii)]$\left \|W- W^{(ap)}\right \|_{\square}\leq \frac{ C}{\sqrt{\log(q_0)}}$. 
	\end{itemize}
	
\end{lem}
	
 The second property (ii) is just the consequence of
 the weak Regularity  Lemma for kernels \cite{frieze_kannan} (see also Corollary 9.13 in \cite{LovaszBook}). The first property, (i),
 follows from the explicit construction of the approximate kernel by Kannan and Frieze (see the proof of Lemma 9.10 in \cite{LovaszBook}). For the sake of completeness, we give the details in the end of this section.  

 Fix $q_0=\lfloor k^{1/4}\rfloor$. Note that $q_0\geq 2$ since we assume that $k\geq 16$. 
We denote by $W_{d}^{(ap)}$ and $W^{(ap,+)}_{d}$ the $q_{0}$--step kernels given by Lemma \ref{regularity_lemma} to respectively approximate $W_{d}$ and $W^{(+)}_{d}$. In virtue of  Property $(i)$, there exist two matrices $\bQ^{(ap)}_{0}$ and $\bQ_{0}^{(ap,+)}$ in $[0,1]^{k\times k}$ such that      
\[W_d^{(ap)}(x,y)= \sum_{a=1}^k \sum_{b=1}^k (\bQ^{(ap)}_{0})_{ab} \mathds{1}_{(\Pi_d)_{ab}}(x,y)\quad\text{ and }\quad W_d^{(ap,+)}(x,y)= \sum_{a=1}^k \sum_{b=1}^k (\bQ^{(ap,+)}_{0})_{ab}
 \mathds{1}_{(\Pi^+_d)_{ab}}(x,y)\ .\] There exist two partitions $\mathcal{P}_d$ and $\mathcal{P}_d^+$ of $[k]$ such that $\bQ^{(ap)}_{0}$ is block constant according to $\mathcal{P}_d$ and $\bQ^{(ap,+)}_{0}$ is block constant according to $\mathcal{P}^+_d$. Let $\mathcal{P}^*$ be the coarsest partition that refines both $\mathcal{P}$ and $\mathcal{P}_d^+$. As a consequence, $\cP^*$ is made of less than $q_0^2\leq q_1$ subsets. By possibly refining $\mathcal{P}^*$, we may  assume without loss of generality that $\cP^*=(P^*_1,\ldots, P^*_{q_1})$ is made of exactly $q_1$ elements. Let $\pi$ be  a permutation of $[k]$ transforming $\mathcal{P}^*$ in a partition $\cP=(P_1,\ldots,P_{q_1})$ with $P_a=\{\pi(b), b\in P^*_a\}$ made of  consecutive intervals. Denoting $\boldsymbol{\Pi}$ the corresponding permutation matrix, we finally take
 \[
  \bQ = \boldsymbol{\Pi}^T {\bf Q}_0\boldsymbol{\Pi},\quad \bQ^{(ap)}= \boldsymbol{\Pi}^T \bQ^{(ap)}_{0}\boldsymbol{\Pi},\ \text{ and } \quad \bQ^{(ap,+)}= \boldsymbol{\Pi}^T \bQ^{(ap,+)}_{0}\boldsymbol{\Pi}\ . 
 \]
Now we are ready to prove \eqref{control_cut1} and \eqref{control_cut2}. Recall that we denote $\phi=\pi \circ \phi_0 $ and  $\lambda_a:= \lambda(\phi^{-1}(a))$ for $a\in[k]$.
Define the sets 
 $\mathcal{B}_1:= \prod_{a=1}^{k}[0,u_{\pi(a)}]$ and $\mathcal{B}_2:= \prod_{a=1}^{k}[0,\lambda_a]$. Since $W_d-W_{d}^{(ap)}$ is a $k$--step function, its cut norm writes as 
\begin{eqnarray}\label{eq:reg_lemma_bound_cut_norm}
 \|W_d-W_{d}^{(ap)}\|_\square &=& \sup_{\gamma\in \mathcal{B}_1\gamma\in \mathcal{B}_2}\left |\sum_{a,b}\gamma_a \gamma_b \left (\bQ_{ab}-\bQ^{(ap)}_{ab}\right ) \right |\\
 & = & \sup_{\epsilon\in \{0,1\}^k,\ \epsilon'\in \{0,1\}^k}\left |\sum_{a,b}\epsilon_a \epsilon'_b \lambda_b u_{\pi(a)} \left (\bQ_{ab}-\bQ^{(ap)}_{ab}\right ) \right |\leq \dfrac{ C}{\sqrt{\log(q_0)}}
\end{eqnarray}
since the supremum is achieved at an extremal point of the convex and in the last inequality we use property (ii) of Lemma \ref{regularity_lemma}. Now \eqref{eq:reg_lemma_bound_cut_norm} and the definition of $u_{\pi(a)}$ imply 
\[
 \sup_{\epsilon\in \{0,1\}^k,\ \epsilon'\in \{0,1\}^k}\left |\sum_{a,b}\epsilon_a \epsilon'_b \lambda_b \sqrt{\lambda_a} \left (\bQ_{ab}-\bQ^{(ap)}_{ab}\right ) \right |\leq  C\frac{\sum_{b\in [k]}\sqrt{\lambda_b}}{\sqrt{\log(q_0)}}\leq C' \sqrt{\frac{k}{\log(k)}}\ , 
\]
by Cauchy-Schwarz inequality. We have proved  \eqref{control_cut1}. The second inequality \eqref{control_cut2} is derived similarly.

 \end{proof}

\begin{proof}[\textbf{ Proof of Lemma \ref{regularity_lemma}}]
We adapt  the proof of 
 the weak Regularity  Lemma for symmetric kernels \cite[Lemma 9.9]{LovaszBook} to non symmetric ones.  
 We use  the following extension of Lemma 9.11(a) in \cite{LovaszBook}. 
\begin {lem}\label{lemma_9.11_Lovasz}
	For every $W \in \cW_{[0,1],[0,1]}[k]$
such that 
 	\[
 	W(x,y)= \sum_{a=1}^k \sum_{b=1}^k \bQ_{ab} \mathds{1}_{S_{a}\times T_{b}}(x,y)
 	\]
 	where $\bQ\in \mathbb{R}^{k\times k}$ and $\mathcal{P}=\left \{\left (S_{1},\dots, S_{k}\right ), \left (T_{1},\dots, T_{k}\right )\right \}$  are two partitions of $[0,1]$ into a finite number of measurable sets,
 	there are two sets $\cA,\cB\subset[k]$ and a real number $0\leq a\leq \max_{a,b} |\bQ_{ab}|$ such that, for $S'=\cup_{a\in\mathcal{A} } S_a\quad \text{and }\quad  T'=\cup_{b\in\mathcal{B} }T_b$, 
 	\[ \left \Vert W- a\mathds{1}_{S'\times T'}\right \Vert^{2}_{2}\leq \Vert W\Vert^{2}_{2}-\Vert W\Vert^{2}_{\cut}\ .\]
 \end{lem}
 Now we apply Lemma \ref{lemma_9.11_Lovasz} repeatedly, to get pairs of sets $S'_i,T'_i$ and real numbers $a_i$ such that for any positive integer $j$,  $W_j=W-\sum_{i=1}^{j}a_{i}\mathds{1}_{S'_{i}\times T'_{i}}$ we have
 \[ \Vert W_j \Vert^{2}_{2}\leq \Vert W \Vert^{2}_{2}-\sum_{i=1}^{j-1} \Vert W_i \Vert^{2}_{\cut}.\]	
 Fix some integer $k_0>0$.  Since the right-hand side of the above equation remains non-negative, there exists  $0\leq i< k_0$ with $\Vert W_i \Vert^{2}_{\cut}\leq 1/k_0$. Now putting $a_{l}=0$ for $l>i$ we get that for any $W\in \cW_{[0,1],[0,1]}[k]$ and any $k_0\geq 1$ there are $k_0$ pairs of subsets $S'_{i},T'_{i}\subset[0,1]$ and $k_0$ real numbers $a_i$ such that
 \begin{equation}\label{lemma9.10_Lovasz}
 \left \Vert W-\sum_{i=1}^{k_0}a_{i}\mathds{1}_{S'_{i}\times T'_{i}}\right \Vert_{\cut}< \dfrac{1}{\sqrt{k_0}}.
 \end{equation}
 Note that the approximation $W^{ap}= \sum_{i=1}^{k_0}a_{i}\mathds{1}_{S'_{i}\times T'_{i}}$ is a step function with at most $2^{k_0}$ steps and $a_i\geq O$, for all $i$. On the other hand, by construction we have that for any $(a,b)\in [k]$, $W^{(ap)}$ is constant on all sets of the form $S_{a}\times T_{b}$. We conclude by  taking $k_0=\lfloor \log(q_0)/\log(2)\rfloor$.
 \end{proof}

\begin{proof}[\textbf{ Proof of Lemma \ref{lemma_9.11_Lovasz}}]
This lemma is proved in \cite[Lemma 9.11]{LovaszBook} for symmetric kernels. For readers convenience we get the details here.  Let $W$ be a $k$--step kernel and let $\left (S_{1},\dots, S_{k}\right ), \left (T_{1},\dots, T_{k}\right )$ be two measurable partitions of $[0,1]$ 
such that $W$ is constant on each set $S_i\times T_j$. Relying on a convexity argument as in the proof of Lemma \ref{lem:decomposition_szemeredy_inflated}, the cut norm is achieved for measurable sets $S$ and $T$ that are unions of $S_i$ and $T_j$ respectively, that is 
	\[ \Vert W\Vert_{\cut}=\left \vert \int_{S\times T} W(x,y)dxdy \right \vert\ ,\]
where $S= \cup_{a\in \cA}S_a$ and $T= \cup_{b\in \cB}T_b$ with $\cA$, $\cB\subset [k]$.
	Let $\mathbf a=\frac{1}{\lambda(S)\lambda(T)}\Vert W\Vert_{\cut}$. Then, we have 
	\[  \Vert W-\mathbf a\mathds{1}_{S\times T}\Vert^{2}_{2}= \Vert W\Vert^{2}_{2}-\dfrac{1}{\lambda(S)\lambda(T)}\Vert W\Vert^{2}_{\cut} \leq \Vert W\Vert^{2}_{2}-\Vert W\Vert^{2}_{\cut}\]
	which completes the proof.
\end{proof}

  \begin{proof}[Proof of Lemma \ref{lem:lovasz_q_random}]
  This proof closely follows  that  of Lemma 10.9 in \cite{LovaszBook}. It is easy to see that
  $$ \|W\|^+_{\square}= \max_{S,T\subset[k]}W[S,T]$$
  so we only need to bound these expressions. Let $Q$ and $Q'$ be independent uniformly chosen $q$-subset of $[k]$ and let $\E_{Q}$ (resp. $\E_{Q'}$) denote the expectation with respect to $Q$      (resp. $Q'$). We shall prove that, for any $S,T\subset[k]$
  \beq \label{eq:expect_q}
   W[S,T]\leq \E_{Q}\big[W[(Q\cap T)^{r,W}, T]\big] + \frac{u \sqrt{k\sum_{a=1}^k \beta_a^2}}{\sqrt{q}}\ . 
  \eeq 
By symmetry, this will imply  
\[
  W[S,T]\leq \E_{Q'}\big[W[S, (Q'\cap S)^{l,W}]\big] + \frac{v \sqrt{k\sum_{a=1}^k \alpha_a^2}}{\sqrt{q}}\ , 
\]
so that gathering both inequalities yields to 
 \[
  W[S,T]\leq \E_{Q,Q'}\big[W[(Q\cap T)^{r,W}, (Q'\cap (Q\cap T)^{r,W} )^{l,W}]\big]+ \frac{u \sqrt{k\sum_{a=1}^k \beta_a^2}+ v \sqrt{k\sum_{a=1}^k \alpha_a^2}}{\sqrt{q}}\ . 
  \]
 Since the above expectation is less than or equal to  $\sup_{R_i,\ |R_i|\leq q}W\left [R_2^{r,W},R_1^{l,W}\right ]$, this will conclude the proof. Thus, we only have to show \eqref{eq:expect_q}. Note that $W[S,T]\leq W[T^{r,W}, T]$ implies that it suffices to prove 
 \beq \label{eq:expect_q2}
 \E_Q\Big[W[T^{r,W}\setminus (Q\cap T)^{r,W}, T]\Big] - \E_Q\left [W[(Q\cap T)^{r,W}\setminus T^{r,W}, T]\right ] \leq  \frac{u \sqrt{k\sum_{a=1}^k \beta_a^2}}{\sqrt{q}}\ . 
 \eeq
Let us denote $Z$ the above difference of expectations. 
 For any $a\in [k]$, write $B_a= \sum_{b\in T}\beta_b \bQ_{ab}$ and $A_a= \sum_{b\in T\cap Q}\beta_b \bQ_{ab}$.  By the definition \eqref{def_sets}, we have that $B_a$ is non-negative for $a\in T^{r,W}$ and $B_a\leq 0$ if $a\not\in T^{r,W}$. In the same way, $A_{a}>0$ for $a\in (Q\cap T)^{r,W}$ and $A_{a}\leq0$ for $a\notin (Q\cap T)^{r,W}$ . Denoting $\P_Q$ the probability with respect to $Q$, we obtain
 \begin{equation}
 \begin{split}
 Z&= \E_Q \left ( \sum_{a\in T^{r,W}}\mathds{1}_{\left \{a\notin  (Q\cap T)^{r,W}\right \}}\alpha_a B_a +  \sum_{a\notin T^{r,W}}\mathds{1}_{\left \{a\in (Q\cap T)^{r,W}\right \}}\alpha_a |B_a|\right )\\
 & =\sum_{a\in T^{r,W}}\P_Q[A_a\leq 0]\alpha_a B_a +  \sum_{a\notin T^{r,W}}\P_Q[A_a>0]\alpha_a |B_a|.
 \end{split}
 \end{equation}

Now, using $\E_Q[A_a]=q B_a/k$, it follows from the Chebyshev inequality that, for $a\in T^{r,W}$, we have $\P_Q[A_a<0]\leq \Var_Q[A_a]/\E_Q^2[A_a]$. Since a probability is smaller or equal to one, it follows that $\P_Q[A_a<0]\leq \sqrt{\Var_Q[A_a]}/\vert\E_Q[A_a]\vert$. Similarly, for $a\notin T^{r,W}$ we also have that $\P_Q[A_a>0]\leq \sqrt{\Var_Q[A_a]}/|\E_Q[A_a]|$. Coming back to $Z$, this yields 
\[
 Z\leq \sum_{a\in [k]} \alpha_a |B_a|\frac{\Var^{1/2}_Q[A_a]}{\left |E_Q[A_a]\right |}=  \frac{k}{q}\sum_{a\in [k]} \alpha_a \Var^{1/2}_Q[A_a] \leq \frac{ku}{q}\max_{a\in [k]}\Var^{1/2}_Q[A_a]\ . 
\]
Working out the variance, we get $\Var_Q[A_a]\leq \tfrac{q}{k}\sum_{b\in T}\beta^2_b \bQ^2_{ab}\leq q (\sum_{b\in [k]}\beta^2_b)/k$, which concludes the proof. 
  \end{proof}

  \begin{proof}[Proof of Lemma \ref{lem:control_Z}]
  	Note that  in \eqref{def_Z}, the definition of $Z_{R_1,R_2}$, the set $R_2^{r,U}$ is deterministic whereas the set $R_1^{l,U}$ only depends on $(\widehat{\lambda}_a)_{a\in R_1}$. We can upper bound $Z_{R_1,R_2}$ in the following way:
  	 \beq
  	 Z_{R_1,R_2}\leq  \sum_{a\in R_2^{r,U}\setminus R_1}  \left |\widehat{\lambda}_a-\lambda_a\right | \sum_{b\in R_1^{l,U}} \lambda_b \left (\bQ_{ab}-\bQ^{ad}_{ab}\right )+ \sum_{a\in R_1} \left |\widehat{\lambda}_a -\lambda_a\right |
  	 \eeq
  	 where we use $\left \vert\sum_{b} \lambda_b (\bQ_{ab}-\bQ^{ad}_{ab})\right \vert\leq 1$. We set  
  	 $$T_{R_1,R_2}=\sum_{a\in R_2^{r,U}\setminus R_1}  \left |\widehat{\lambda}_a-\lambda_a\right | \sum_{b\in R_1^{l,U}} \lambda_b \left (\bQ_{ab}-\bQ^{ad}_{ab}\right ).$$
  	 Conditionally to $(\widehat{\lambda}_a)_{a\in R_1}$, $T_{R_1,R_2}$ is distributed as a function of $n-n\sum_{a\in R_1}\widehat{\lambda}_a$ i.i.d. random variables $\xi'_i$ such that $\P[\xi'=a]= \lambda_a/ (1-\sum_{a\in R_1}\lambda_a)$ for any $a\in [k]\setminus R_1$. Besides, if we change the values of one of these $\xi'_i$ the value of this expression changes by at most $2/n$.  It then follows from  the bounded difference inequality (Lemma \eqref{lem:mac_diarmid}) that, for any $t>0$
  	 \beq \label{eq:concentration_TR}
  	 \P\left \{T_{R_1,R_2}\geq \E\left [T_{R_1,R_2}|(\widehat{\lambda}_a)_{a\in R_1}\right ]+ \sqrt{\frac{2t}{n}}\Big|(\widehat{\lambda}_a)_{a\in R_1}\right \}\geq 1-e^{-t}.
  	 \eeq
  	 Let us bound this conditional expectation: 
  	 \begin{eqnarray}\nonumber
  	 \E\left [T_{R_1,R_2}|(\widehat{\lambda}_a)_{a\in R_1}\right ]  &= & \sum_{a\in R_2^{r,U}\setminus R_1}  \E\left [\left |\widehat{\lambda}_a-\lambda_a\right |\big|(\widehat{\lambda}_c)_{c\in R_1}\right ] \sum_{b\in R_1^{l,U}} \lambda_b \left (\bQ_{ab}-\bQ^{ad}_{ab}\right ) \\ \label{eq:exp_T}
  	 &\leq &  \sup_{S\subset [k]\setminus R_1, T\subset [k]}\sum_{a\in S, b\in T}\E\left [\left  |\widehat{\lambda}_a-\lambda_a\right |\big|(\widehat{\lambda}_c)_{c\in R_1}\right ] \lambda_b (\bQ_{ab}-\bQ^{ad}_{ab}) \ . 
  	 \end{eqnarray}
  	 Now, using Cauchy-Schwarz inequality, we have
  	 \[
  	 \E\left [\left |\widehat{\lambda}_a-\lambda_a\right |\big|(\widehat{\lambda}_c)_{c\in R_1}\right ]\leq \sqrt{\frac{\lambda_a}{n\left (1- \sum_{b\in R_1}\lambda_b\right )}}\leq  \sqrt{ \frac{\lambda_a}{n\left (1- 2q/k\right )}}\leq 2\sqrt{\frac{\lambda_a}{n}}\ , 
  	 \]
  	 where we used that $\lambda_b\leq 2/k$, $|R_1|\leq q \leq k^{1/2}$ and $k\geq 8$. The supremum in \eqref{eq:exp_T} is achieved for subsets ($S^*,T^*$) such that for all $a\in S^*$, $\sum_{b\in T^*}\lambda_b (\bQ_{ab}-\bQ^{ad}_{ab})$ is non-negative (otherwise this contradicts the optimality of $S^*,T^*$). As a consequence, we can plug the upper bounds on $\E\left [\left |\widehat{\lambda}_a-\lambda_a|\right |(\widehat{\lambda}_c)_{c\in R_1}\right ]$ into \eqref{eq:exp_T}: 
  	 \[
  	 \E\left [T_{R_1,R_2}|(\widehat{\lambda}_a)_{a\in R_1}\right ]  \leq   \frac{2}{\sqrt{n}}\sup_{S\subset [k]\setminus R_1, T\subset [k]}\sum_{a\in S, b\in T}\sqrt{\lambda_a} \lambda_b (\bQ_{ab}-\bQ^{ad}_{ab})\leq 
  	 C \sqrt{\frac{k}{n\log(k)}}\ , 
  	 \]
  	 where we used the property \eqref{control_cut1} of $\bQ^{ad}$. Coming back to \eqref{eq:concentration_TR} and integrating the deviation inequality with respect to $(\widehat{\lambda}_a)_{a\in R_1}$, we conclude that, for any $t>0$
  	 \[
  	 \P\left [T_{R_1,R_2}\geq  C \sqrt{\frac{k}{n\log(k)}}+ \sqrt{\frac{2t}{n}}\right ]\geq 1-e^{-t}.
  	 \]
  	 Fixing $t= 2\log(k) q + \sqrt{k}/\log(k)$ and taking an union bound over all possible $R_1$, $R_2$, we derive that 
  	 \begin{equation}\label{eq:lemma_control_Z}
  	 \max_{R_1,R_2\ , |R_1|\leq q |R_2|\leq q} Z_{R_1,R_2}\leq   C \sqrt{\frac{k}{n\log(k)}} +   q \max_{a=1, \dots k}\left  |\widehat{\lambda}_a-\lambda_a\right |
  	 \end{equation}
  	 on an event of probability higher than $1-\exp(-\sqrt{k}/\log(k))$.

  	 Next we bound $\max_{a=1, \dots k} |\widehat{\lambda}_a-\lambda_a|$. Recall that $n\widehat{\lambda}_a$ has a binomial distribution with parameters ($n$, $\lambda_a$) and $\lambda_{a}\leq 2/k$. For any $a\in[k]$, applying Bernstein's inequality  to $|\widehat{\lambda}_a-\lambda_a|$ we get
  	 	\begin{equation*}
  	 	\mathbb P\left \{n|\widehat{\lambda}_a-\lambda_a|\geq t\right \}\leq 2\exp\left (-\dfrac{t^{2}}{4n/k+2t/3}\right ).
  	 	\end{equation*}
  	 	Taking $t=C \sqrt{n/\log(k)}$ (for a suitable constant $C>0$) and applying
  	 	the union bound, we derive that  with probability larger than $1- 2k\exp(-\sqrt{k}/\log(k))$
  	 	\begin{equation}\label{eq:max_bound}
  	 	\sqrt{k}\max_{a=1, \dots k}|\widehat{\lambda}_a-\lambda_a|\leq C \sqrt{k/(n\log(k))}.
  	 	\end{equation}
  	 The bound \eqref{eq:lemma_control_Z} together with \eqref{eq:max_bound} imply the statement of Lemma \ref{lem:control_Z}.
  \end{proof}

\begin{proof}[Proof of Lemma \ref{lem:second_term_decomposition}]
 As the control of $(W-\widehat{W})|_{\mathcal{R}\times \mathcal{R}}$ is quite similar to that of $(W-\widehat{W})|_{\mathcal{R}\times \mathcal{R}^c}$, we only sketch the main steps. Relying on the graphon $W_1^+$ (defined in \eqref{eq:def_W}), we have the following decomposition:
  \beq\label{eq:decomposition_W_1^+niveau_2}
 \| (W-\widehat{W})|_{\mathcal{R}\times \mathcal{R}}\|_{\square}\leq \| (W_1^+-\widehat{W}_1^+)\big|_{\mathcal{R}\times \mathcal{R}}\|_{\square}+ \|(W-W_1^+)\big|_{\mathcal{R}\times \mathcal{R}}\|_{\square} + \|(\widehat{W}-\widehat{W}_1^+)\big|_{\mathcal{R}\times \mathcal{R}}\|_{\square}.
\eeq 
Since $(W_1^+-\widehat{W}_1^+)(x,y)$ is zero except if $x\in \mathcal{R}_2$ or $y\in \mathcal{R}_2$, we bound the first expression by its $l_1$ norm as for $W_1-\widehat{W}_1$: 
\beq
\E\big[\| (W_1^+-\widehat{W}_1^+)\big|_{\mathcal{R}\times \mathcal{R}}\|_{\square}\big]\leq 2\E[\lambda(\mathcal{R}_2)]\leq 2\frac{k^{1/4}}{\sqrt{n}}.
\label{eq:upper_W1-W1_hat_2}
\eeq

The two last expressions in \eqref{eq:decomposition_W_1^+niveau_2} are  bounded by the cut norm of a  kernel $V$ defined as follows.  For any $a,b\in [k]$, define $\widetilde{\Pi}_{a,b}= [F_{\widehat{\lambda}^{\delta}}(a-1),F_{\widehat{\lambda}^{\delta}}(a))\times [F_{\widehat{\lambda}^{\delta}}(b-1), F_{\widehat{\lambda}^{\delta}}(b))$ where $F_{\widehat{\lambda}^{\delta}}(.)$ has been defined in \eqref{eq:F_lambda_delta}. 
Let $V$ be the $k\times k$ step kernel on $\left [0, \sum_a|\widehat{\lambda}_a-\lambda_a|\right ]^2$ given by $$V(x,y):= \sum_{a,b=1}^{k}\left [{\bf Q}_{ab}- {\bf Q}^{(ap,+)}_{ab}\right ]\mathds{1}_{\widetilde{\Pi}_{a,b}}(x,y).$$  Now, as for the restrictions of $W-W_1$ and $\widehat W-\widehat W_1$ to $\cR\times \cR^c$,  we have 
\beq
 \left \|(W-W_1^+)\big|_{\mathcal{R}\times \mathcal{R}}\right \|_{\square} \bigvee \left \|(\widehat{W}-\widehat{W}_1^+)\big|_{\mathcal{R}\times \mathcal{R}}\right \|_{\square}\leq \|V\|_{\square}\ . \label{eq:W-W1_RRc2}
\eeq
Thus, it boils down to controlling $\mathbb{E}\left [\|V\|_{\square}\right ]$.
Since $V$ is a $k$--step kernel, its cut norm writes as 
\[
\|V\|_{\square}= \sup_{S,T\subset[k]} \Big| \sum_{a\in S,\ b\in T} |\widehat{\lambda}_a-\lambda_a| |\widehat{\lambda}_a-\lambda_a|  \big({\bf Q}_{ab}- {\bf Q}^{(ap,+)}_{ab}\big)\Big|.
\]
As for the kernel $U$ in the main proof, we rely on the Lemma \ref{lem:lovasz_q_random}. The random variables  $\sum_{a}|\widehat{\lambda}_a-\lambda_a|$ and $(\sum_{a}|\widehat{\lambda}_a-\lambda_a|^2)^{1/2}$ are controlled as in \eqref{eq:upper_lambda_a} and \eqref{eq:upper_lambda_a2}. 

 Fix any two subsets $R_1, R_2\subset [k]$  of size less than or equal to $q$ and define 
  \[
  Z_{R_1,R_2}:= V[R_2^{r,V},R_1^{l,V}]= \sum_{a\in R_2^{r,V}}\sum_{b\in R_1^{l,V}} |\widehat{\lambda}_a-\lambda_a|  |\widehat{\lambda}_b-\lambda_b| (\bQ_{ab}-\bQ^{ad,+}_{ab})\ . 
  \]
 The set $R_1^{l,V}$ only depends on $(\widehat{\lambda}_a)_{a\in R_1}$ and  $R_2^{r,V}$ only depends on $(\widehat{\lambda}_a)_{a\in R_2}$. We have 
 \[
 Z_{R_1,R_2}\leq  \sum_{a\in R_2^{r,V}\setminus (R_1\cup R_2)} \sum_{b\in R_1^{l,V}\setminus (R_1\cup R_2)}  |\widehat{\lambda}_a-\lambda_a| |\widehat{\lambda}_b-\lambda_b| (\bQ_{ab}-\bQ^{ad,+}_{ab})+ 4\sum_{a\in R_1\cup R_2} |\widehat{\lambda}_a -\lambda_a|\ , 
 \]
 since  $\sum_{a\in[k]}|\widehat{\lambda}_a -\lambda_a|\leq 2$. We set
  $$T_{R_1,R_2}=\sum_{a\in R_2^{r,V}\setminus (R_1\cup R_2)} \sum_{b\in R_1^{l,V}\setminus (R_1\cup R_2)}  |\widehat{\lambda}_a-\lambda_a| |\widehat{\lambda}_b-\lambda_b| (\bQ_{ab}-\bQ^{ad,+}_{ab}).$$ 
  Write $R:= R_1\cup R_2$ and $\widehat{\lambda}_{\{R\}}:= (\widehat{\lambda}_a)_{a\in R}$. Conditionally to $\widehat{\lambda}_{\{R\}}$, $T_{R_1,R_2}$ is a function of $n-n\sum_{a\in R}\widehat{\lambda}_a$ independent random variables. Besides, if we change the values of one of these independent random variables the value of $T_{R_1,R_2}$ changes by at most $4/n$. Hence,  the bounded difference inequality  enforces that, for any $t>0$, 
 \beq\label{eq:concentration_TR2}
 \P\Big[T_{R_1,R_2}\geq \E[T_{R_1,R_2}|\widehat{\lambda}_{\{R\}}]+ 8\sqrt{\frac{2t}{n}}\Big|\widehat{\lambda}_{\{R\}}\Big]\geq 1-e^{-t}.
\eeq
The conditional expectation is upper bounded by
\beq\label{eq:upper_E_TR}
  \E[T_{R_1,R_2}|\widehat{\lambda}_{\{R\}}]\leq  \sup_{S\subset [k]\setminus R,\, T\subset [k]\setminus R}\sum_{a\in S, b\in T}\E\big[|\widehat{\lambda}_a-\lambda_a||\widehat{\lambda}_b-\lambda_b|\big|\widehat{\lambda}_{\{R\}}\big]  (\bQ_{ab}-\bQ^{ad,+}_{ab}) \ . 
\eeq
Here, unfortunately, we cannot directly replace $\E\big[|\widehat{\lambda}_a-\lambda_a||\widehat{\lambda}_b-\lambda_b|\big|\widehat{\lambda}_{\{R\}}\big]$ by an upper bound of it because this expression does not factorize. 
We shall prove that $\E\big[|\widehat{\lambda}_a-\lambda_a||\widehat{\lambda}_b-\lambda_b|\big|\widehat{\lambda}_{\{R\}}\big] $ is, up to a small loss,  close to a product of expectations.

Write $N:= n- n\sum_{c\in R} \widehat{\lambda}_c$, $\lambda_{R}:= \sum_{c\in R} \lambda_c$  and $\widehat{ \lambda}_{R}= \sum_{c\in R} \widehat \lambda_c$.
Note that $n\widehat{\lambda}_R$ has a binomial distribution with parameters ($n$, $\lambda_R$). Applying Bernstein's inequality  to $|\widehat{\lambda}_R-\lambda_R|$ we get
\begin{equation}\label{bernstein_R}
\mathbb P\left \{n|\widehat{\lambda}_R-\lambda_R|\geq t\right \}\leq 2\exp\left (-\dfrac{t^{2}}{4n/\sqrt{k}+2t/3}\right ).
\end{equation}
Let $\cR= \left \{|\widehat{\lambda}_R-\lambda_R|\leq \frac{1}{\sqrt{n\log(k)}}\right \}$. Taking $t=\sqrt{n/\log(k)}$ in \eqref{bernstein_R} we have that 
$$\P(\cR)\geq 1-2e^{-\sqrt{k}/\log(k)}.$$ 
In what follows we assume that the event $\cR$ is true.
Take any two distinct elements $a$ and $b$ of $[k]\setminus R$. We shall prove that the conditional expectations $\E\left [\left |\widehat{\lambda}_a-\lambda_a\right \vert\left |\widehat{\lambda}_b-\lambda_b\right |\Big\vert\widehat{\lambda}_{\{R\}}\right ]$ are close to the products $\E\left [\left |\widehat{\lambda}_a-\lambda_a\right |\Big\vert\widehat{\lambda}_{\{R\}}\right ]\E\left [\left |\widehat{\lambda}_b-\lambda_b\right |\Big\vert\widehat{\lambda}_{\{R\}}\right ]$. 
It is easy to see that  conditionally on $(\widehat{\lambda}_{\{R\}},\widehat{\lambda}_a)$, $n\widehat{\lambda}_b$ follows the Binomial distribution with parameters $( (N-n\widehat{\lambda}_a),\lambda_b/(1- \lambda_{R} - \lambda_a))$. On the other hand, conditionally on $\widehat{\lambda}_{\{R\}}$, $n\widehat{\lambda}_b$ follows the Binomial distribution with parameters $(N,\lambda_b/(1- \lambda_{R}))$. Let $z_1,z_2,\ldots, $ be a sequence of independent Bernoulli random variables with parameters $\lambda_b/(1- \lambda_a - \lambda_{R})$, $w_1,w_2\ldots, $ be an independent sequence of Bernoulli random variables with parameters $(1- \lambda_a -\lambda_{R})/(1- \lambda_{R})$ and $v_1,v_2,\ldots, $  be an independent sequence of Bernoulli random variables with parameters  $\lambda_b/(1- \lambda_{R})$. We define the following  random variables: 
 \[X:= \sum_{i=1}^{N-n\widehat{\lambda}_a}z_i\ , \quad \quad \quad Y:= \sum_{i=1}^{N-n\widehat{\lambda}_a}z_iw_i+ \sum_{i=1}^{n\widehat{\lambda}_a}v_i\]
where we use $\lambda_c\leq 2/k$ and $|R|\leq 2\sqrt{k}$. It is easy to see that $X$ follows the Binomial distribution with parameters $(N-n\widehat{\lambda}_a)$ and $\lambda_b/(1- \lambda_{R} - \lambda_a)$ and $Y$ follows the Binomial distribution with parameters $N$ and $\lambda_b/(1- \lambda_{R})$. Hence, we have that
\begin{equation}\label{eq:cal_exp_approx}
\Big|\E\big[|\widehat{\lambda}_a-\lambda_a||\widehat{\lambda}_b-\lambda_b|\big|\widehat{\lambda}_{\{R\}}\big] - \E\big[|\widehat{\lambda}_a-\lambda_a|\widehat{\lambda}_{\{R\}}] \E[|\widehat{\lambda}_b-\lambda_b|\big|\widehat{\lambda}_{\{R\}}\big] \Big| \leq \frac{1}{n}\E\Big[|X-Y||\widehat{\lambda}_a-\lambda_a|\Big|\widehat{\lambda}_{\{R\}}\Big].
\end{equation}
Relying our coupling between $X$ and $Y$, we obtain
\begin{eqnarray} \nonumber
\frac{1}{n}\E\Big[|X-Y|\Big|\widehat{\lambda}_{\{R\}},\widehat{\lambda}_a\Big]&\leq& \frac{1}{n}\E\left  [\sum_{i=1}^{N-n\widehat{\lambda}_a}z_i(1-\omega_{i})\Big|\widehat{\lambda}_{\{R\}},\widehat{\lambda}_a\right ]+
\frac{1}{n}\E\left [\sum_{i=1}^{n\widehat{\lambda}_a}v_i\Big|\widehat{\lambda}_{\{R\}},\widehat{\lambda}_a\right ]\\
&& =   \dfrac{N-n\widehat{\lambda}_a}{n}\frac{\lambda_b\lambda_a}{(1-\lambda_R)(1-\lambda_a-\lambda_R)}+\widehat{\lambda}_a\frac{\lambda_b}{1- \lambda_{R}} \nonumber \\
& \leq & \frac{\lambda_b\lambda_a}{(1-\lambda_R)(1-\lambda_a-\lambda_R)}+\frac{\widehat{\lambda}_a\lambda_b}{1- \lambda_{R}}.
\label{eq:cond_exp_0} 
\end{eqnarray}
On the other hand, conditionally on $\widehat{\lambda}_{\{R\}}$, $n\widehat{\lambda}_a$ follows the Binomial distribution with parameters $(N,\lambda_a/(1- \lambda_R))$ so that  Cauchy-Schwarz inequality  implies
\begin{eqnarray} \nonumber
\E\left [\left |\widehat{\lambda}_a-\lambda_a\right |\Big|\widehat{\lambda}_{\{R\}}\right ]&=&
\E\left [\left |\widehat{\lambda}_a-\dfrac{N\lambda_a}{(1- \lambda_R)n}+\dfrac{N\lambda_a}{(1- \lambda_R)n}-\lambda_a\right |\Big|\widehat{\lambda}_{\{R\}}\right ]\\ \nonumber
&\leq&
\E\left [\left |\widehat{\lambda}_a-\dfrac{N\lambda_a}{(1- \lambda_R)n}\right |\Big|\widehat{\lambda}_{\{R\}}\right ]+\left |\dfrac{N\lambda_a}{(1- \lambda_R)n}-\lambda_a\right |\\ \nonumber
&\leq &C\sqrt{\lambda_a/n}+\lambda_a\left |\lambda_R-\widehat{\lambda}_R\right |\\
&\leq& C\sqrt{\lambda_a/n}+\dfrac{4}{\sqrt{kn\log(k)}}
\label{eq:cond_exp_1}
\end{eqnarray}
where we use that $\lambda_a\leq 2/k$ and the definition of the event $\cR$.
Similarly we compute 
\begin{equation}\label{eq:cond_exp_2}
\begin{split}
\E\left [\widehat{\lambda}_a\left |\widehat{\lambda}_a-\lambda_a \right |\Big|\widehat{\lambda}_{\{R\}}\right ]\leq C\left (\dfrac{1}{kn}+\dfrac{1}{k\sqrt{kn}}\right )
\end{split}
\end{equation}
Plugging (\ref{eq:cond_exp_0} -- \ref{eq:cond_exp_2}) into \eqref{eq:cal_exp_approx} we get
\beqn 
\lefteqn{\Big|\E\big[|\widehat{\lambda}_a-\lambda_a||\widehat{\lambda}_b-\lambda_b|\big|\widehat{\lambda}_{\{R\}}\big] - \E\big[|\widehat{\lambda}_a-\lambda_a|\big|\widehat{\lambda}_{\{R\}}] \E[|\widehat{\lambda}_b-\lambda_b|\big|\widehat{\lambda}_{\{R\}}\big] \Big|} &&\\
&\leq & 
\E\left [\frac{\lambda_b\lambda_a}{(1-\lambda_R)(1-\lambda_a-\lambda_R)}|\widehat{\lambda}_a-\lambda_a|\Big|\widehat{\lambda}_{\{R\}}\right ]+ \E\left [\widehat{\lambda}_a\frac{\lambda_b}{1- \lambda_{R}}  |\widehat{\lambda}_a-\lambda_a|\Big|\widehat{\lambda}_{\{R\}}\right ]\\
&\leq & C\left [\frac{1}{k^{5/2}n^{1/2}}+ \frac{1}{nk^2}\right ]\ ,
\eeqn
where we use $\lambda_b,\lambda_a\leq 2/k$. For $a=b$, \eqref{eq:cond_exp_1} implies that the above difference is of order $(kn)^{-1}$. 
Going back to \eqref{eq:upper_E_TR}, we obtain that
\begin{equation*}
\begin{split}
\E\left [T_{R_1,R_2}|\widehat{\lambda}_{\{R\}}\right ]&\leq  \sup_{S\subset [k]\setminus R,\, T\subset [k]\setminus R}\sum_{a\in S, b\in T}\E\big[|\widehat{\lambda}_a-\lambda_a|\big|\widehat{\lambda}_{\{R\}}\big]\E\big[|\widehat{\lambda}_b-\lambda_b|\big|\widehat{\lambda}_{\{R\}}\big]  (\bQ_{ab}-\bQ^{ad,+}_{ab}) + \frac{C}{\sqrt{n}}.
\end{split}
\end{equation*}
Take $S^*$ and $T^*$ being two sets maximizing the above expression. Then, for all  $a\in S^*$ we have that
$\sum_{ b\in T^*}\E\big[|\widehat{\lambda}_b-\lambda_b|\big|\widehat{\lambda}_R\big]  (\bQ_{ab}-\bQ^{ad,+}_{ab})$ is non-negative. As a consequence, using \eqref{eq:cond_exp_1}, we have 
that 
\[
  \E\left [T_{R_1,R_2}|\widehat{\lambda}_{\{R\}}\right ]\leq C \sup_{S\subset [k]\setminus R,\, T\subset [k]\setminus R}\sum_{a\in S, b\in T}\left ( \sqrt{\frac{\lambda_a}{n}}+\dfrac{4}{\sqrt{kn\log(k)}}\right )\E\left [|\widehat{\lambda}_b-\lambda_b|\big|\widehat{\lambda}_{\{R\}}\right ] (\bQ_{ab}-\bQ^{ad,+}_{ab}) + \frac{C'}{\sqrt{n}}\ , 
\]
as soon as the event $\cR$ holds.
The same reasoning and $\vert \bQ_{ab}-\bQ^{ad,+}_{ab}\vert\leq 2 $ leads to 
\beqn
  \E\left [T_{R_1,R_2}|\widehat{\lambda}_{\{R\}}\right ]&\leq& C \sup_{S\subset [k]\setminus R,\, T\subset [k]\setminus R}\sum_{a\in S, b\in T} \frac{\sqrt{\lambda_b\lambda_a}}{n} \left (\bQ_{ab}-\bQ^{ad,+}_{ab}\right ) +C''\left (\frac{k}{n\sqrt{\log(k)}}+\frac{1}{\sqrt{n}}\right ) \\
  &\leq & C\sqrt{\frac{k}{n\log(k)}}\ ,
\eeqn 
as soon as the event $\cR$ holds.
Going back to \eqref{eq:concentration_TR2} and integrating the deviation inequality with respect to $\widehat{\lambda}_{\{R\}}$, we conclude that
\[
 \P\left [T_{R_1,R_2}\geq  C \sqrt{\frac{k}{n\,\log(k)}}+ 8\sqrt{\frac{2t}{n}}\right ]\geq 1-e^{-t}-\P[\overline{\cR}]\geq 1-e^{-t}- 2e^{-\sqrt{k}/\log(k)}
\]
where we use  $\P(\cR)\geq 1-2e^{-\sqrt{k}/\log(k)}$.
From this point the proof is identical to that of the main proof: we fix $t= 2\log(k) q + \sqrt{k}/\log(k)$ and take an union bound over all possible $R_1$ and $R_2$ to  derive that 
\[
\max_{R_1,R_2: |R_1|\leq q, |R_2|\leq q} Z_{R_1,R_2}\leq   C \sqrt{\frac{k}{n\,\log(k)}} +   4q \max_{a=1,\dots, k}|\widehat{\lambda}_a-\lambda_a|
\]
on an event of probability higher than $1-3\exp(-\sqrt{k}/\log(k))$. Then, as in the main proof, 
Lemma \ref{lem:lovasz_q_random} together with \eqref{eq:resuming_second_bound} and \eqref{eq:max_bound} enforce that 
$  \|V\|^{+}_{\square}\leq     C \sqrt{K/(n\log(k))}$ with probability larger than $1 - (5+2k)\exp(-\sqrt{k}/\log(k))$. By symmetry, we can find  an event $\cA$ of probability larger than $1- (10+4k)\exp(-\sqrt{k}/\log(k))$ such that, on $\cA$, 
\[
  \|V\|_{\square}\leq C\sqrt{\frac{k}{n\,\log(k)}}\ . 
\]
In order to  control  $\E[\|V\|_{\square}]$ on the complementary event $\bar{\cA}$ we use the rough bound $$\|V\|_{\square}\leq \|V\|_1\leq \sum_{a,b=1}^k|\widehat{\lambda}_a-\lambda_a||\widehat{\lambda}_b-\lambda_a|\leq 2 \sum_{a=1}^k|\widehat{\lambda}_a-\lambda_a|$$ which implies
\begin{eqnarray} \nonumber
 \E[\|V\|_{\square}]&\leq & E[\|V\|_{\square}\1_{\cA}]+ \E[\|V\|_{\square}\1_{\bar{\cA}}]\\ \nonumber
 &\leq& C\sqrt{\frac{k}{n\,\log(k)}} + 2\P^{1/2}\left [\bar{\cA}\,\right ]\left [\E\left (\sum_{a=1}^k |\widehat{\lambda}_a-\lambda_a|\right )^2\right ]^{1/2}\\
 &\leq & C\sqrt{\frac{k}{n\,\log(k)}} + C' e^{-\sqrt{k}/(2\log(k))} \frac{k}{\sqrt{n}}\leq C'' \sqrt{\frac{k}{n\,\log(k)}} \nonumber
\end{eqnarray}  
 where we use \eqref{eq:upper_lambda_a}.
  Together with the decomposition  \eqref{eq:decomposition_W_1^+niveau_2}, \eqref{eq:upper_W1-W1_hat_2} and \eqref{eq:W-W1_RRc2}, we conclude that 
  \[
  \mathbb{E}\left [\left \| (W-\widehat{W})|_{\mathcal{R}\times \mathcal{R}}\right \|_{\square}\right ]\leq   C \sqrt{\frac{k}{n\log(k)}}\ .
  \]

 \end{proof}

   \section{Proof of Theorem \ref{prp:lower_minimax_cut}}\label{proof_prp:lower_minimax_cut}
    
   It is enough to prove  separately the following two  minimax lower bounds:
    \begin{eqnarray} 
     \inf_{\widehat{f}}\sup_{W_0\in \cW^+[k]}\E_{W_0}[\delta_\cut(\widehat{f}, \rho_n W_0)] &\geq& C \rho_n \sqrt{\frac{k}{n\log(k)}}\ ,\label{eq:lower_minimax_integrated_risk_distance_cut} \\
     \inf_{\widehat{f}}\sup_{W_0\in \cW^+[2]}\E_{W_0}[\delta_\cut(\widehat{f}, \rho_n W_0)] &\geq& C  
    \left( \sqrt{\frac{\rho_n}{n}}\wedge \rho_n\right).
     \label{eq:lower3_cut_norm}
    \end{eqnarray} 
  The proof of \eqref{eq:lower3_cut_norm} is identical to the proof of (45) in \cite{klopp_graphon} so we just sketch the main idea. Fix some $0<\epsilon\leq 1/4$.
We consider  $W_1$ to be the constant graphon with $W_1(x,y)\equiv 1/2$, and  $W_2\in \cW^+[2]$ to be the $2$--step graphon with 
 $W_{2}(x,y)=1/2+\epsilon$ if $x,y\in [0,1/2)^2\cup [1/2,1]^2$ and $W_{2}(x,y)=1/2-\epsilon$ elsewhere. Obviously, we have $\delta_{\cut}[\rho_n W_1,\rho_n W_2]=\rho_n\epsilon$. Then, standard testing arguments~\cite{tsybakov_book} ensure that the minimax risk $   \inf_{\widehat{f}}\sup_{W_0\in \cW^+[2]}\E_{W_0}[\delta_\cut(\widehat{f}, \rho_n W_0)] $ is at least of the order $\rho_n \epsilon$ when $\epsilon$ is chosen small enough so that the $\chi^2$-distance $\chi^2(\P_{W_2},\P_{W_1})$ is smaller than $1/4$. According to  Lemma 4.9 in~\cite{klopp_graphon}, this is the case when $\epsilon$ is small in front of $(\rho_n n)^{-1/2}$ which proves \eqref{eq:lower3_cut_norm}. 
  
  \smallskip
  Henceforth, we  only focus on  \eqref{eq:lower_minimax_integrated_risk_distance_cut}. 
  We first consider the case of $k$ multiple of $32$ and such that $k\geq C_0$ and $k\leq C_1n$ for some sufficiently large numerical constants $C_0$ and $C_1$. As the  collections $\cW^+[k]$ are nested this will imply \eqref{eq:lower_minimax_integrated_risk_distance_cut} for all $k\in [32 C_0, n]$. Afterwards, it will suffice to show \eqref{eq:lower_minimax_integrated_risk_distance_cut} for $k=2$ to prove the proposition.
   So, we assume that $k$ is a multiple of 32, $k$ is large enough and that $k$ is small in front of $n$. Define $k_1:=k/2$, $M_k :=\lceil  128\log(k)\rceil$, $\eta_0:=1/16$ and $\eta_1 :=7/8$.

  As for Proposition \ref{lemma_low-bound_prob}, we will rely on Fano's method (Lemma \ref{lem:Fano_tsybakov}). Hence, we shall build a collection $(W_u)$ of graphons that are well-spaced in cut distance and such that the  Kullback-Leibler divergence between the associated distribution $\P_{W_u}$ remains small enough. 
  All the graphons considered in this collection will be based on a  $k_1\times M_k$ matrix  $\bB$ such that (i) the rows of $B$ are almost orthogonal and (ii) such that the $l_1$ distance between permutation and convex combinations of the columns of $\bB$ are bounded from below.  Such a property will turn out to be useful when taking a lower bound on the $\delta_{\cut}$ distance between the corresponding graphons.

  \begin{lem}\label{lem:existence_hadamard_matrix}
  For $k$ large enough, there exists a  matrix  $\bB\in \{-1,1\}^{k_1\times M_k}$ satisfying the following two properties:
  \begin{itemize}
   \item [(i)] For any $(a,b)\in [k_1]$ with $a\neq b$, the inner product of two columns $\langle \bB_{a,\cdot} , \bB_{b,\cdot}\rangle$ satisfies
   \beq\label{eq:property_1_B}
   |\langle \bB_{a,\cdot} , \bB_{b,\cdot}\rangle|\leq M_k/4.
   \eeq
   \item [(ii)]For any two subsets $X$ and $Y$ of $[k_1]$ satisfying $|X|=|Y|=\eta_0 k_1$ and $X\cap Y=\emptyset$, any labellings $\pi_1:[\eta_0  k_1 ]\to X$ and $\pi_{2}:[\eta_0 k_1 ]\to Y$, 
  any subset $Z$ of $[M_k]$ of size larger than $\eta_1 M_k$ and any $Z\times M_k$ stochastic  matrix $\omega$, we have
   \beq\label{eq:property_2_B}
   \sum_{a=1}^{ \eta_0 k_{1}}\sum_{b\in Z}\big|\bB_{\pi_1(a),b}-\sum_{c\in [M_k]}\omega_{b,c}\bB_{\pi_2(a),c}\big|\geq   C M_k k_{1},
   \eeq
   for some universal constant $C>0$.
  \end{itemize}
  \end{lem}
  
 Taking $\bB$ as in Lemma \ref{lem:existence_hadamard_matrix}, 
 we define the connection probability matrix  $\bQ : = (\bJ + \bB)/2$ where  $\bJ$ is the $k_1\times M_k$ matrix with all entries equal to 1.  Now we define a collection of step graphons based on $\bQ$ that will only slightly differ by the weight of each step.

 Fix some $\epsilon < 1/(8k_1)$ and denote by $\cC_0$ the collection of vectors $u\in \{-\epsilon,\epsilon\}^{k_1}$ satisfying $\sum_{a=1}^{k_1} u_a=0$. For any $u\in \cC_0$,  define the cumulative distribution $F_u$ on $\{0,\ldots,  k_1\}$ by the relations $F_u(0)=0$ and   $F_u(a)= a/(2k_1) + \sum_{b=1}^a u_b$ for $a\in [k_1]$ and the cumulative distribution $G$ on $\{0,\ldots, M_k\}$ by $G(0)= 1/2$ and $G(b)=1/2 + b/(2M_k) $. Note that $F_u$ takes values in $[0,1/2]$ and $G$ takes values in $[1/2,1]$. Then, set $\Pi_{ab}(u)=[F_{u}(a-1),F_{u}(a))\times [G(b-1),G(b))$ and define the graphon $W_u\in \cW^+[k_{1}+M_{k}]$ by
   \begin{equation*}
   W_u(x,y)=\left\{
   \begin{array}{cc}
  \sum_{(a,b)\in [k_1]\times [M_k]}{\bQ}_{ab}\mathds{1}_{\Pi_{ab}(u)}(x,y)&\text{ if } x\in [0,1/2]\text{ and } y\in (1/2,1]\\
  W_u(y,x)&\text{ if  }  x\in (1/2,1]\text{ and } y\in [0,1/2]\\
   1/2  &\text{ else}\ .
   \end{array}
   \right.
   \end{equation*}
   See Figure \eqref{:fig1} for a drawing of $W_u$. 
  Note that $W_u$ is a fairly unbalanced $(k_1+M_k)$--step graphon: $M_k$ of its steps have a large weight of order $1/\log(k)$. Besides, the $k_1$ smaller steps are  slightly unbalanced as the weight of each class is either $1/k-\epsilon$ or $1/k+\epsilon$.  The purpose of these $M_k$ big steps is to make the cut distances between   $W_u$ and $W_{v}$ the largest possible (see the proof of Lemma \ref{lem:separated_cut_norm}). 
  
  \begin{figure}
  
 \begin{center}

 \begin{tikzpicture}
 \draw [thick] (-3,-3) rectangle (3,3);
 \draw [thick] (-3,-1) -- (3,-1);
 \draw [thick] (-3,1) -- (3,1);

 \draw [thick] (-2.4,-3) -- (-2.4,3);
 \draw [thick] (-1.8,-3) -- (-1.8,3);
 \draw [thick] (-1.2,-3) -- (-1.2,3);
 \draw [thick] (-0.6,-3) -- (-0.6,3);
 \draw [thick] (-0,-3) -- (-0,3);
 \draw [thick] (2.4,-3) -- (2.4,3);
 \draw [thick] (1.8,-3) -- (1.8,3);
 \draw [thick] (1.2,-3) -- (1.2,3);
 \draw [thick] (0.6,-3) -- (0.6,3);

 \draw [thick](-3,-3.2) node[below] {$0$};
 \draw [thick](3,-3.2) node[below] {$0.5$};
 \draw [thick](-3.2,-3) node[left] {$0.5$};
 \draw [thick](-3.2,3) node[left] {$1$};
 
 \draw [thick,<->,color=blue] (-3.3,-1) -- (-3.3,1);
 \draw [thick](-3.3,0) node[left] {$1/(2M_k)$};
 
 \draw [thick,<->,color=blue] (-0.6,-3.3) -- (0,-3.3);
 \draw [thick](-0.3,-3.3) node[below] {$1/k \pm \epsilon$};

  \draw [color=red](-.3,0) node { $\scriptstyle \bQ_{a,b}$};
 
 \end{tikzpicture}
\end{center} 
\caption{Restriction of $W_u$ to $[0,1/2]\times [1/2,1]$.  \label{:fig1}}
  \end{figure}
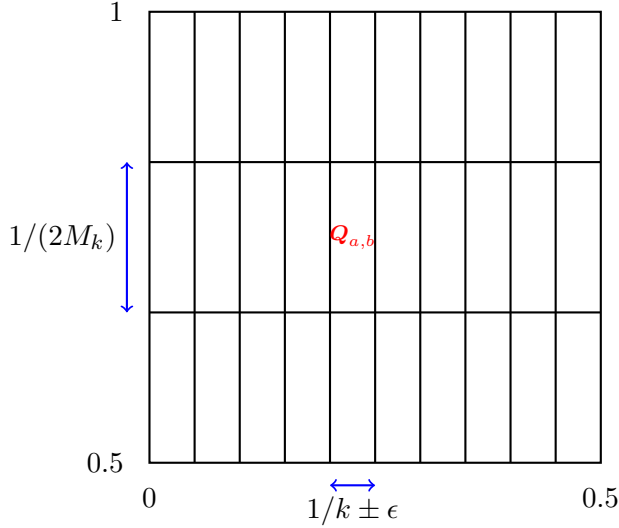

  Next, we shall consider a subcollection $\cC$ of $\cC_0$ such that the graphons $W_u$ with $u\in \cC$ are well spaced. The following  combinatorial result is  in the spirit of the Varshamov-Gilbert lemma~\cite[Lemma 2.9]{tsybakov_book}. It is borrowed from \cite{klopp_graphon} (Lemma 4.4). 
 For $u\in \cC_0$, let  $\cA_u:=\{a\in [k_1]:\ u_a=\epsilon\}$. Notice that, by definition of $\cC_0$, we have~$|\cA_u|=k_1/2$ for all $u\in \cC_0$. 
  
  \begin{lem}\label{lem:varshamov_variation}
   There exists a subset $\cC$ of $\cC_0$ such that $\log |\cC|\geq k_1/16$ and 
   \beq
   |\cA_u\Delta \cA_v|> k_1/4\ . 
   \eeq
   for any $u\neq v\in \cC$.
  \end{lem}
  
  Lemmas~\ref{lem:existence_hadamard_matrix} and~\ref{lem:varshamov_variation} are used to obtain the following  lower bound on the distance $\delta_{\square}(W_u,W_v)$ between two distinct graphons with $u$ and $v$ in $\cC$. This lemma is the main ingredient of the proof.
  
  \begin{lem}\label{lem:separated_cut_norm}
   There exists two positive universal constants $C_1$ and $C_2$ such that if
   $k\epsilon\leq C_2$ then,
  for any $(u,v)\in \mathcal{C}$ with $u\neq v$, we have 
  \beq\label{eq:lower_separated_cut_norm}
   \delta_{\square}\big(W_u,W_v\big)\geq C_1 \frac{k\epsilon}{\sqrt{M_k}}
   \eeq
   which implies
   \beq\label{separ_cut_weight_graphon}
   \delta_{\cut}(\rho_n W_u,\rho_n W_v)\geq C_1 \rho_n \frac{ k\epsilon}{\sqrt{M_k}}\ . 
    \eeq 
  \end{lem}
 
  Note that for any $u$ and $v$ in $\cC$ it is possible to build a measure-preserving transformation $\tau$ such that $W_u-W_v^{\tau}$ is null expect on a measurable set of Lebesgue measure of order $k\epsilon$ (see the proof of Theorem \ref{prp_agnostic_block} in Section \ref{proof_agnostic_cut} for such  construction). Hence, the $l_1$ norm of $W_u-W_v^{\tau}$ is of order $k\epsilon$. Lemma \ref{lem:separated_cut_norm} states, that by taking the infimum over all $\tau$ and by considering the weaker norm $\|.\|_{\square}$, one still has a lower bound of the same order. The $M^{-1/2}_k$ factor arises as a consequence of Lemma \ref{lem:1cut_l1}. See the proof for more details.

 To apply Fano's method, we need to upper bound the Kullback-Leibler divergence between the distribution corresponding to any two graphon $W_u$ and $W_v$ with $u$ and $v$ in $\cC$.  Let $\P_{W_u}$ denote the distribution of $\bA$ sampled according to the sparse $W$-random graph model \eqref{sparse_graphon_mod} with $W_0=W_u$. Since the matrix $\bQ$ is fixed the difficulty in distinguishing between the distributions $\P_{W_u}$ and $\P_{W_v}$ for $u\ne v$ comes from the randomness of the design points $\xi_1,\ldots,\xi_n$ in the $W$-random graph model \eqref{sparse_graphon_mod} rather than from the randomness of the realization of the adjacency matrix $\bA$ conditionally on $\xi_1,\ldots,\xi_n$. The following lemma gives an upper bound on the Kullback-Leibler  divergences $\mathcal{KL}(\P_{W_u},\P_{W_v})$:
  \begin{lem}\label{lem:kullback}
   For all $u,v\in \cC_0$ we have
   $$
  \mathcal{KL}(\P_{W_u},\P_{W_v})\le 32n k_1^2 \epsilon^2/3.
   $$
  \end{lem}

  Now, choose $\epsilon$ such that $\epsilon^2=\frac{3}{(16)^3 nk_1}$. When $k$ is small in front of $n$, this choice of $\epsilon$ satisfies the conditions of Lemma \ref{lem:separated_cut_norm}.
  Then it follows from Lemmas \ref{lem:varshamov_variation} and \ref{lem:kullback} that
   \beq\label{kull}
  \mathcal{KL}(\P_{W_u},\P_{W_v})\le \frac1{16}\log |\cC|, \quad \forall \ u, v\in \cC: u\ne v.
   \eeq
  In view Fano's Lemma (Lemma \ref{lem:Fano_tsybakov}), inequalities \eqref{separ_cut_weight_graphon} and \eqref{kull} imply that
  \[\inf_{\widehat{f}}\sup_{W_0\in \cW^+[k]}\E_{W_0}[\delta_{\square}(\widehat{f}, \rho_n W_0)] \geq C \rho_n  \sqrt{\frac{k}{n\log(k)}}\]
  where $C>0$ is an absolute constant. This completes the proof for $k$ large enough.
  
  \bigskip

  Now we turn to the case $k=2$. We reduce the lower bound to the problem of testing two hypotheses. Consider the matrix $\bB=\left(\begin{array}{cc} 1  &   1 \\ 1 & -1 \end{array}\right)$. Given $u\in \{-\epsilon,+\epsilon\}$ define $F_u(0)=0$, $F_u(1)=1/2+u$ and $F_u(2)=1$. Then, we set  $\Pi_{ab}(u)=[F_u(a-1),F_u(a))\times [F_u(b-1),F_u(b))$ for any $a,b\in \{1,2\}$ and define  graphons $$W_{u}(x,y):= \sum_{a,b=1}^2 \frac{(\bB_{ab}+1)}{2}\1_{\Pi_{ab}(u)}(x,y).$$
  For any measure preserving bijection $\tau$, $(W_{\epsilon}-W^{\tau}_{-\epsilon})$ is a four-step graphon. Thanks to Lemma \ref{lem:1cut_l1}, we deduce that 
  $\delta_{\square}(W_{\epsilon},W_{-\epsilon})\geq C \delta_{1}(W_{\epsilon},W_{-\epsilon})$. Then, it is not hard to see that $\delta_{1}(W_{\epsilon},W_{-\epsilon})\geq C'\epsilon$ so that $\delta_{\square}(\rho_n W_{\epsilon},\rho_n W_{-\epsilon})\geq C'\rho_n\epsilon$.
   Moreover, exactly as in Lemma \ref{lem:kullback}, the Kullback-Leibler divergence between $\P_{W_{\epsilon}}$ and $\P_{W_{-\epsilon}}$ is bounded by $Cn\epsilon^2$.  Taking  $\epsilon$ of the order $n^{-1/2}$,
   this divergence is small. It is therefore impossible to reliably distinguish $\P_{W_{\epsilon}}$ from $\P_{W_{-\epsilon}}$ and the estimation error is at least of order $\rho_n \epsilon$. More formally, we use Theorem 2.2 from \cite{tsybakov_book} to conclude that
  \[\inf_{\widehat{f}}\sup_{W_0\in \cW^+[2]}\E_{W_0}[\delta_{\square}(\widehat{f}, \rho_n W_0)] \geq C \rho_n \sqrt{\frac{1}{n}}\]
  where $C>0$ is an absolute constant.

\begin{proof}[\textbf{Proof of Lemma \ref{lem:existence_hadamard_matrix}}]
Let $\bB$ be a  $k_1\times M_k$ random matrix whose entries are independent Rademacher variables. We shall prove that, with positive probability, $\bB$ satisfies both \eqref{eq:property_1_B} and \eqref{eq:property_2_B}. In particular, this implies the existence of $\bB$ satisfying both \eqref{eq:property_1_B} and \eqref{eq:property_2_B}.

Fix $a\neq b$. Then,   $\langle \bB_{a,\cdot} , \bB_{b,\cdot}\rangle$ is distributed as a sum of $k_1$ independent Rademacher variables. Using Hoeffding's inequality, we have that
\[\P\left[|\langle \bB_{a,\cdot} , \bB_{b,\cdot}\rangle|\geq M_k/4 \right]\leq 2\exp[-M_k/32].\]
By the union bound, property \eqref{eq:property_1_B} is satisfied for all $a\neq b$ with probability greater than $1- k_1^2\exp[-M_k/32]$. Since $M_k\geq 128\log(k)$, for $k$ greater than some absolute constant, this probability is greater than $3/4$.

Turning to \eqref{eq:property_2_B}, we first fix $X$, $Y$, $Z$, $\pi_1$, $\pi_2$, and $\omega$. Let $$T_{X,Y,Z,\pi_1,\pi_2,\omega}:= \sum_{a=1}^{ \eta_0 k_1}\sum_{b\in Z}\big|\bB_{\pi_1(a),b}-\sum_{c\in [M_k]}\omega_{b,c}\bB_{\pi_2(a),c}\big|.$$
We have that, conditionally on $(\bB_{b,c})_{b\in Y, c\in [M_k]}$, $T_{X,Y,Z,\pi_1,\pi_2,\omega}$ stochastically dominates a binomial distribution with parameters $(\eta_0 k_{1})\times |Z|$ and $1/2$. Then, Hoeffding's inequality  yields
\[\P\left\{T_{X,Y,Z,\pi_1,\pi_2,\omega}\leq \eta_0k_{1}|Z|  /4 \right\}\leq 2\exp(-\eta_0\eta_1 k_1M_k/8).\]
Given any integer $Z\in [\eta_1 M_k, M_k]$, define $\Omega_Z$ the collection of $Z\times [M_k]$ stochastic  matrices taking values in the discrete set $\{0,1/(8M_k),2/(8M_k),\ldots, 1\}$. 
 Since $X,Y\subset[k_1]$ and $Z\subset M_k$, it is easy to see that the cardinality of the set of all possible tuples
 $(X,Y,Z,\pi_1,\pi_2,\omega)$ with $\omega\in\Omega_Z$ is bounded by 
\[
 2^{2k_1+M_k}\left ((\eta_0 k_1)!\right )^2 (8M_k+1)^{M_k^2}.
\]
Now, taking the union bound, we derive that, simultaneously for all such parameters,
 $$T_{X,Y,Z,\pi_1,\pi_2,\omega}> \eta_0 k_{1} |Z|  /4$$
with probability greater than $1- 2^{2k_1+M_k+1} (\eta_0 k_1)!^2 (8M_k+1)^{M_k^2}\exp[-\eta_0\eta_1 k_1M_k/8]$. 
 Using Stirling's approximation 

 and $\eta_1 M_k\geq 64 \log(k)$ we get that this probability is larger than $3/4$ for $k$ large enough. 
 
 Finally, let us consider a general case,  when matrix $\omega$ does not necessarily belong to $\Omega_Z$. Observe that in this case, there exists a matrix $\omega'\in \Omega_Z$ such that $\max_{b\in Z}\sum_{c\in [M_k]}|\omega_{b,c}-\omega'_{b,c}|\leq 1/8$. This implies that 
 \[
  T_{X,Y,Z,\pi_1,\pi_2,\omega}\geq T_{X,Y,Z,\pi_1,\pi_2,\omega'} - \frac{|Y||Z|}{8}\geq \eta_0\eta_1 k_{1}M_k/8\ .
 \]
We have proved that \eqref{eq:property_2_B} holds with probability larger than $3/4$. As a consequence, $\bB$ satisfies both \eqref{eq:property_1_B} and \eqref{eq:property_2_B} with probability larger than $1/2$. 
\end{proof}

  \smallskip
\begin{proof}[\textbf{Proof of Lemma \ref{lem:separated_cut_norm}}]
 We fix $u$ and $v$, two different vectors in $\mathcal{C}$, and fix   $\tau$, a measure-preserving bijection  on $[0,1]\rightarrow [0,1]$. We shall prove that for $k\epsilon$ small enough
 \beq\label{eq:lower_cut_separated}
\left  \|W_u- W_v^{\tau}\right \|_{\square}\geq C \frac{k\epsilon}{\sqrt{M_k}}\ .
 \eeq
 Since $\delta_{\square}\big(W_u,W_v\big)= \inf_{\tau}\|W_u(.,.) - W_v(\tau . ,\tau .)\|_{\square}$ both \eqref{eq:lower_separated_cut_norm} and \eqref{separ_cut_weight_graphon} straightforwardly follow from \eqref{eq:lower_cut_separated}. 
We 
denote 
\begin{align}
 \mathcal{B}_{11}:= \tau^{-1}\left ([0,1/2]\right )\cap [0,1/2), &\quad\mathcal{B}_{12}:= \tau^{-1}\left ([0,1/2]\right )\cap (1/2,1],\nonumber\\\mathcal{B}_{21}:= \tau^{-1}\left ((1/2,1]\right )\cap [0,1/2],&\quad \mathcal{B}_{22}:= \tau^{-1}\left ((1/2,1]\right )\cap (1/2,1]\label{eq:def_B}.
 \end{align}
Since $\tau$ is measure-preserving, we have
\begin{align}\label{eq:leb_measures}
\lambda(\mathcal{B}_{11})= \lambda(\mathcal{B}_{22})= 1/2 - \lambda(\mathcal{B}_{12})= 1/2 - \lambda(\mathcal{B}_{21}).
\end{align}
Now, we consider three cases (i) $\lambda(\mathcal{B}_{12})\leq k_1\epsilon/64$, (ii) $k_1\epsilon/64<\lambda(\mathcal{B}_{12})\leq 1/2 - k_1\epsilon/64$ and (iii) $\lambda(\mathcal{B}_{12})> 1/2 - k_1\epsilon/64$. In the Case (i) we shall focus on the restriction of $W_u$ and $W_v^{\tau}$ on $\cB_{11}\times \cB_{22}$ so that these restrictions are $k_1\times M_k$--step functions. In the Case (ii), we focus on restrictions to $\cB_{21}\times \cB_{22}$, so that $W_{v}^{\tau}$ is constant on this restriction. In the pathological case (iii), we introduce a subset such that the restriction of $W_u$ is a $M_k\times k_1$--step function and  the restriction of $W_v^{\tau}$ is a $k_1\times M_k$--step function.

\bigskip

\noindent 
{\bf Case (i)}. We focus our attention on coordinates $(x,y)$ in  $\mathcal{B}_{11}\times \mathcal{B}_{22}$. Recall that the cumulative distribution function $G$ is defined by $G(0)=1/2$ and $G(b)=1/2+b/(2M_k)$ for  $b\in [M_k]$. 
For any $(r,s)\in [M_k]^2$, define 
$$\omega_{r,s}:= \lambda\big\{[G(r-1), G(r)) \cap \tau^{-1}\left ([G(s-1),G(s))\right )\big\}.$$
In other words, $\omega_{r,s}$ stands for the weight of indices corresponding to class $r$ in $W_u$ and class $s$ in $W_{v}^{\tau}$.
By definition of $\omega_{r,s}$, for any $r\in [M_k]$, we have 
$$\omega_{r\bullet}:=\sum_{s\in [M_k]}\omega_{r,s}\leq 1/(2M_k)\quad\text{and}\quad \sum_{r,s}\omega_{r,s}= \lambda(\mathcal{B}_{22}).$$ Let  $\cR$ denote the sets of $r\in [M_k]$ such that $[G(r-1),G(r))$ has a large intersection with $\tau^{-1}([1/2,1]$:
\begin{equation}\label{eq:def_R_Y}
\mathcal{R}:=\{r\in [M_k] \text{ s.t. } \omega_{r\bullet}\geq 3/(7M_k)\}\quad  \text{and}\quad \mathcal{Y}:=\cup_{r\in \mathcal{R}} [G(r-1), G(r))\bigcap \mathcal{B}_{22}.
\end{equation}
Denote $\bar{\cR}$ the complementary set of $\cR$. 
We have that $\lambda(\mathcal{B}_{22})=1/2 - \lambda(\mathcal{B}_{12})\geq 1/2 - k_1\epsilon/64\geq \tfrac{27}{56}$ for $k_1\epsilon$ small enough.
Hence, it follows that
\begin{align}
\dfrac{27}{56}\leq \sum_{r,s}\omega_{r,s}=\sum_{r\in [M_k]} \omega_{r\bullet}=\sum_{r\in \mathcal{R}} \omega_{r\bullet}+\sum_{r\in \bar{\mathcal{R}}} \omega_{r\bullet}
\end{align}
which implies that $|\mathcal{R}|\geq 3 M_k/4$ and $\lambda(\mathcal{Y})=\sum_{r\in \mathcal{R}} \omega_{r\bullet}\geq 9/28$.

Now, denoting $\mathcal{X}:=\mathcal{B}_{11}$, we 
define a new kernel $\overline{W}_v^{\tau}: \mathcal{X}\times  \mathcal{Y}\to [0,1]$ by 
\begin{eqnarray}\label{eq:definition_W_tau}
\overline{W}_v^{\tau}(x,y)&:=& \sum_{r\in \mathcal{R}} \mathds{1}_{\{y\in [G(r-1),G(r))\}}\frac{1}{\lambda\big\{[G(r-1),G(r))\cap \cY\big\}}\int_{[G(r-1),G(r))\cap \cY} W_v(\tau(x),\tau(z))dz \nonumber \\
& =& 
\sum_{a=1}^{k_1}\sum_{r\in \mathcal{R}} \mathds{1}_{\left \{y\in [G(r-1),G(r))\right \}}  \mathds{1}_{\{\tau(x)\in [F_v(a-1),F_v(a))\}} \sum_{s\in [M_k]} \frac{\omega_{r,s}}{\omega_{r\bullet}}  \frac{\left (1+ \mathbf{B}_{a s}\right )}{2}. 
\end{eqnarray}
We can view $\overline{W}_v^{\tau}$ as a smoothed version of the restriction of $W_v^{\tau}$ to $\mathcal{X}\times \mathcal{Y}$.
The marginal functions $\overline{W}_v^{\tau}(x,.)$ are step functions with at most $|\mathcal{R}|\leq M_k$ steps of the form   $[G(r-1),G(r))\cap \mathcal{B}_{22}$. Moreover, on each interval $[G(r-1),G(r))\cap \mathcal{B}_{22}$, $\overline{W}_v^{\tau}(x,y)$ is equal to the mean of $W_v^{\tau}(x,z)$ for $z$ ranging on this set. 
Equipped with this notation, we can control the cut distance between $W_u$ and $W_v^{\tau}$ in terms of the $l_1$ distance between the restriction of $W_u$ to $\mathcal{X}\times \mathcal{Y}$  and $\overline{W}_{v}^{\tau}$.
For ease of notation, we still write $W_u$ for  for the restriction of $W_u$ to $\mathcal{X}\times \mathcal{Y}$ when there is no ambiguity.

The following lemma provides a lower bound of the cut norm $\|W_u-W_v^{\tau}\|_{\square}$ in terms of the $l_1$ norm of $\|W_u-\overline{W}_{v}^{\tau}\|_{1}$.
\begin{lem}\label{lem:cut_l1}
 For any $u$, $v$ in $\mathcal{C}$ and any measure-preserving transformation $\tau$, we have
\begin{equation}\label{eq:lower_cut_l1}
 \|W_u-W_v^{\tau}\|_{\square}  \geq \frac{1}{4\sqrt{2M_k}}\|W_u-\overline{W}_{v}^{\tau}\|_{1} \ ,
\end{equation}
where $\overline{W}_{v}^{\tau}$ is defined in  \eqref{eq:definition_W_tau}.
\end{lem}
In view of Lemma \ref{lem:cut_l1} it is enough to control the $l_1$ norm $\|W_u-\overline{W}_{v}^{\tau}\|_{1}$. We can do it in  a similar way as it is done in the proof of Lemma 4.5 in \cite{klopp_graphon}.
For  $a\neq b$ and any $x\in\left  [F_u(a-1),F_u(a)\right )\cap \mathcal{X}$ and $x'\in\left  [F_u(b-1),F_u(b)\right )\cap \mathcal{X}$ , the inner product between  $W_u(x,.)$ and $ W_v(x',.)$ satisfies 
\begin{eqnarray}
\lefteqn{\left |\int_{\mathcal{Y}} (W_u(x,y)-1/2)(W_v(x',y)-1/2)dy\right |}&& \nonumber\\ \nonumber &\leq  &  \left |\int_{[1/2,1]} (W_u(x,y)-1/2)(W_v(x',y)-1/2)dy\right | +  \frac{1}{4}\lambda\big\{[1/2,1]\setminus \mathcal{Y}\big \}\\
&\leq & \frac{1}{8M_k}\left |\sum_{c=1}^{M_k} \bB_{ac}\bB_{bc}\right |  +\frac{5}{112} \leq  \frac{1}{32}+ \frac{5}{112} \label{eq:upper_inner_graphon} \
\end{eqnarray}
where we used \eqref{eq:property_1_B} in the last line.
For any $a,b\in [k_1]$, let $\psi_{ab}$ denote the  Lebesgue measure of the set $$[F_{u}(a-1), F_u(a))\cap \tau^{-1}([F_v(b-1),F_v(b)))\cap \mathcal{X}.$$
 Since $\tau$ is measure preserving, it follows that $\sum_b \psi_{ab}\leq 1/(2k_1)+u_a$ and $\sum_a \psi_{ab}\leq  1/(2k_1)+v_b$. For any $y\in \cY$, we set $$h_{u,a}(y):= W_u(F_u(a-1),y)-1/2\quad \text{and} \quad k_{v,b}(y):= \overline{W}^{\tau}_v( \tau^{-1}(F_v(b-1)),y)-1/2\ .$$ Equipped with this notation, we have
 \beqn
\int_{\mathcal{X}\times \mathcal{Y}} \vert W_u(x,y) - \overline{W}^{\tau}_v(x,y)\vert dx dy = \sum_{a=1}^{k_1} \sum_{b=1}^{k_1} \psi_{a,b}\int_{\cY} \vert h_{u,a}(y) - k_{v,b}(y)\vert dy .
 \eeqn
 Now take any
$a_1\neq a_2$. By \eqref{eq:upper_inner_graphon},  $|h_{u,a}(y)|=1/2$  and using the triangle inequality, we derive that 
\beqn 
\|h_{u,a_1}-k_{v,b}\|_1 + \|h_{u,a_2}-k_{v,b}\|_1&\geq& \|h_{u,a_1}-h_{u,a_2}\|_1\\
&\geq& \|h_{u,a_1}-h_{u,a_2}\|^{2}_{2}\\
&\geq &  2\left [\frac{1}{4}\lambda(\cY) -  \frac{1}{32}-\frac{5}{112}\right ] \geq \frac{1}{112}\ ,
\eeqn 
where we used $\lambda(\cY)\geq 9/28$ in the last line. 
As a consequence, for any $ b\in [k_1]$ there exists at most one $a\in [k_1]$ such that $\|h_{u,a}- k_{v,b}\|_1<  1/224$. If such index $a$ exists, it is denoted by $\pi(b)$. Then, it is possible to extend $\pi$ as a function from $[k_1]$ to $[k_1]$. Since $\sum_{a,b}\psi_{a,b}=\lambda(\mathcal{X})$, we get
\beqn 
\|W_u - \overline{W}^{\tau}_v\|_1 & \geq&  \frac{1}{224}\sum_{b=1}^{k_1} \sum_{a\neq \pi(b)} \psi_{a,b}=  \frac{1}{224}\sum_{b=1}^{k_1}\left [\left(1/(2k_1)+ v_b -\psi_{\pi(b),b}  \right)  - \left  (\frac{1}{2} - \lambda[\mathcal{\mathcal{X}}]\right )\right ]\\
&= &  \frac{1}{224}\sum_{b=1}^{k_1}\left [\left(1/(2k_1)+ v_b -\psi_{\pi(b),b}  \right) -  \lambda[\mathcal{B}_{1,2}]\right ]\\
&\geq & \frac{1}{224}\sum_{b=1}^{k_1}\left [\left(1/(2k_1)+ v_b -\psi_{\pi(b),b}  \right) - k_1\epsilon/ 64 \right ]\ ,
\eeqn 
since  $\lambda[\mathcal{B}_{1,2}]\leq k_1\epsilon/ 64$. If the sum $\sum_{b=1}^{k_1} 1/(2k_1)+ v_b -\psi_{\pi(b),b}$ is greater than $k_1\epsilon/32$, then \eqref{eq:lower_cut_separated} is satisfied. Thus, we can assume in the sequel that  $\sum_{b=1}^{k_1}1/(2k_1)+ v_b -\psi_{\pi(b),b}\leq k_1\epsilon/32$. 

 Using that $\psi_{a,b}\leq (1/(2k_1)+u_a)\wedge (1/(2k_1)+v_b)$ and that the cardinality of the collection $\{b\in[k_1]:\, v_b>0\}$ is $k_1/2$ we deduce  that the  collection $\{b\in[k_1]:\, v_b>0,\ u_{\pi(b)}>0\text{ and } \psi_{\pi(b),b}\geq 1/(2k_1)\}$ has cardinality greater than $7k_1/16$. Now, Lemma \ref{lem:varshamov_variation} implies that   $|\cA_u\cap \cA_v|\leq 3k_1/8$ for $u\neq v\in \cC$. Then, there exist subsets $A\subset \cA_{u}$ and $B\subset \cA_{v}$ of cardinality  $\eta_0k_1$ (recall that $\eta_0=1/16)$ such that $\pi(B)=A$, $A\cap B=\emptyset$, and $\psi_{\pi(b),b}\geq 1/(2k_1)$ for all $b\in B$. The condition $\psi_{\pi(b),b}\geq 1/(2k_1)$ implies that $\pi$ is injective on $B$.
Hence, 
\beqn
\| W_u - \overline{W}^{\tau}_v\|_1&\geq & \sum_{b\in B} \psi_{\pi(b),b}\int_{\mathcal{Y}}\left |h_{u,\pi(b)}(y) - k_{v,b}(y)\right |dy
\\
&\geq &   \frac{C}{k_1 M_k}  \sum_{b\in B} \sum_{c\in \mathcal{R}} \left |\bQ_{\pi(b),c} - \frac{\sum_{d\in [M_k]}\omega_{b,d}\bQ_{b,d}}{\omega_{b,\bullet}}\right |
\\
&  & =   \frac{C'}{k_1 M_k}  \sum_{b\in B} \sum_{c\in \mathcal{R}} \left |\bB_{\pi(b),c} - \frac{\sum_{d\in [M_k]}\omega_{c,d}\bB_{b,d}}{\omega_{c,\bullet}}\right |\ ,
\eeqn 
where the second inequality follows from  $\psi_{\pi(b),b}\geq 1/(2k_1)$ and the fact that $h_{u,\pi(b)}$ and $k_{v,b}$ are step functions with steps larger than $3/(7 M_k)$ (see \eqref{eq:def_R_Y}, the definition of $\mathcal{R}$ and $\mathcal{Y}$). 
 Finally, we apply the property \eqref{eq:property_2_B} of $\bB$ to conclude that 
\[\int \vert W_u(x,y) - \overline{W}^{\tau}_v (x,y)\vert dx dy \geq C \ge C' k_1\epsilon , \]
which, together with Lemma \ref{lem:cut_l1}, proves \eqref{eq:lower_cut_separated}.

\bigskip 

\noindent 

{\bf Case (ii)}. Now we assume that  $k_1\epsilon/64<\lambda(\mathcal{B}_{12})<1/2 - k_1\epsilon/64$. Take $\mathcal{X}= \mathcal{B}_{21}$ and $\mathcal{Y}= \mathcal{B}_{22}$. We have that, on $\mathcal{X}\times \mathcal{Y}$, $W_v^{\tau}$ is constant and equals $1/2$. Denote $U$ the restriction of $W_u-1/2$ to $\mathcal{X}\times \mathcal{Y}$. Then, it follows that $\|W_u-W_v^{\tau}\|_{\square}\geq \|U\|_{\square}$. The kernel $U$ is at most  $k_1\times  M_k$ step function. By Lemma \ref{lem:1cut_l1}, we obtain 
\[
 \|U\|_{\square}\geq \frac{1}{4\sqrt{2M_k}}\|U\|_1= \frac{1}{8\sqrt{2M_k}}\lambda(\mathcal{X})\lambda(\mathcal{Y})=\frac{1}{8\sqrt{2M_k}}\lambda(\mathcal{X})\left(\frac{1}{2}-\lambda(\mathcal{X})\right)\ ,
\]
where the last equality follows from  \eqref{eq:leb_measures}. 
Using $\lambda(\mathcal{X})=\lambda(\mathcal{B}_{12})$ and $x(1/2-x)\geq 1/4\min\left(x,(1/2-x)\right)$ we obtain \eqref{eq:lower_cut_separated}.

\bigskip 

\noindent 
{\bf Case (iii)}. Now we assume that  $\lambda(\mathcal{B}_{12})\geq 1/2 - k\epsilon/64$ and  take $\mathcal{X}=\mathcal{B}_{21}$ and $\mathcal{Y}= \mathcal{B}_{12}$ so that $\lambda(\mathcal{X})=\lambda(\mathcal{B}_{12})\geq 1/2-k_1\epsilon/64$. Define the smoothed kernel $\overline{W}_v^{\tau}: \mathcal{X}\times  \mathcal{Y}\to [0,1]$ by 
\begin{eqnarray}\nonumber 
\overline{W}_v^{\tau}(x,y)&:=& \sum_{a=1}^{M_r} \mathds{1}_{\{y\in [G(a-1),G(a))\}}\frac{1}{\lambda\big\{[G(a-1),G(a))\cap \cY\big\}}\int_{[G(a-1),G(a))\cap \cY} W_v(\tau(x),\tau(z))dz \nonumber \label{eq:definition_W_tau_2}\ .
\end{eqnarray}
As a consequence, $\overline{W}_v^{\tau}$ is $M_k\times M_k$ block-constant on subsets of the form $ \big(\tau^{-1}[G(a-1),G(a))\cap \cX\big)  \times \big([G(b-1),G(b))\cap \cY\big)$. Arguing as in the proof of Lemma \ref{lem:cut_l1}, we derive that 
\begin{equation}\label{eq:case_iii}
\|W_u-W_v^{\tau}\|_{\square} \geq \frac{1}{4\sqrt{2M_k}}\|W_u-\overline{W}_v^{\tau}\|_1\ .
\end{equation}
 
For any $a$ such that $[F_u(a-1),F_u(a))\cap \mathcal{X}\neq \emptyset$ define the function $h_{u,a}$ on $\mathcal{Y}$ by $h_{u,a}(y):= W_u(F_u(a-1),y)-1/2$.  Arguing as in Case (i), we observe that $\|h_{u,a_1}-h_{u,a_2}\|_1 \geq 1/112$ for any $a_1\neq a_2$. We have that the kernel $\overline{W}_v^{\tau}$ is a $M_k\times M_k$ step function. Hence, there exists a partition $(\mathcal{X}_b)_{b=1,\ldots, M_k}$  of $\mathcal{X}$ and $M_k$ functions $k_{b}(y)$ such that $\left (\overline{W}_v^{\tau}-1/2\right )(x,y)= \sum_{b=1}^{M_k}\mathds{1}_{x\in \mathcal{X}_b} k_{b}(y)$. Then, the triangular inequality ensures that, for any $a_1\neq a_2$ and any $b\in [M_k]$, we have $\|h_{u,a_1}-k_{b}\|_1 + \|h_{u,a_1}-k_{b}\|_1\geq \|h_{u,a_1}-h_{u,a_2}\|_1\geq  1/112$. As a consequence, for any $b\in [M_k]$ there exists at most one $a$, which we will denote by $\pi(b)$, such that $\|h_{u,\pi(b)}-k_{b}\|_1\leq  1/224$. Now we compute
\beqn 
\left \|W_u-\overline{W}_v^{\tau}\right \|_1&=& \sum_{b=1}^{M_k}\sum_{a=1}^{k_1} \lambda(\mathcal{X}_b\cap\big[F_u(a-1),F_u(a))\cap \mathcal{X}\big)\|h_{u,a}-k_{b}\|_1
\\&\geq& \frac{1}{224}\sum_{b=1}^{M_k} \lambda\Big[\mathcal{X}_b\setminus\big[F_u(\pi(b)-1),F_u(\pi(b)))\cap \mathcal{X}\big]\Big]\\
&\geq &\frac{1}{224} \left [\lambda(\mathcal{X}) - \sum_{b=1}^{M_k} \frac{1}{2k_1}+u_{\pi(b)} \right ]\\
&\geq &\frac{1}{224} \left [\lambda(\mathcal{X}) - \frac{M_k}{2k_1}- M_k\epsilon\right ]\geq C',
\eeqn 
where we used $\lambda(\mathcal{X})\geq 1/4$, $M_k/k\leq 1/8$, and  that $M_k\epsilon\leq k\epsilon$ is small enough. 
Together with \eqref{eq:case_iii}, we obtain the desired result  \eqref{eq:lower_cut_separated}.
\end{proof}

\smallskip
\begin{proof}[\textbf{Proof of Lemma \ref{lem:cut_l1}}]
We first prove that $\|W_u-\overline{W}_{v}^{\tau}\|_{\square}\leq \|W_u-W_{v}^{\tau}\|_{\square}$. Fix any measurable subset  $S\subset \mathcal{X}$.  Since functions $\left [W_{u}-\overline{W}_{v}^{\tau}\right ](x,\cdot)$ are constant on each set $[G(r-1),G(r))\cap \mathcal{Y}$, the supremum 
$
\sup_{T\subset \mathcal{Y}}\left |\int_{S\times T}W_u(x,y)-\overline{W}_{v}^{\tau}(x,y)dx dy\right |
$
is achieved by a subset $T$ which is an union of some of  $[G(r-1),G(r))\cap \mathcal{Y}$, that is $T=\cup_{r\in\mathcal{R}'\subset\mathcal{R}}[G(r-1),G(r))\cap \mathcal{Y}$. For such  $T$, the definition \eqref{eq:definition_W_tau} of $\overline{W}^{\tau}_{v}$ implies
$ \int_{S\times T}\overline{W}_{v}^{\tau}(x,y)dx dy=\int_{S\times T}W_{v}^{\tau}(x,y)dx dy$ so that 
\[
 \sup_{T\subset \mathcal{Y}}\left |\int_{S\times T}W_u(x,y)-\overline{W}_{v}^{\tau}(x,y)dx dy\right |\leq \sup_{T\subset \mathcal{Y}}\left |\int_{S\times T}W_u(x,y)-W_{v}^{\tau}(x,y)dx dy\right |\ .
\]
Taking the supremum over all $S$  leads to $\|W_u-\overline{W}_{v}^{\tau}\|_{\square}\leq \|W_u-W_{v}^{\tau}\|_{\square}$. 
By definition of $W_u$ and $\overline{W}_{v}^{\tau}$ we have that $U$ is a $ k_1^2\times M_k$ step function. Then, Lemma \ref{lem:1cut_l1} allows us to conclude 
\[
  \|W_u-W_{v}^{\tau}\|_{\square}\geq \frac{1}{4\sqrt{2M_k}}\|W_u-\overline{W}_{v}^{\tau}\|_{1}\ .
\]
\end{proof}

\smallskip

\begin{proof}[\textbf{Proof of Lemma \ref{lem:kullback}}]
The proof of Lemma \ref{lem:kullback} follows the lines of the proof of of Lemma 4.3 in \cite{klopp_graphon} and we give it here for completeness. For $u\in \cC_0$, let $\zeta(u)=(\zeta_1(u),\ldots ,\zeta_n(u))$ be the vector of $n$ i.i.d. random variables with the discrete distribution on $[k_{1}+M_{k}]$ defined by 
\begin{equation}\label{eq:def_probab_clusters}
\P[\zeta_1(u)=a]= \left \{
\begin{array}{ll}
 1/(2k_1) + u_a\quad &\text{if}\quad a\in [k_1]\\
 1/(2M_{k})\quad &\text{if}\quad k_1+1\leq a\leq M_k+k_1
\end{array} \right.
\end{equation}
Let $\bTheta_0$ be the $n\times n$ symmetric matrix with elements $(\bTheta_0)_{ii}=0$ and $(\bTheta_0)_{ij}=\rho_n \bQ_{\zeta_i(u),\zeta_j(u)}$ for $i\neq j$. Assume that, conditionally on $\zeta(u)$, the adjacency matrix $\bA$ is sampled according to the network sequence model with such probability matrix $\bTheta_0$. Notice that in this case the observations $\bA'=(\bA_{ij}, 1\le j<i\le n)$ have the probability distribution $\P_{W_u}$.  
 Using this remark and introducing the probabilities $\alpha_{\boldsymbol{a}} (u) = \P[\zeta(u)=\boldsymbol{a}] $ and $p_{A\boldsymbol{a}}=\P[\bA'=A\vert \zeta(u)=\boldsymbol{a}]$
for $\boldsymbol{a}\in [k_1+M_k]^n$, we can write the Kullback-Leibler divergence between $\P_{W_u}$ and $\P_{W_v}$ in the form
$$
\mathcal{KL}(\P_{ W_u},\P_{ W_v})=\sum_{A}\sum_{\boldsymbol{a}}p_{A\boldsymbol{a}}\alpha_{\boldsymbol{a}} (u) \log\left(\frac{\sum_{\boldsymbol{a}}p_{A\boldsymbol{a}}\alpha_{\boldsymbol{a}} (u)}{\sum_{\boldsymbol{a}}p_{A\boldsymbol{a}}\alpha_{\boldsymbol{a}} (v)}\right)
$$
where the sums in $\boldsymbol{a}$ are over $[k_1+M_k]^n$ and the sum in $A$ is over all triangular upper halves of matrices in $ \{0,1\}^{n\times n }$.   Since 
the function $(x,y) \mapsto x\log(x/y)$ is convex we can apply Jensen's inequality to get
\begin{equation*} 
\mathcal{KL}(\P_{ W_u},\P_{ W_v}) \le \sum_{\boldsymbol{a}}\alpha_{\boldsymbol{a}} (u) \log\left(\frac{\alpha_{\boldsymbol{a}} (u)}{\alpha_{\boldsymbol{a}} (v)}\right)=
n\sum_{a\in[k_1+M_k]} \P[\zeta_1(u)=a]\log\left(\frac{\P[\zeta_1(u)=a]}{\P[\zeta_1(v)=a]}\right)
 \end{equation*}
where the last equality follows from the fact that $\alpha_{\boldsymbol{a}} (u)$ are $n$-product probabilities. Using \eqref{eq:def_probab_clusters} we get 
\begin{equation} \label{kullb_1}
\mathcal{KL}(\P_{ W_u},\P_{ W_v}) \le 
n\sum_{a\in[k_{1}]} (1/(2k_1) + u_a)\log\left(\frac{1/(2k_1)+u_a}{1/(2k_1)+v_a}\right)\ , 
 \end{equation}
which is equal to $n/2$ times the Kullback-Leibler divergence between two discrete distribution. 
Since the Kullback-Leibler divergence  is less than the chi-square divergence we obtain  
  $$
\sum_{a\in[k_1]} (1/k_1 + 2u_a)\log\left(\frac{1/k_1+2u_a}{1/k_1+2v_a}\right) \le
\sum_{a\in[k_1]} \frac{4(u_a-v_a)^2}{1/k_1+2v_a}\le 64k^2\epsilon^2/3,
 $$
 where last inequality we use  $|v_a|\le \epsilon \leq 1/(8k_1)$, and $|u_a-v_a|\le 2\epsilon$. Combining this with \eqref{kullb_1} proves the lemma.

\end{proof}

  \section{Proof of Proposition \ref{prp:lower_minimax_graphon} }\label{proof_prp:lower_minimax_graphon}
  To prove   \eqref{eq:lower_minimax_L1}, it is enough to prove  separately the following three  minimax lower bounds:
  \begin{eqnarray} \label{eq:lower1}
   \inf_{\widehat{f}}\sup_{W_0\in \cW^+[k]}\E_{W_0}[\delta_1(\widehat{f}, \rho_n W_0)] &\geq& C \rho_n \sqrt{\frac{k-1}{n}}\ ,\\
   \inf_{\widehat{f}}\sup_{W_0\in \cW^+[k]}\E_{W_0}[\delta_1(\widehat{f}, \rho_n W_0)] &\geq& C \min 
  \left(\sqrt{\rho_n} \frac{k}{n}, \rho_n\right),
   \label{eq:lower2} \\
   \inf_{\widehat{f}}\sup_{W_0\in \cW^+[2]}\E_{W_0}[\delta_1(\widehat{f}, \rho_n W_0)] &\geq& C  \min
  \left( \sqrt{\frac{\rho_n}{n}},\rho_n\right).
   \label{eq:lower3}
  \end{eqnarray}
  The proof of \eqref{eq:lower1} follows from the proof of (43) in \cite{klopp_graphon} using the trivial inequality \begin{equation}\label{eq:l1-l2}
  \Vert W_{u}(x,y)-W_{v}(\tau(x),\tau(y))\Vert^{2}_{2}\leq \Vert W_{u}(x,y)-W_{v}(\tau(x),\tau(y))\Vert_{1}.
\end{equation} 
The proof of \eqref{eq:lower2} follows the lines  of the proof of (44) using that $\Vert \mathbf{B}\Vert^{2}_{2}=\Vert \mathbf{B}\Vert_{1}$ for matrices with entries in $\{-1,1\}$. The proof of  \eqref{eq:lower3} is identical to the proof of (45) in \cite{klopp_graphon}.

In order to prove the upper bound \eqref{eq:minimax_risk_graphon_l1}, the proof of  Proposition 3.2 in \cite{klopp_graphon} can be easily modified to get an upper bound on the agnostic error measured in $l_1$-distance:
\begin{lem}[Agnostic error measured in $l_1$-distance]\label{prp_agnostic_l1}
Consider the $W$-random graph model. 
 For all integer $k\leq n$, $W_0\in \cW^+[k]$ and $\rho_n>0$, we have 
\[
\E\left[\delta_{1}(\widetilde f_{\bTheta_0}, f_0)\right]\leq C\rho_n\sqrt{\frac{k}{n}}\ .
\]
\end{lem}
Now \eqref{eq:minimax_risk_graphon_l1} follows from Lemma \ref{prp_agnostic_l1} and \eqref{eq:minimax_matrix_L1}. Finally, the $\rho_n$ convergence rate is simply achieved by the constant estimator $\widehat{f}\equiv 0$. 
\section{Proof of Proposition \ref{prop:unbounded} } \label{proof_prop:unbounded}
For $\bTheta_0$ generated according to the sparse $W$-random graph model 
\eqref{sparse_graphon_mod_p} with graphon $W_0\in\cW^+_{1}$, 
integrating \eqref{eq:bound_bernstein_l1} with respect to $\bxi$ and using $\Vert W_0\Vert_{1}=1$, we get
 $$\E_{W_0}\left[\|\bA-\bTheta_0\|_\cut\right]\leq 6\sqrt{\dfrac{\rho_n}{n}}.$$
 So, using the triangle inequality \eqref{eq:integrated_risk_decomposition} it is enough to bound the agnostic error $\E_{W_0}[\delta_{\cut}(\widetilde{f}_{\bTheta_0} , f'_0)]$.
We take $W^{*}\in \cW^+_1[k,\mu]$ (or $W^{*}\in \cW^+_2[k]$ in the case of $L_2$ graphons) such that 
\begin{equation}\label{eq:def_inf}
\delta_{1}\left(W^{*}, W'_0\right)\leq \inf_{W\in \cW^+_1[k,\mu]}\delta_{1}\left(W, W'_0\right)(1+1/n^{2})\ , 
\end{equation}
or $\delta_{2}\left(W^{*}, W'_0\right)\leq \inf_{W\in \cW^+_2[k]}\delta_{2}\left(W, W'_0\right)(1+1/n^{2})$ for $L_2$ graphons. 
Without lost of generality we can assume that $\rho_nW^{*}(x,y)\leq 1$. Let $f^{*}=\rho_nW^{*}$ and  $\bTheta^{*}=(\bTheta^{*}_{ij})$ be such such that for $i\neq j$
$\bTheta^{*}_{ij}= W^*[\xi_i,\xi_j]$ where $(\xi_i)$ are the same as for $\bTheta_0$. 
Triangle inequality implies
  \beqn
  \E_{W_0}\left[\delta_{\cut}\left(\widetilde{f}_{\bTheta_0} , f'_0\right)\right]&\leq & \delta_{\cut}\left(f'_0,f^{*}  \right)+\E_{W_0}\left[\delta_{\cut}\left( f^{*},\widetilde{f}_{\bTheta^{*}} \right)\right]+\E_{W_0}\left[\delta_{\cut}\left(\widetilde{f}_{\bTheta^{*}} ,\widetilde{f}_{\bTheta_0} \right)\right]\\
  &\leq& 2\delta_{1}\left(f'_0,f^{*}  \right)+\E_{W\emph{*}}\left[\delta_{\cut}\left(f^{*}, \widetilde{f}_{\bTheta^{*}} \right)\right]
  \eeqn
where we use $\delta_{\cut}( f'_0,f^{*} )\leq \delta_{1}( f'_0,f^{*} )$ and $\E_{W_0}[\delta_{\cut}(\widetilde{f}_{\bTheta^{*}} , \widetilde{f}_{\bTheta_0})]\leq \delta_{1}( f'_0,f^{*} )$ and that $ \widetilde{f}_{\bTheta^{*}}$ is distributed as under $W^*$.  Similarly for $L_2$ graphons, we obtain $\E_{W_0}[\delta_{\cut}(\widetilde{f}'_{\bTheta_0} , f_0)]\leq 2\delta_{2}(f'_0,f^{*}  )+\E_{W\emph{*}}[\delta_{\cut}(f^{*}, \widetilde{f}_{\bTheta^{*}} )]$. Then, 
we use the following lemma:
\begin{lem}\label{lem:agnistic}
\begin{itemize}

\item[(i)]
Consider any  $W^* \in \cW^+_1[k,\mu]$ and $\rho_n\geq 1/n$ such that $\rho_nW^* (x,y)\leq 1$. Then
\beqn 
\E_{W^*}\left[\delta_{\cut}\left(\widetilde{f}_{\bTheta^{*}} , f^{*}\right)\right]
\leq C \left  [
\rho_n\Vert W^{*}\Vert_{1} \sqrt{\frac{k}{ \mu n}}+\sqrt{\dfrac{\rho_n}{n}}\right].
\eeqn
\item [(ii)] Consider any $W^{*}\in \cW^+_2[k]$ and $\rho_n\geq 1/n$ such that $\rho_nW^{*}(x,y)\leq 1$. Then,
\beq\label{lem:agnistic_l2}
\E_{W^*}\left[\delta_{\cut}\left(\widetilde{f}_{\bTheta^{*}} , f^{*}\right)\right]
\leq C \left  [
\rho_n\Vert W^{*}\Vert_{2} \sqrt{\frac{k}{n}}+\sqrt{\dfrac{\rho_n}{n}}\right].
\eeq
\end{itemize}
\end{lem}
Now  \eqref{eq:minimax_risk_graphon_biais_1} follows from  (i) of Lemma \ref{lem:agnistic} and $\Vert W^{*}\Vert_{1}\leq \Vert W_0\Vert_1(2+n^{-2})$. The proof of \eqref{eq:minimax_risk_graphon_biais_2} follows the same lines using (ii) of Lemma \ref{lem:agnistic}.

To prove \eqref{eq:minimax_risk_graphon_biais_11} and \eqref{eq:minimax_risk_graphon_biais_22} we only need to prove that $\E_{W_0}\left[\|\widetilde\bTheta_{\lambda}-\bTheta_0\|_\cut\right]\leq C\sqrt{\rho_n/n}$. Using the definition of $\widetilde \bTheta_{\lambda}$~\eqref{hard_estimator} we compute
	\begin{equation*}
	\begin{split}
	\E_{W_0}\left[\|\widetilde\bTheta_{\lambda}-\bTheta_0\|_\cut\right]&\leq \E_{W_0}\left[\|\bA-\bTheta_0\|_\cut\right]+\E_{W_0}\left[\|\widetilde\bTheta_{\lambda}-\bA\|_\cut\right]\\
	&\leq 6\sqrt{\dfrac{\rho_n}{n}}+ \E_{W_0}\left[\dfrac{\|\widetilde\bTheta_{\lambda}-\bA\|_{2\rightarrow 2}}{n}\right]\\ 	&\leq C\sqrt{\dfrac{\rho_n}{n}}
	\end{split}\ ,
	\end{equation*}
where we used that $\|\bB\|_{\square}\leq \|\bB\|_{2\rightarrow 2}/n$ and the definition of $\widetilde{\bTheta}_{\lambda}$. 
This completes the proof of Proposition \ref{prop:unbounded}.

\begin{proof}[\textbf{Proof of Lemma \ref{lem:agnistic}}] 
Consider the matrix ${\bTheta}'$ with entries $({\bTheta}')_{ij}=\rho_nW^{*}(\xi_i,\xi_j)$ for all $i,j$. As opposed to $\bTheta^{*}$, the diagonal entries of ${\bTheta}'$ are not constrained to be null. By the triangle inequality, we  get 
\beq\label{eq:agnostic_decomposition_unbounded}
\E_{W^{*}}\left[\delta_{\cut}\left(\widetilde{f}_{\bTheta^{*}} , f^{*}\right)\right]\leq \E_{W^{*}}\left[\delta_{\cut}\left(\widetilde{f}_{\bTheta^{*}} , \widetilde{f}_{{\bTheta}'}\right)\right]+ \E_{W^{*}}\left[\delta_{\cut}\left(\widetilde{f}_{{\bTheta}'} , f^{*}\right)\right]\ .
\eeq
Since the entries of $\bTheta^{*}$ coincide with those of ${\bTheta}'$ outside the diagonal, the difference $\widetilde{f}_{\bTheta^{*}}- \widetilde{f}_{{\bTheta}'}$ is null outside of a set of measure $1/n$. Also, the entries of ${\bTheta}'$ are smaller than $1$. It follows that 
$ \E[\delta_{\cut}(\widetilde{f}_{\bTheta^{*}} , \widetilde{f}_{{\bTheta}'})]\leq 1/n\leq \sqrt{\rho_n/n}$. Since $\delta_{\cut}(\widetilde{f}_{{\bTheta}'} , f^{*})\leq \delta_1(\widetilde{f}_{{\bTheta}'} , f^{*})$, it suffices   to prove that 
\beqn 
\E_{W^{*}}[\delta_{1}(\widetilde{f}_{{\bTheta}'} , f^{*})]&\leq& C\rho_n\Vert W^{*}\Vert_{1} \sqrt{\frac{k}{\mu n}} \ , \quad \text{ for $W^*\in \cW^+_1[k,\mu]$}\\
\E_{W^{*}}[\delta_{1}(\widetilde{f}_{{\bTheta}'} , f^{*})]&\leq& C\rho_n\Vert W^{*}\Vert_{2} \sqrt{\frac{k}{n}} \ , \quad \text{ for $W^*\in \cW^+_2[k]$}\ .
\eeqn 

Since $W^*$ is a $k$-step function, we can reorganize $f^*$ and $\widetilde{f}_{\bTheta'}$ in such a way that these two graphon are equal on a set of large Lebesgue value. More precisely, we adopt the same approach as in the proof of Theorem  \ref{prp_agnostic_block} and we only sketch the result here. Let $\bQ\in (\mathbb{R}^+)^{k\times k}_{sym}$ and $\phi:[0,1]\times [k]$ that characterize $W^*$. For $a=1,\ldots, k$, denote $\lambda_a=\lambda(\phi^{-1}(\{a\}))$. 
For any $b\in [k]$, define the cumulative distribution function  $F_{\phi}(b)=\sum_{a=1}^{b} \lambda_a$ and set $F_{\phi}(0)=0$. For any $(a,b)\in [k]\times [k]$ define $\Pi_{ab}(\phi)=[F_{\phi}(a-1),F_{\phi}(a))\times [F_{\phi}(b-1),F_{\phi}(b))$. Define
 $W'(x,y)= \sum_{a=1}^k \sum_{b=1}^k \bQ_{ab} \mathds{1}_{\Pi_{ab}(\phi)}(x,y)$.  Obviously, $f'=\rho_n W'$ is weakly isomorphic to $f^{*}=\rho_n W^{*}$. 
Now, let  $\widehat{\lambda}_a=\frac1n \sum_{i=1}^n \mathds{1}_{\{ \xi_i\in \phi^{-1}(a)\}}$ 
 be the (unobserved) empirical frequency of group $a$. Consider a function $\psi:[0,1]\rightarrow [k]$  such that: 
\begin{itemize}
	\item[(i)] $\psi(x)= a$ for all $a\in [k]$ and $x\in [F_{\phi}(a-1), F_{\phi}(a-1)+ \widehat{\lambda}_a\wedge \lambda_a)$,
	\item[(ii)] $\lambda(\psi^{-1}(a))= \widehat{\lambda}_a$ for all $a\in [k]$. 
\end{itemize}
Such  a function $\psi$ exists (for details see the Step 2 of the proof of Theorem \ref{prp_agnostic_block}). 
Finally define the graphon $\widehat{f}'(x,y)= \bQ_{\psi(x),\psi(y)}$. Notice that  $\widehat{f}'$ is weakly isomorphic to the empirical graphon $\widetilde{f}_{\bTheta^{*}}$. Since $\delta_1(\cdot,\cdot)$ is a metric on the quotient space of graphons, we have
\[
\delta_{1}(\widetilde{f}_{\Theta^{*}},f^{*})= \delta_{1}(\widehat{f}',f')\leq \Vert \widehat{f}'-f'\Vert _{1}\ .
\]
The two functions $f'(x,y)$ and $\widehat{f}'(x,y)$ are equal except possibly the case when either $x$ or $y$ belongs to one of the intervals $[F_{\phi}(a-1)+ \widehat{\lambda}_a\wedge \lambda_a, F_{\phi}(a-1)+\lambda_a)$ for $a\in [k]$ and we have
\beqn 
\Vert \widehat{f}'-f' \Vert _{1}&=&\rho_n\Big\Vert \sum_{a=1}^k \sum_{b=1}^k \bQ_{ab} \mathds{1}_{\Pi_{ab}(\phi)}(x,y) -\sum_{a=1}^k \sum_{b=1}^k \bQ_{ab} \mathds{1}_{\Pi_{ab}(\psi)}(x,y)\Big \Vert_{1}\\
&\leq& \rho_n \sum_{a=1}^k \sum_{b=1}^k \bQ_{ab}\lambda\left (\Pi_{ab}(\phi)\triangle \Pi_{ab}(\psi)\right )\\&\leq& 
\rho_n \sum_{a=1}^k \sum_{b=1}^k \bQ_{ab}\left \{\vert\lambda_a-\widehat \lambda_a\vert\lambda_b+\vert\lambda_b-\widehat \lambda_b\vert\lambda_a+\vert\lambda_a-\widehat \lambda_a\vert\vert\lambda_b-\widehat \lambda_b\vert\right \}\ .
\eeqn
Since $\xi_1,\ldots ,\xi_n$ are i.i.d. uniformly distributed random variables,  $n\widehat{\lambda}_a$ has a binomial distribution with parameters ($n$, $\lambda_a$). By Cauchy-Schwarz inequality we get $\E[|\lambda_a-\widehat{\lambda}_a|]\leq \sqrt{\lambda_a(1-\lambda_a)/n}$ and $\bbE (\vert\lambda_a-\widehat \lambda_a\vert\vert\lambda_b-\widehat \lambda_b\vert)\leq \sqrt{\lambda_a\lambda_b}/n$. Then,   we get
\begin{equation*} 
\begin{split}
\E_{W^{*}}\Vert \widehat{f}'-f'\Vert _{1}&\leq 
\dfrac{\rho_n}{\sqrt{n}} \sum_{a=1}^k \sum_{b=1}^k \bQ_{ab}\left \{\sqrt{\lambda_a}\lambda_b+\sqrt{\lambda_b}\lambda_a+\sqrt{\dfrac{\lambda_a\lambda_b}{n}}\right \}.
\end{split}
\end{equation*}
Now for $W^{*}\in \cW^+_{1}[k,\mu]$ we use $\lambda_{a}\geq \mu/k$ for all $a\in[k]$ to get
\begin{equation*} 
\E_{W^{*}}\Vert \widehat{f}'-f'\Vert _{1}\leq 
C\rho_n\sqrt{\dfrac{k}{ \mu n}} \sum_{a=1}^k \sum_{b=1}^k \bQ_{ab}\lambda_a\lambda_b\big(1+ \sqrt{\frac{k}{\mu n}}\big)=C\rho_n\Vert W^{*}\Vert_{1}\sqrt{\dfrac{k}{n}}\ ,
\end{equation*}
since we assume that $k\leq \mu n$. 
For $W^{*}\in \cW^+_2[k]$ we use the Cauchy-Schwarz inequality:
\begin{equation*} 
\E_{W^{*}}\Vert \widehat{f}'-f'\Vert _{1}\leq 
\dfrac{2\rho_n}{\sqrt{n}} \sqrt{\sum_{a=1}^k \sum_{b=1}^k \bQ^{2}_{ab}\lambda_a\lambda_b}\sqrt{\sum_{a=1}^k \sum_{b=1}^k \lambda_b}+\dfrac{\rho_nk}{n}\sqrt{\sum_{a=1}^k \sum_{b=1}^k \bQ^{2}_{ab}\lambda_a\lambda_b}=C\rho_n\Vert W^{*}\Vert_{2}\sqrt{\dfrac{k}{n}}\ ,
\end{equation*}
since $k\leq n$. 
\end{proof}



\end{document}